\documentclass[12pt,leqno]{uconnthesis}

\usepackage{amsmath} 
\usepackage{amscd} 
\usepackage{amsthm}  
\usepackage{calc} 

\usepackage{graphicx} 

\usepackage{enumerate} 
\usepackage{dsfont, marvosym, latexsym, amssymb, wasysym, mathbbol} 
\usepackage{color} 
\usepackage{fancyhdr} 
\usepackage{fancybox} 
\usepackage{ifthen} 
\usepackage{rotate} 

\usepackage[mathcal]{euscript}       
\usepackage{mathrsfs}       

\usepackage{verbatim}

\usepackage{pxfonts}

\swapnumbers     

\numberwithin{equation}{section}                      

\theoremstyle{definition}

\newtheorem{theorem}{Theorem}[section]

\newtheorem{lemma}[theorem]{Lemma}

\newtheorem{proposition}[theorem]{Proposition}

\newtheorem{corollary}[theorem]{Corollary}

\newtheorem{remark}[theorem]{Remark}

\newcommand{\Z}{\hbox{$\mathbb{Z}$}}

\newcommand{\abs}[1]{\hbox{$\left| {#1} \right|$}}

\renewcommand{\tilde}{\widetilde}
\renewcommand{\hat}{\widehat}

\renewcommand{\Im}{\operatorname{Im}}

\def\it{\itshape}

\def\tt{\texttt}
\def\bf{\textbf}

\def\IR{\mathbb{R}}
\def\IQ{\mathbb{Q}}

\def\I1{\mathbb{1}}
\def\IC{\mathbb{C}}

\def\ccK{\mathscr{K}}

\def\ccT{\mathscr{T}}

\def\limx0{\lim_{x \to 0}}

\def\intxyleq1{\underset{\| x - y  \| \leq 1}{\int}}
\def\intxygeq1{\underset{\| x - y  \| \geq 1}{\int}}
\def\intxizetaleq1{\underset{\| \xi - \zeta  \| \leq 1}{\int}}
\def\intxizetageq1{\underset{\| \xi - \zeta \| \geq 1}{\int}}

\def\tab{\hskip 1mm}

\newcounter{hours}
\newcounter{minutes}
\newcommand\printtime{%
  \setcounter{hours}{\the\time/60}%
  \setcounter{minutes}{\the\time-\value{hours}*60}%
  \ifthenelse{\value{hours} > 12}
     {
       \setcounter{hours}{\value{hours}-12}%
       \thehours:\theminutes \ p.m.                
     }
     {
       \thehours:\theminutes \ a.m.                
     } 
}

\def\putdate{{\tt Compiled on \the\month-\the\day-\the\year \ at\printtime} \\}

\begin{document}



\myThesisTitle{The Bi-sequences of Approximation Coefficients for Gauss-like and Renyi-like Maps on the Interval}


\myName{Avraham Bourla}

\myNewDegreeShort{Ph.D.}
\myNewDegreeLong{Doctor of Philosophy}

\myPreviousDegreeShort{M.Sc. Mathematics, B.Sc. Mathematics.}
\myPreviousDegreeLong{
  M.Sc. Mathematics, University of Connecticut, Storrs, Connecticut, 2007 \\
  B.Sc. Mathematics, University of Massachusetts, Amherst, Massachusetts, 2004
}



\myMajorAdvisor{Prof. Andrew Haas}
\myAssociateAdvisorA{Prof. William Abikoff}
\myAssociateAdvisorB{Prof. Alvaro Lozano-Robledo}


\myTOCDepth{2}



\begin{abstractpage}
\noindent We will establish several arithmetic and geometric properties regarding the bi-sequences of approximation coefficients (BAC) associated with the two one-parameter families of piecewise-continuous  M$\operatorname{\ddot{o}}$bius transformations introduced by Haas and Molnar. The Gauss and Renyi maps, which lead to the expansions of irrational numbers on the interval as regular and backwards continued fractions, are realized as special cases. The results are natural generalizations of theorems from Diophantine approximation.
\end{abstractpage}

\begin{titlepage}
\end{titlepage}



\begin{romanpages}

\begin{copyrightpage}
\end{copyrightpage}

\begin{approvalpage}
\end{approvalpage}

\begin{dedicationpage}
To Tammy
\end{dedicationpage}

\begin{acknowledgementpage}
First and foremost I would like to thank my advisor, professor Andrew Haas. I enjoyed working with him and benefited tremendously from his patience and rigor. I would also like to extend my gratitude to the members of my thesis committee - professors Alvaro Lozano-Robledo and William Abikoff, as well as to professors Stephen Miller and Sarah Glaz for their help and guidance. In addition, I am thankful to the graduate directors Professors Everist Gine and Ron Blei and the graduate program assistants Sharon McDermott and Monique Roy for supporting me as a graduate student and teaching assistant. My officemates throughout the years: Amy, Andre, Ivan, Jeff, John, Meng, Phil, Savander and Upendra made this experience fun and memorable. Finally, I want to thank my wife Tammy, to whom this work is devoted, for encouraging me throughout the process of writing this dissertation and for giving me the sense of urgency to complete it in a timely manner.\\
\end{acknowledgementpage}

\begin{contentpage}
  \begin{flushleft}
  \end{flushleft}
  \tableofcontents
  \begin{flushleft}
  \end{flushleft}
\end{contentpage}



\end{romanpages}



 \setcounter{myLineSpacing}{2}


\begin{thethesiscore}
  
  \chapter{Introduction}{}
We define the sequence of approximation coefficients and motivate its study from Diophantine approximation. More recent results due to Borel, Perron, Jurkat, Peyerimhoff, Bagemihl, McLaughlin and Tong show that this sequence has an elegant internal structure as well as simple connections to the sequence of digits for the regular continued fraction expansion. This merits the study of the sequence of approximation coefficients for other continued fractions theories. Our goal is to extend these theorems to the generalized continued fraction theories introduced by Haas and Molnar in \cite{HM}. We realize these continued fractions using the modern approach of  symbolic dynamics on certain classes of piecewise M$\operatorname{\ddot{o}}$bius transformations on the interval.\\


\section{Diophantine approximation and regular continued fractions}{}
\noindent Diophantine approximation is a subfield in number theory concerning the approximation of irrational numbers using rational numbers. Given a real number $r$ and a rational number, which write as the unique quotient $\frac{p}{q}$ of the two integers $p$ and $q$ with $\gcd(p,q)=1$ and $q>0$, our fundamental object of interest is the \bf{approximation coefficient} $\theta(r,\frac{p}{q}) := q^2\abs{r-\frac{p}{q}}$. Small approximation coefficients suggest high quality approximations, combining accuracy with simplicity. For instance, the error in approximating $\pi$ using $\frac{355}{113}=3.1415920353982$ is considerably smaller than the error of its decimal expansion to the fifth digit $\tab 3.14159 = \frac{314159}{100000}$. Since the former rational also has a much smaller denominator, it is of far better quality than the latter. Indeed $\theta\big(\pi,\frac{355}{113}\big)< 0.0341$ whereas $\theta\big(\pi,\frac{314159}{100000}\big)> 26,535$.\\

Given an irrational number $r \in \IR - \IQ$, we obtain the high quality approximations for $r$ by using the euclidean algorithm to write $r$ as a \bf{regular continued fraction} or \bf{RCF}:
\[r = b_0 + \dfrac{1}{b_1 +\dfrac{1}{b_2 + ...}} =  b_0 + [b_1,b_2,...]_0\]
(the reason for the zero subscript $[\cdot]_0$ will become clear later). The \bf{digits of expansion}, also known as partial quotients, $b_0 = b_0(r) \in \mathbb{Z}$ and $b_n = b_n(r) \in \mathbb{N} := \mathbb{Z} \cap [1,\infty)$ for all $n \ge 1$, are uniquely determined by $r$. For all $n \ge 0$, the quantity $\frac{p_n}{q_n} = \frac{p_n}{q_n}(r) := b_0 + [b_1,b_2,...,b_n]_0$ is called the n$^{th}$ \bf{convergent} of $r$. The sequence of convergents $\big\{\frac{p_n}{q_n}\big\}_0^\infty$ tends to $r$ as $n$ tends to infinity. Furthermore, every convergent $\frac{p_n}{q_n}$ is a \bf{best approximate} for $r$, that is, its approximation error $\abs{x - \frac{p_n}{q_n}}$ is less than that of any other rational number whose denominator is smaller than $q_n$. For more about the basic facts regarding RCF, refer to \cite{Hardy} or \cite{K}.\\ 

We define the approximation coefficient of the n$^{\operatorname{th}}$ convergent of $r$ by 
\[ \theta_n = \theta_n(r) := \theta\bigg(r,\frac{p_n}{q_n}(r)\bigg) = q_n^2\abs{r - \frac{p_n}{q_n}}\]
and refer to the sequence $\{\theta_n\}_0^\infty$ as the \bf{sequence of approximation coefficients}. Lagrange \cite[Theorem 5.1.7]{DK} proved that for all irrational numbers $r$ and $n \ge 0$, we have $\theta_n < 1$. Conversely, Legendre \cite[Corollary 5.1.8]{DK} proved that if $\theta(r,\frac{p}{q}) < \frac{1}{2}$ then $\frac{p}{q}$ is a convergent of $r$. For instance, we write
\[\pi = 3 + \dfrac{1}{7 + \dfrac{1}{15+ \dfrac{1}{1+ \dfrac{1}{292 + ...}}}} = 3+ [7,15,1,292,...]_0,\]
obtaining the first five convergents $\big\{\frac{p_n}{q_n}\big\}_0^4= \big\{3,\frac{22}{7},\frac{333}{106},\frac{355}{113}, \frac{103993}{33102} \big\}$. The best upper bounds for $\{\theta_n\}_0^4$ using a four digit decimal expansion are \\ $\{0.1416,0.0612, 0.9351, 0.0034, 0.6333\}$. This helps explain why the first convergent $\frac{22}{7}$ and especially the third convergent $\frac{355}{113}$, first used by Archimedes (c. 287BC - c. 212BC), were popular approximations for $\pi$ throughout antiquity.

\section{The Markoff Sequence}{}

\noindent In 1891, Hurwitz proved that for all irrational numbers $r$, there exist infinitely many pairs of integers $p$ and $q$, such that $\theta(r,\frac{p}{q}) < \frac{1}{\sqrt{5}}$. Furthermore, he proved that this constant, called the \bf{Hurwitz Constant}, is the best possible in the sense that this result is in general no longer true if we replace this constant with a smaller one. Therefore, all irrational numbers possess infinitely many high quality approximations using rational numbers, whose associated approximation coefficients are less than $\frac{1}{\sqrt{5}}$. Using Legendre's result, we see that all these high quality approximations must belong to the sequence of regular continued fraction convergents for $r$ and, in particular, are all best approximates for $r$. We thus restate Hurwitz's theorem as 
\begin{equation}\label{Hurwitz}
\displaystyle{\liminf_{n\to\infty}}\big\{\theta_n(r)\big\} \le \frac{1}{\sqrt{5}}
\end{equation}
for all irrational numbers $r$, where this weak inequality cannot be sharpen to a strict one.\\   

We use the value of $\displaystyle{\liminf_{n \to \infty}\theta_n(r)}$ to measure how well can $r$ be approximated by rational numbers. The set of values taken by $\displaystyle{\liminf_{n \to \infty}}\big\{\theta_n(r)\big\}$, as $r$ varies in the set of all irrational numbers, is called the \bf{Lagrange Spectrum} and those irrational numbers $r$ for which $\displaystyle{\liminf_{n \to \infty}}\big\{\theta_n(r)\big\}  > 0$ are called \bf{badly approximable numbers}. It is known \cite[Theorem 23]{K} that $r$ is badly approximable if and only if the sequence digits $\big\{b_n(r)\big\}_{n=0}^\infty$ in the RCF expansion for $r$ is bounded. In particular, all \bf{quadratic surds}, i.e. irrational numbers of the form $r = \frac{p \pm \sqrt{D}}{q}$ where $p, q$ and $D$ are integers and $D$ is positive and square-free, are badly approximable, since their sequence of RCF digits must eventually exhibit a repetitive pattern \cite{Burger}.\\ 

The largest member of the Lagrange spectrum is the supremum of the set $\displaystyle{\liminf_{n \to \infty}}\big\{\theta_n(r)\big\}$, where $r$ is taken over the set of irrational numbers. From the inequality \eqref{Hurwitz}, this number is the Hurwitz Constant $\frac{1}{\sqrt{5}}$. Furthermore, it is known that this supremum is attained precisely for these $r$'s belonging to the set of \bf{noble numbers}. This set of irrational numbers consists of all the quadratic surds whose sequence of digits has an infinite tail of 1's, amongst which the simplest is the \bf{golden ratio} $\phi = \frac{\sqrt{5}-1}{2} = [\tab \overline{1} \tab ]_0 = [1,1,1,...]_0$. Once we remove the countable set of noble numbers $\IR_1$ from the irrational numbers, a gap emerges in the spectrum, as seen from the quantity $\frac{1}{\sqrt{8}}$, which is the value of $\displaystyle{\sup_{\IR - (\IQ \cup \IR_1)}\bigg\{\liminf_{n \to \infty}}\big\{\theta_n(r)\big\}\bigg\}$. Furthermore, this supremum is attained precisely for those $r$'s whose continued fraction expansion has an infinite tail of 2's, amongst which the simplest is the \bf{silver ratio} $\sqrt{2}-1 = [\tab \overline{2} \tab]_0$. 

\section{Structure for the sequence of approximation coefficients}{}\label{SOAC_classical}

\noindent  In this section we state more recent results about the sequence of approximation coefficients, revealing elegant internal structure as well as simple connections to the sequence of RCF digits. Fix $r  \in \IR - \IQ$ and for all $n \ge 0$, let $b_n := b_n(r)$ and $\theta_n := \theta_n(r)$ be the n$^{\operatorname{th}}$ member in the sequences of RCF digits and approximation coefficients for $r$. In 1895, Vahlen proved that for all $n \ge 1$ we have $\min\{\theta_{n-1},\theta_n\} < \frac{1}{2}$. In 1903, Borel improved Hurwitz result \eqref{Hurwitz} and proved that 
\begin{equation}\label{borel}
\min\{\theta_{n-1},\theta_n,\theta_{n+1}\} < \dfrac{1}{\sqrt{5}}, \hspace{1pc} n \ge 1.
\end{equation}
Furthermore, this constant is the best possible bound, that is, it cannot be replaced with any smaller constant. In 1921, Perron \cite{Perron} proved that
\begin{equation}\label{Perron}
\dfrac{1}{\theta_{n-1}} = [b_{n+1},b_{n+2},...]_0 + b_n + [b_{n-1},b_{n-2},...,b_1]_0, \hspace{1pc} n \ge 2.
\end{equation}
This equation implies that we need to know the value for every member in the sequence $\{b_n\}_1^\infty$ in order to generate a single member in the sequence $\{\theta_n\}_1^\infty$.\\

In 1978, Jurkat and Peyerimhoff \cite{JP}  characterized the \bf{space of approximation pairs} $\Gamma_0 := \big\{(\theta_{n-1},\theta_n) \subset \IR^2: r \in \IR - \IQ, \tab n \ge 1 \big\}$. They showed this space is a proper full measure subset of the region in the Cartesian plane which is the interior of the triangle with vertices $(0,0), \tab (0,1)$ and $(1,0)$. This proves for all irrational numbers and $n \ge 1$ that $\theta_{n-1} + \theta_n < 1$ and, in particular, implies Vahlen's result. In addition, they partitioned $\Gamma_0$ to subregions $\big\{P_a^\#\big\}_{a=0}^\infty$ so that $a_{n+1} = a$ precisely when $(\theta_{n-1},\theta_n) \in P_a^\#$. Finally, they showed that the sequence $\{\theta_n\}_1^\infty$ has the following elegant symmetries:
\begin{equation}\label{theta_n+1_classical}
\theta_{n+1} = \theta_{n-1} + b_{n+1}\sqrt{1 - 4\theta_{n-1}\theta_n} - b_{n+1}^2\theta_n, \hspace{2pc} n \ge 0,
\end{equation}
\begin{equation}\label{theta_n-1_classical}
\theta_{n-1} = \theta_{n+1} + b_{n+1}\sqrt{1 - 4\theta_{n+1}\theta_n} - b_{n+1}^2\theta_n, \hspace{2pc} n \ge 1.
\end{equation}
The inequality
\begin{equation}\label{triple_min_classical}
\min\big\{\theta_{n-1},\theta_n,\theta_{n+1}\big\} < \frac{1}{\sqrt{b_{n+1}^2 + 4}}, \hspace{1pc} n \ge 1,
\end{equation}
which is an improvement of Borel's result, is due to Bagemihl and McLaughlin \cite{BM}, while the symmetric result 
\begin{equation}\label{triple_max_classical}
\max\big\{\theta_{n-1},\theta_n,\theta_{n+1}\big\} > \frac{1}{\sqrt{b_{n+1}^2 + 4}}, \hspace{1pc} n \ge 1
\end{equation}
was proved in 1983 by Tong \cite{Tong}.

\section{Dynamic approach to regular continued fractions}{}\label{DA}

\noindent From a dynamic point of view, the regular continued fraction expansion is a concrete realization of the symbolic representation of irrational numbers in the unit interval under the iterations of the \bf{Gauss Map} 
\[T_0: [0,1) \to [0,1), \hspace{1pc} T_0(x) := \frac{1}{x} - \bigg\lfloor \frac{1}{x} \bigg\rfloor, \hspace{1pc} T_0(0) := 0.\] 
The Gauss map is the fractional part of the homeomorphism $A_0:(0,1) \to (1, \infty), \tab x \mapsto \frac{1}{x}$. This homeomorphism extends to the M$\operatorname{\ddot{o}}$bius transformation $\hat{A}_0:\hat{\IC} \to \hat{\IC}, \tab z \mapsto \frac{1}{z}$ mapping $[0,1]$ bijectively to $[1,\infty]$ with $0 \mapsto \infty$ and $1 \mapsto 1$.\\ 

\newpage

Given a real number $x_0 \in (0,1)$, we expand $x_0$ as a regular continued fraction by defining the \bf{digit} and \bf{future} of $x_0$ at time $n \ge 1$ to be $b_n :=  \lfloor A_0(x_{n-1}) \rfloor$ and 
\[x_n :=   A_0(x_{n-1}) - b_n = A_0(x_{n-1}) - \lfloor A_0(x_{n-1}) \rfloor =  T_0(x_{n-1})  = T_0\circ{T_0}(x_{n-2}) = ... =  T_0^n(x_0).\] 
We repeat this iteration process as long as $x_{n-1} > 0$, in which case $x_{n-1} = \frac{1}{b_n + x_n}$. This yields the RCF expansion 
\[x_0 = \dfrac{1}{b_1 + x_1} = \dfrac{1}{b_1 + \dfrac{1}{b_2 + x_2}} = ...\]
If $x_N = 0$ for some $N \ge 1$, which happens precisely when $x_0$ is a rational number, we terminate the iteration process and write $x_0 = [b_1,...,b_N]_0$. For instance
\[ [1,1,2]_0 = \frac{1}{1 + \dfrac{1}{1 + \dfrac{1}{2}}} = \dfrac{3}{5}.\] 
If this iteration process continued indefinitely, we let $N := \infty$ and write $x_0 = [b_1,b_2,...]_0$. When $x_0$ is an irrational number the map $T_0$ is realized as a left shift operator on the sequence of digits
\[ x_n = [b_{n+1},b_{n+2},...]_0 = T_0\big([b_n,b_{n+1},b_{n+2},...]_0\big) = T_0(x_{n-1}), \hspace{1pc} n \ge 1.\]
Furthermore, the probability measure on the interval $\mu_0(E) := \frac{1}{\ln{2}}\int_E\frac{1}{x+1}dx$, known as the \bf{Gauss Measure}, is both invariant and ergodic with respect to $T_0$.

\newpage

\section{Backwards continued fractions}{}

\noindent Another well known continued fraction theory is the \bf{backwards continued fractions} or \bf{BCF}, stemming from the \bf{Renyi Map} 
\[T_1: [0,1) \to [0,1), \hspace{1pc} T_1(x) := \frac{1}{1-x} - \bigg\lfloor \frac{1}{1-x} \bigg\rfloor. \] 
The Renyi map is the fractional part of the homeomorphism $A_1:(0,1) \to (1, \infty), \tab x \mapsto \frac{1}{1-x}$. This homeomorphism extends to the M$\operatorname{\ddot{o}}$bius transformation $\hat{A}_1:\hat{\IC} \to \hat{\IC}, \tab z \mapsto \frac{1}{1-z}$ mapping $[0,1]$ bijectively onto $[1,\infty]$ with $0 \mapsto 1$ and $1 \mapsto \infty$.\\

Given a real number $x_0 \in (0,1)$, we expand $x_0$ as a BCF by defining the digit and future of $x_0$ at time $n \ge 1$ to be $b_n :=  \lfloor A_1(x_{n-1}) \rfloor$ and 
\[x_n := T_1(x_{n-1}) = A_1(x_{n-1}) - \lfloor A_1(x_{n-1}) \rfloor =  A_1(x_{n-1}) - b_n.\]
We repeat this iteration process as long as $x_{n-1} > 0$, in which case $x_{n-1} = 1 - \frac{1}{a_n + x_n}$. This yields the continued fraction expansion 
\[x_0 = 1 - \dfrac{1}{b_1 + x_1} = 1 - \dfrac{1}{b_1 + 1 - \dfrac{1}{b_2 + x_2}} = ...\]
We write $x_0 := [b_1,b_2,...]_1$, where we use the subscript $[\cdot]_1$ to distinguish  this expansion from the RCF expansion $[\cdot]_0$. If $x_N = 0$ for some $N \ge 1$, which happens precisely when $x_0$ is a rational number, we terminate the iteration process and write $x_0 = [b_1,...,b_N]_1$. For instance
\[ [2,2]_1 = 1 - \dfrac{1}{2 + 1 - \dfrac{1}{2}} =  \dfrac{3}{5}. \]
If this iteration process continued indefinitely, we let $N := \infty$ and write $x_0 = [b_1,b_2,...]_1$. We define the associated convergents and sequence of approximation coefficients as in the RCF case, that is, $\frac{p_n}{q_n} := [b_1,...,b_n]_1$ and $\theta_n := q_n^2\abs{x_0-\frac{p_n}{q_n}}$.\\

When $x_0$ is an irrational number, the map $T_1$ is realized as a left shift operator on the infinite sequence of digits $\{b_n\}_1^\infty$. Furthermore, the measure $\mu_1(E) := \int_E\frac{1}{x}dx$ on the interval is both invariant and ergodic with respect to $T_1$. However, unlike the Gauss Measure, it is not finite on the interval. This deficiency helps explain why the backwards continued fraction theory did not gain nearly as much attention as the RCF theory, even though it may lead to quicker expansions, as seen from our $\frac{3}{5}$ example. For more about the metrical properties and symbolic dynamics of the BCF theory, refer to \cite{HBackwards,Renyi}.   

\section{Gauss-like and Renyi-like continued fractions}{}

\noindent In general, a (mod 1)  M$\operatorname{\ddot{o}}$bius transformation may lead to a generalized continued fraction theory. Given a real number $k>0$, the homeomorphisms $A_{(0,k)}:(0,1) \to (0,\infty), \tab x \mapsto \frac{k(1-x)}{x}$ extend to M$\operatorname{\ddot{o}}$bius transformations, which map $[0,1]$ bijectively onto $[0,\infty]$ with $0 \mapsto \infty$ and $1 \mapsto 0$. Given a real number $k>1$, the homeomorphisms $A_{(1,k)}:(0,1) \to (0,\infty), \tab x \mapsto \frac{k{x}}{1-x}$ extend to M$\operatorname{\ddot{o}}$bius transformations, which map $[0,1]$ bijectively onto $[0,\infty]$ with $0 \mapsto 0$ and $1 \mapsto \infty$. The maps, $T_{(m,k)}(x) := A_{(m,k)}(x) - \lfloor A_{(m,k)}(x) \rfloor$ are called \bf{Gauss-like} and \bf{Renyi-like} for $m=0$ and $m=1$ respectively.
 
Given $m \in \{0,1\}$, a real number $k >m$ and an initial seed $x_0 \in (0,1)$, we expand $x_0$ as an (m,k)-continued fraction by defining the digit and future of $x_0$ at time $n \ge 1$ to be $a_n :=  \lfloor A_{(m,k)}(x_{n-1}) \rfloor$ and 
\[x_n :=  A_{(m,k)}(x_{n-1}) - a_n = A_{(m,k)}(x_{n-1}) - \lfloor A_m(x_{n-1}) \rfloor =  T_{(m,k)}(x_{n-1}) = T_{(m,k)}^n(x_0).\] 
We repeat this iteration process as long as $x_{n-1} > 0$, in which case $x_{n-1} = m + \frac{(-1)^m{k}}{a_n + k + x_n}$, yielding the continued fraction expansion 
\[x_0 = m + \dfrac{(-1)^m{k}}{a_1 + k + x_1} = m + \dfrac{(-1)^m{k}}{a_1 + k + m + \dfrac{(-1)^m{k}}{a_2 + k + x_3}} = ...\]
If $x_{N} = 0$ for some $N \ge 1$, we terminate the iteration process and write $x_0 = [a_1,...,a_N]_{(m,k)}$. If $x_0$ has an infinite expansion with respect to $T_{(m,k)}$ then this map is realized as a left shift operator on the infinite sequence of digits, that is 
\[  x_n = [b_{n+1},b_{n+2},...]_{(m,k)} = T_{(m,k)}\big([b_n,b_{n+1},b_{n+2},...]_0\big) = T_{(m,k)}(x_{n-1}), \hspace{1pc} n \ge 1.\]
\begin{remark}
Replacing $k$ with $1$ in the digit expansion associated with $T_{(m,k)}$ results in the RCF and BCF theories for $m=0$ and $m=1$ respectively, but with digits which are smaller by one.
For instance, plugging $k=1$ into
\[ [0,1,2]_{(0,k)} = \dfrac{k}{0+k+\dfrac{k}{1+k+\dfrac{k}{2+k}}}\] 
yields the fraction $[1,2,3]_0 = \frac{7}{10}$ and plugging $k=1$ into
\[ [0,1,2]_{(1,k)} = 1 - \dfrac{k}{0+k+1-\dfrac{k}{1+k+1-\dfrac{k}{2+k}}}\] 
yields the fraction $[1,2,3]_1 = \frac{5}{13}$.
 We changed the labeling of the digits from $b_n$ to $a_n = b_n - 1$ to help avoid this confusion. 
\end{remark}

\section{Statement of major results}{}


\noindent We fix $m \in \{0,1\}$ and a real number $k > m$. The \bf{natural extension} of the dynamical system induced by $T_{(m,k)}$ will allow us extend the sequences $\{a_n\}_1^\infty$ and $\{\theta_n\}_1^\infty$ to \bf{bi-sequences} $\{a_n\}_{-\infty}^\infty$ and $\{\theta_n\}_{-\infty}^\infty$, continuing indefinitely from the left as well as from the right. Many beautiful results from the classical theory generalize well to this setting. For instance, in \cite{HM} Haas and Molnar proved a generalized version of the Khintchine-L$\operatorname{\acute{e}}$vi Theorem.\\

In theorems \ref{thm_a_n+1} and \ref{thm_a_n+1_R}, we will be able to  determine the digit $a_{n+1}$ directly from either one of the pairs $\{\theta_{n-1},\theta_n\}$ or $\{\theta_n,\theta_{n+1}\}$, proving that for all $n \in \mathbb{Z}$ we have
 \[a_{n+1} =  \bigg\lfloor \frac{(-2)^m{k}\theta_{n-1}}{\big(1 - \sqrt{1+(-4)^{m+1}k{\theta_{n-1}}\theta_n}\big)} -k \bigg\rfloor = \bigg\lfloor\frac{(-2)^m{k}\theta_{n+1}}{\big(1 - \sqrt{1+(-4)^{m+1}k{\theta_{n+1}}\theta_n}\big)} -k \bigg\rfloor.\] 
 If either $m=0$ and $k \ge 1$ or $k>m=1$ , we will prove in theorems \ref{theta_pm_1} and \ref{theta_pm_1_R} that for all $n \in \mathbb{Z}$ we have
\[ \theta_{n+1} = \theta_{n-1} + (-1)^m\dfrac{\sqrt{1+(-4)^{m+1}k{\theta_{n-1}}\theta_n}}{k}(a_{n+1}+k+m) + (-1)^{m+1}\dfrac{\theta_n}{k}(a_{n+1}+k+m)^2,\]
\[ \theta_{n-1} = \theta_{n+1} + (-1)^m\dfrac{\sqrt{1+(-4)^{m+1}k{\theta_{n+1}}\theta_n}}{k}(a_{n+1}+k+m) + (-1)^{m+1}\dfrac{\theta_n}{k}(a_{n+1}+k+m)^2.\]
Consequently, we will be able to recover the entire bi-sequence of approximation coefficients from a single pair of successive members.

For any non-negative integer $a$, we will define the (m,k)-irrational constants 
\[\xi_{(m,k,a)} := [ \tab \overline{a} \tab ]_{(m,k)} = [a,a,...]_{(m,k)} =  \frac{1}{2}\big(\sqrt{(a+k-m)^2 +4k}-(a+k-m)\big)\]
and
\[ C_{(m,k,a)} := \dfrac{1}{\sqrt{(a+k+m)^2 + (-1)^m{4k}}}.\]
Then in theorems \ref{theta_n_constant} and \ref{theta_n_constant_R}, we will prove that the following are equivalent:
\begin{enumerate}
\item $a_n = a$ for all $n \in \mathbb{Z}$.
\item $(x_0,y_0) = (\xi_a,-a-k-\xi_a)$.
\item $\theta_{-1} = \theta_0 = C_a$.   
\item $\theta_n = C_a$ for all $n \in \mathbb{Z}$.  
\end{enumerate}

For the Gauss-like case $m=0$, we will generalize the inequalities \eqref{triple_min_classical} and \eqref{triple_max_classical} for all $k \ge 1$ in theorem \ref{triple_thm}, proving
\[\min\{\theta_{n-1},\theta_n,\theta_{n+1}\} < \dfrac{1}{\sqrt{(a_{n+1}+k)^2 + 4k}}, \hspace{1pc} n \in \mathbb{Z},\]
\[\max\{\theta_{n-1},\theta_n,\theta_{n+1}\} > \dfrac{1}{\sqrt{(a_{n+1}+k)^2 + 4k}}, \hspace{1pc} n \in \mathbb{Z},\]
where these constants are the best possible. In tandem with theorem \ref{silver}, we will show that the constants $\frac{1}{\sqrt{k^2 + 4k}}$ and $\frac{1}{\sqrt{k^2 + 6k +1}}$ are the first two members for the associated Markoff Sequence, which generalize the quantities $\frac{1}{\sqrt{5}}$ and $\frac{1}{\sqrt{8}}$ assigned to the golden and silver ratios respectively.

In theorem \ref{a_n+1_classical}, we will prove the special case of theorems \ref{thm_a_n+1} applied to the classical RCF sequence of digits and approximation coefficients, proving that if $\{b_n\}_1^\infty$ and $\{\theta_n\}_1^\infty$ are the RCF sequence of digits and approximation coefficients for $x_0 \in (0,1) - \IQ$, then  
\[b_{n+1} =  \bigg\lfloor \frac{2\theta_{n-1}}{1 - \sqrt{1-4\theta_{n-1}\theta_n}} \bigg\rfloor = \bigg\lfloor \frac{2\theta_{n+1}}{1 - \sqrt{1 - 4\theta_{n+1}\theta_n}}\bigg\rfloor,  \hspace{1pc} n \ge 2.\]
Together with formulas \eqref{theta_n+1_classical} and \eqref{theta_n-1_classical}, we obtain
\[\theta_{n+1} = \theta_{n-1} +  \bigg\lfloor \frac{2\theta_{n-1}}{1 - \sqrt{1-4\theta_{n-1}\theta_n}} \bigg\rfloor\sqrt{1 - 4\theta_{n-1}\theta_n} -  \bigg\lfloor \frac{2\theta_{n-1}}{1 - \sqrt{1-4\theta_{n-1}\theta_n}} \bigg\rfloor^2\theta_n, \hspace{1pc} n \ge 2,\]  
\[\theta_{n-1} = \theta_{n+1} + \bigg\lfloor \frac{2\theta_{n+1}}{1 - \sqrt{1 - 4\theta_{n+1}\theta_n}}\bigg\rfloor\sqrt{1 - 4\theta_{n+1}\theta_n} - \bigg\lfloor \frac{2\theta_{n+1}}{1 - \sqrt{1 - 4\theta_{n+1}\theta_n}}\bigg\rfloor^2\theta_n, \hspace{1pc} n \ge 2.\]
Consequently, we will be able to recover the entire classical (one sided) sequence of approximation coefficients from a single pair of successive members.

For the Renyi case $m=1$, we let $l$ and $L$ be the essential bounds on the bi-sequence of digits, that is 
\[0 \le l := \liminf_{n \in \mathbb{Z}}\{a_n\} \le \limsup_{n \in \mathbb{Z}}\{a_n\} =: L \le \infty.\] 
Then we will prove in theorem \ref{thm_mu_bounds_R} that
\[ \dfrac{1}{\sqrt{(L+k+1)^2 - 4k}} \le \displaystyle{\liminf_{n \in \mathbb{Z}}}\{\theta_n\} \le \displaystyle{\limsup_{n \in \mathbb{Z}}}\{\theta_n\} \le \dfrac{1}{\sqrt{(l+k+1)^2 - 4k}},\]
where these constants are the best possible (we take the left and right hand sides to be zero when $L = \infty$ and $l = \infty$ respectively). 
  \chapter{Preliminaries}{}
In this chapter, we summarize the work of Haas and Molnar \cite{HM,HM2,HBackwards,HNatural, Molnar} and give the basic definitions and facts, which we intend to use in the body of this work.\\ 
\section{Basic definitions}{} 

\noindent Given $m \in \{0,1\}$ and $k > m$, define the homeomorphisms $A_{(m,k)}:[0,1) \to [0,\infty)$ by
\[A_{(0,k)}(x) :=
\begin{cases}
0 & \text{if $x=0$}\\
\dfrac{k(1-x)}{x} & \text{otherwise}
\end{cases}
\]
and 
\[A_{(1,k)}(x) := \dfrac{k{x}}{1-x}.\] 
The \bf{Gauss-like} and \bf{Renyi-like} transformations $T_{(m,k)}: [0,1) \to [0,1)$
\begin{equation}\label{T_m_k}
T_{(m,k)}(x) := A_{(m,k)}(x) - \lfloor A_{(m,k)}(x) \rfloor
\end{equation}
are defined to be the fractional part of $A_{(m,k)}$ for $m=0$ and $m=1$ respectively. We expand the initial seed $x_0 \in [0,1)$ in a continued fraction using the following iteration process: 
\begin{enumerate}
\item Set $n :=1$.
\item If $x_{n-1} = 0$, stop and write $x_0=0$ if $n=1$ or $x_0 = [a_1,...,a_{n-1}]_{(m,k)}$ if $n>1$ and exit. 
\item Set  the \bf{reminder} of $x_0$ at time $n$ to be $r_n := A_{(m,k)}(x_{n-1}) \in (0,\infty)$ and write the (m,k)-CF expansion for $x_0$ at time $n$ as $x_0=[r_1]_{(m,k)}$ if $n=1$ or $x_0 = [a_1,...,a_{n-1},r_n]_{(m,k)}$ if $n > 1$. 
\item Set the \bf{digit} and \bf{future} of $x_0$ at time $n$ to be $a_n := \lfloor r_n \rfloor \in \mathbb{Z}^+$, where $\mathbb{Z}^+ := \mathbb{Z} \cap [0,\infty)$, and $x_n := r_n - a_n \in [0,1)$ respectively. Increase $n$ by one and go to step (ii). 
\end{enumerate}

If for a given $m,k$ and $x_0$ this scheme stops during the $N^{th}$ iteration, we say $x_0$ has a finite $(m,k)$-expansion. If this iteration scheme continues indefinitely, we say $x_0$ has an infinite $(m,k)$-expansion and take $N+1 = N := \infty$. For all $1 \le n < N + 1$ we have $r_n := A_{(m,k)}(x_{n-1}), \tab a_n := \big\lfloor A_{(m,k)}(x_{n-1}) \big\rfloor$ and
\[x_n :=  r_n - a_n = A_{(m,k)}(x_{n-1}) - \big\lfloor A_{(m,k)}(x_{n-1}) \big\rfloor = T_{(m,k)}(x_{n-1}) = T_{(m,k)}^n(x_0).\]
The (m,k)-expansion for $x_0$ is defined for all times $n$ with $1 \le n < N$ as $[a_1,...,a_n,r_n]_{(m,k)}$. If $N < \infty$, the (m,k)-expansion at time N is defined as $[a_1,...,a_N]_{(m,k)}$. The future of $x_0$ at time $n \ge 1$ is
\begin{equation}\label{x_n} 
x_n= [r_{n+1}]_{(m,k)} = [a_{n+1},r_{n+2}]_{(m,k)} = [a_{n+1},a_{n+2}, r_{n+3}]_{(m,k)} = ...
\end{equation} 
Haas and Molnar \cite{HM} showed the sequence $\{a_n\}_1^N$ is uniquely determined by $x_0$, hence the sequence $\{r_n\}_1^N$ is also uniquely determined by $x_0$.
\begin{remark}\label{digit_remark}
The special case $k=1$ corresponds with the classical Gauss and Renyi maps for $m=0$ and $m=1$ respectively, but with partial quotients which are smaller than the classical representation by one. For instance, 
\[[0,1,2]_{(0,k)} = \dfrac{k}{0+k+\dfrac{k}{1+k+\dfrac{k}{2+k}}}\]
and
\[[0,1,2]_{(1,k)} = 1 - \dfrac{k}{0+k+1-\dfrac{k}{1+k+1-\dfrac{k}{2+k}}}\]
will yield, after plugging $k=1$, the rational numbers $[1,2,3]_0$ and $[1,2,3]_1$ respectively.
\end{remark}

\section{Terminal points and intervals of monotonicity}{}
\noindent As in Diophantine approximation and the RCF expansion, the first step in developing the theory for the sequence of approximation coefficients is to remove those numbers in the unit interval with a finite continued fraction expansion. Since for the classical Gauss Map, these numbers correspond with the set of rational numbers, we analogously define for all $N \ge 0$ the set of \bf{(m,k)-rationals of rank N} by 
\[\mathbb{Q}^{(N)}_{(m,k)} := \big\{x \in [0,1): T^n_{(m,k)}(x) = 0 \tab \text{for some $n \le N$} \big\}.\]
For instance 
\[\mathbb{Q}^{(0)}_{(0,k)} = \mathbb{Q}^{(0)}_{(1,k)} = \{0\}, \hspace{1pc} \mathbb{Q}^{(1)}_{(0,k)} = \{0\} \cup \bigg\{\dfrac{k}{a+k} : a \in \mathbb{N}\bigg\}, \hspace{1pc} \mathbb{Q}^{(1)}_{(1,k)} = \bigg\{\frac{a}{a+k} : a \in \mathbb{Z}^+\bigg\}.\]
We further define the set of \bf{(m-k)-rationals} to be $\mathbb{Q}_{(m,k)} := \displaystyle{\lim_{n \to \infty}}\mathbb{Q}^{(n)}_{(m,k)}$. Then $x \in \IQ_{(m,k)}$ if and only if $x=0$ or $x$ has a finite (m,k)-expansion, that is, there exist a unique finite sequence of digits $\{a_n\}_1^N$ such that $x =[a_1,a_2,...,a_N]_{(m,k)}$.\\

The \bf{interval of monotonicity}, also known as the cylinder set, of rank $N \ge 0$ associated with the finite sequence of $N$ non-negative integers $\{a_1,...,a_N\}$ is $\Delta^{(0)} := (0,1)$ and $\Delta^{(N)}_{a_1,...,a_N} := \big\{x_0 \in (0,1): a_n(x_0) = a_n \tab \text{for all $1 \le n \le N$}\big\}$. In \cite{HM}, Haas and Molnar proved that the restriction of $T_{(m,k)}^N$ to the interior of any interval of monotonicity of rank $N$ is a homeomorphism onto $(0,1)$ and for all $N \ge 0$ we have
\[(0,1) = \displaystyle{\bigcup_{a_1,...,a_N \in \mathbb{Z}^+}}\Delta^{(N)}_{a_1,...,a_N},\] 
where this union is disjoint in pairs.
\section{The natural extension}
\noindent The probability measures on the interval $\mu_{(m,k)}(E) := \bigg(\ln\big(\frac{k+1-m}{k-m}\big)\bigg)^{-1}\int_E\frac{1}{x+k-m}dx$ are both invariant and ergodic with respect to the map $T_{(m,k)}$. The induced \bf{dynamical system} $\{(0,1),\mathcal{L},\mu,T\}_{(m,k)} := \big\{(0,1) - \mathbb{Q}_{(m,k)}, \mathcal{L}, \mu_{(m,k)}, T_{(m,k)}\big\}$, where $\mathcal{L}$ is the Lebesgue $\sigma$-algebra, is not invertible since $T_{(m,k)}$ is not a bijection. However, there is a canonical way to extend non-invertible dynamical systems to invertible ones \cite{Rohlin}. The realization for the natural extension we are about to present was originally introduced to the special case $m=0, \tab k=1$ by Nakada in \cite{N} and plays a vital role in the proof of the Doblin-Lenstra conjecture \cite{BJW}. Define the regions 
\begin{equation}\label{Omega'}
\Omega' _{(m,k)} := [0,1) \times (-\infty,m-k]
\end{equation} 
and 
\[\IQ'_{(m,k)} := \big\{m - k - b - q : b \in \mathbb{Z}^+ \tab \text{and} \tab q \in \IQ_{(m,k)} \big\} \subset (-\infty, m-k]. \]
Define the \bf{space of dynamic pairs}
\begin{equation}\label{Omega}
\Omega_{(m,k)} := \Omega'  - \bigg(\big([0,1) \times \IQ'_{(m,k)}\big) \cup \big(\IQ_{(m,k)} \times (-\infty,m-k]\big)\bigg).
\end{equation}  
The \bf{natural extension maps} $\ccT_{(m,k)}:\Omega_{(m,k)} \to \Omega_{(m,k)}$ are defined by
\begin{equation}\label{ccT_intro}
\ccT_{(m,k)}(x,y)= \big(A_{(m,k)}(x) - \lfloor A_{(m,k)}(x) \rfloor, \hspace{1pc} A_{(m,k)}(y) - \lfloor A_{(m,k)}(x) \rfloor\big). 
\end{equation}
In \cite[Theorem 1]{HM}, Haas proved that after fixing $m \in \{0,1\}$ and $k >m$, the map $\ccT_{(m,k)}$ is both invariant and ergodic with respect to the probability measure $\rho_{(m,k)}(D) := \ln\big(\frac{k+1-m}{k - m}\big)^{-1}\iint_D\frac{dxdy}{(x-y)^2}$. Furthermore, the dynamical system $\{(0,1),\mathcal{L},\mu,T\}_{(m,k)}$ is realized as a left factor of the invertible dynamical system $\big\{\Omega_{(m,k)},\mathcal{L}^2,\rho_{(m,k)},\ccT_{(m,k)}\big\}$. \\

Given the initial seed pair $(x_0,y_0) \in \Omega_{(m,k)}$ and $n \in \mathbb{Z}$, we define the \bf{dynamic pair} of $(x_0,y_0)$ at time $n$ to be $(x_n,y_n) := \ccT^n_{(m,k)}(x_0,y_0)$. Since $y_0 < m-k$ there exists a unique non-negative integer $a_0$ such that $m - k - a_0 - y_0 \in [0,1)$. Furthermore, $y_0 \notin \IQ'_{(m,k)}$ implies that $m -k -a_0 - y_0 \notin \IQ_{(m,k)}$. also,  $x_0 \notin \IQ_{(m,k)}$, hence both $x_0 \in (0,1)$ and $m - k - a_0 - y_0 \in (0,1)$ are (m,k)-irrational numbers, enjoying infinite $(m,k)$-expansions. Conclude that $(x_n,y_n)$ is well defined for all $n \in \mathbb{Z}$.\\

Fix $N \in \mathbb{Z}$ and let $\{a_n\}_{n = N+1}^\infty$ be the unique sequence of non-negative integers and $\{r_n\}_{n = N+1}^\infty$ be the unique sequence of non-negative (m,k)-irrationals such that 
\[x_N = [r_{N+1}]_{(m,k)} = [a_{N+1},r_{N+2}]_{(m,k)} = [a_{N+1},a_{N+2},r_{N+3}]_{(m,k)} = ...\] 
When $N \ge 1$, this extends the formula \eqref{x_n} for the future of $x_0$ at time $N$. Accordingly, we call $x_N$ the \bf{future} of $(x_0,y_0)$ at time $N \in \mathbb{Z}$. We further let $\{a_n\}_{n = N}^{-\infty} = \{a_N, a_{N-1},...\}$ be the unique sequence of non-negative integers and $\{s_n\}_{n = N}^{-\infty}$ be the unique sequence of non-negative (m,k)-irrationals such that  
\begin{equation}\label{y_n}
m - k - a_N - y_N = [s_N]_{(m,k)} = [a_{N-1},s_{N-1}]_{(m,k)} = [a_{N-1},a_{N-2},s_{N-2}]_{(m,k)} = ... 
\end{equation}
$y_N$ is called the \bf{past} of $(x_0,y_0)$ at time $N \in \mathbb{Z}$. The bi-sequence $\{a_n\}_{-\infty}^\infty$ is called the (m,k)-\bf{digit bi-sequence} for $(x_0,y_0)$. We have
\begin{equation}\label{intro_ccT_explicit}
\ccT\big([a_1,r_2]_{(m,k)},  m - k - a_0 - [s_0]_{(m,k)}\big) = \big([r_2]_{(m,k)}, m - k - a_1 - [a_0,s_0]_{(m,k)}\big)  
\end{equation}
and
\begin{equation}\label{intro_ccTinv_explicit}
\ccT^{-1}\big([r_2]_{(m,k)},  m - k - a_1 - [a_0,s_0]_{(m,k)}\big) = \big([a_1,r_2]_{(m,k)}, m - k - a_0 - [s_0]_{(m,k)}\big),  
\end{equation}
that is, $\ccT$ is realized as an invertible left shift operator on the (m,k)-digit bi-sequence.
\section{Approximation coefficients via the natural extension}{}
\noindent Starting with $m \in \{0,1\}, \tab k > m$ and $x_0 \in (0,1) - \mathbb{Q}_{(m,k)}$, the (m,k)-rational number $\frac{p_n}{q_n} = [a_1,...,a_n]_{(m,k)}$ is called the \bf{convergent} of $x_0$ at time $n \ge 1$. The sequence of approximation coefficients $\big\{\theta_n(x_0)\big\}_1^\infty$ is defined just like the classical object 
\begin{equation}\label{theta_x_0}
\theta_n(x_0) := q_n^2\abs{x_0 - \frac{p_n}{q_n}}, \hspace{1pc} n \ge 1. 
\end{equation}
For all $n \ge 0$ define the \bf{past} of $x_0$ at time $n$ to be $Y_0 := m - k, \tab Y_1 := m-k - a_1\in (-\infty,m-k]$ and 
\begin{equation}\label{Y_n}
Y_n := m - k - a_N - [a_{N-1},...,a_1]_{(m,k)} \in (-\infty,m-k), \hspace{1pc} n \ge 2.
\end{equation}
Then in $\cite{HM}$, Haas proved  
\begin{equation}\label{theta_future_past}
\theta_{n-1}(x_0) = \dfrac{1}{x_n - Y_n}, \hspace{1pc} n \ge 1,
\end{equation}
which is the generalization of Perron's result \eqref{Perron}. Analogously, for all $n \in \mathbb{Z}$, we define the approximation coefficient of the dynamic pair $(x_0,y_0) \in \Omega_{(m,k)}$ at time $n-1$ to be 
\begin{equation}\label{theta_dynamic}
\theta_{n-1}(x_0,y_0) := \dfrac{1}{x_n - y_n}, \hspace{1pc} n \in \mathbb{Z}.
\end{equation}
\begin{proposition}\label{mu}
Given $m \in \{0,1\}, \tab k > m$ and $(x_0,y_0) \in \Omega_{(m,k)}$. Then 
\[\displaystyle{\liminf_{n\to\infty}}\big\{\theta_n(x_0)\big\} =  \displaystyle{\liminf_{n \to \infty}}\big\{\theta_n(x_0,y_0)\big\}\]
and
\[\displaystyle{\limsup_{n\to\infty}}\big\{\theta_n(x_0)\big\} =  \displaystyle{\limsup_{n \to \infty}}\big\{\theta_n(x_0,y_0)\big\}.\]
\end{proposition}
\begin{proof}
Let $\{a_n\}_{-\infty}^\infty$ be the digit bi-sequence for the (m,k)-expansion of $(x_0,y_0)$. From the definitions of $Y_n$ \eqref{Y_n} and $y_n$ \eqref{y_n} we see that, for all $n \ge 2$, both $m + k + a_n - y_n$ and $m + k + a_n - Y_n$ belong to the interval of monotonicity $\Delta_{(m,k)}^{a_{n-1},a_{n-2},...,a_1}$, whose depth is $n-1$. The length (as in Lebesgue measure) of intervals of monotonicity tends to zero as their depth tends to infinity \cite{HM}, we have $(y_n - Y_n) \to 0$. Since $x_n >0$ and $y_n < m-k$, we see that the sequence $\{x_n - y_n\}_0^\infty$ is uniformly bounded from below by the positive number $k - m$. Thus we conclude that  
\[\abs{\theta_{n+1}(x_0) - \theta_{n+1}(x_0,y_0)} = \abs{\frac{1}{x_n - Y_n} - \frac{1}{x_n - y_n}} \to 0 \tab \text{as $n \to \infty$},\] 
from which the result follows.
\end{proof}
  \chapter{BAC for Gauss-like maps}{}
In this chapter, we improve the results of Haas and Molnar in \cite{HM,Molnar}. The new results start from section \ref{SOAPRev} and the material leading to this section is given for the sake of completeness. Since we are only interested in the Gauss case $m=0$, we fix $k > 0$ and, in order to ease the notation, omit the subscript $\square_{(0,k)}$ throughout. 
\section{From dynamic pairs to approximation pairs}{}
\noindent We begin our investigation of the bi-sequence of approximation coefficients for $\ccT$ by defining the map $\Psi_{(0,k)} = \Psi: \{(x,y) \in \IR^2: x - y \ne 0\} \to \IR^2$
\begin{equation}\label{Psi} 
\Psi(x,y) := \bigg(\dfrac{1}{x-y}, -\dfrac{x{y}}{k(x-y)}\bigg).
\end{equation}
Clearly, $\Psi$ is well defined and continuous on the region $\{(x,y) \in \IR^2: x - y < 0\}$, hence it is well defined and continuous on its subset $\Omega' = [0,1) \times (-\infty,-k]$. 
\begin{proposition}\label{Psi_fold}
$\Psi$ is invariant under reflection about the line $x+y=0$ and is injective on the region $\{(x,y) \in \IR^2 : x + y < 0\}$. In particular, $\Psi$ is injective on $\Omega'$ if and only if $k \ge 1$.   
\end{proposition}
\begin{proof}
The equality $\Psi(x, y) = \big(\frac{1}{x-y}, -\frac{xy}{k(x-y)}\big) = \Psi(-y,-x)$ proves the first statement. Let $(x_1,y_1)$ and $(x_2,y_2)$ be points which are on or below the line $x + y = 0$, so that $x_1 + y_1 \le 0, \tab x_2 + y_2 \le 0$ and let $u_1,v_1,u_2,v_2 \in \IR$ be such that $(u_1,v_1) = \Psi(x_1,y_1)=  \Psi(x_2,y_2) = (u_2,v_2)$. Then the definition of $\Psi$ \eqref{Psi} implies that $\frac{1}{x_1-y_1} = u_1 = u_2 = \frac{1}{x_2-y_2}$, hence
\begin{equation}\label{x-y_G}
x_1 - y_1 = x_2 - y_2 \ne 0.
\end{equation}
Also 
\[-\dfrac{u_1}{k}x_1y_1 = v_1 = v_2 = -\dfrac{u_2}{k}x_2y_2  =  -\dfrac{u_1}{k}x_2y_2\]
and $u_1,u_2 > 0$ imply that  $x_1y_1 = x_2y_2$, so that
\[(x_1 +y_1)^2 = (x_1 - y_1)^2 + 4x_1y_1 = (x_2 - y_2)^2 + 4x_2y_2 = (x_2 + y_2)^2.\]
Since $x_1 + y_1 \le 0$ and $x_2 + y_2 \le 0$, this last equation proves $x_1 + y_1 = x_2 + y_2$, which in tandem with condition \eqref{x-y_G}, proves $x_1=x_2$ and $y_1 = y_2$. Hence $\Psi$ is injective on or below the line $x+y=0$. When $k \ge 1$ this condition holds for all $(x,y) \in \Omega$ and implies that $\Psi$ is injective on $\Omega'$.  When $0<k<1, \tab k<x_0<1$ and $-1<y_0<-k$, both the points $(x_0, y_0)$ and $(-y_0, -x_0)$, belong to $\Omega'$. Since their image under $\Psi$ is identical, $\Psi$ is not injective on $\Omega'$ in this case. 
\end{proof}
From the formulas \eqref{theta_dynamic} and \eqref{ccT_explicit}, we obtain 
\[\dfrac{1}{\theta_n} = x_{n+1} - y_{n+1} = \bigg(\dfrac{k}{x_n} - k -a_{n+1}\bigg)- \bigg(\dfrac{k}{y_n} - k -a_{n+1}\bigg) = -\dfrac{k(x_n-y_n)}{x_n{y_n}},\] 
so that
\begin{equation}\label{Psi_theta} 
\Psi(x_n,y_n) = (\theta_{n-1}, \theta_n).
\end{equation}
is the \bf{approximation pair} for $(x_0,y_0)$ at time $n$. 
\section{The space of approximation pairs}{}\label{SOAP}
\noindent Our goal in this section is to determine the finer structure of the \bf{space of approximation pairs} $\Gamma_{(0,k)} = \Gamma := \operatorname{Im}(\Psi) = \Psi(\Omega) \subset \IR^2$. In this section we will find $\Gamma$ when $k \ge 1$. However, except for one technical issue, which we defer to the next section, the theory remains valid to the $0<k<1$ case as well. We begin by defining the regions $P_{(0,k,a)} = P_a := (0,1) \times(-k-a-1,-k-a]$ when $a >0$, $P_{(0,k,0)} = P_0 := (0,1) \times(-k-1,-k)$ and $P_{(0,k,a)}^\# = P_a^\# := \Psi(P_a) \subset \IR^2$ for all $a \ge 0$. Then $P_a \cap P_b = \emptyset$ whenever $a \ne b$ and $\Omega' = \displaystyle{\bigcup_{a\ge0}P_a}$. From proposition \ref{Psi_fold}, we know that $\Psi$ is injective on and below the diagonal $x+y=0$. Thus, we conclude that 
\begin{proposition}\label{a=0,0<k<1}
$\Psi$ is injective on $P_a$ unless $0<k<1$ and $a=0$.
\end{proposition}
\begin{theorem}\label{P_a}
When $a>0$, $P_a^\#$ is the region in the $(u,v)$ plane, which is the intersection of the unbounded regions $v > 0, \tab (a+k)^2u + k{v} \le a+k, \tab (a+k+1)^2u + k{v} > a+k+1$ and $u{k}+v < 1$. $P_0^\#$ is the open quadrangle in the $(u,v)$ plane, which is the intersection of the unbounded regions $v > 0, \tab k{u} + v < 1, \tab (k+1)^2u + k{v} > k+1$ and $u + k{v} < 1$. 
\end{theorem}	
To prove this theorem, we first let $p_{(0,k,a)} = p_a$ be the open horizontal line segment $(0,1) \times \{-k-a\}$ and let $p_a^\#$ be its image under $\Psi$,
\begin{lemma}\label{p_a}
$p_a^\#$ is the open line segment $(a+k)^2{u} +kv = a+k$ between $\big(\frac{1}{a+k}, 0\big)$ and $\big(\frac{1}{a+k+1},\frac{a+k}{k(a+k+1)}\big)$.
\end{lemma}
\begin{proof}
Let $(x,y) \in p_a \subset P_a \subset \Omega$ so that $y = -k-a$ and let $(u,v) := \Psi(x,y)$. From the definition of $\Psi$ \eqref{Psi}, we obtain
\begin{equation}\label{u=f(x,y)}
u = \dfrac{1}{x-y} = \dfrac{1}{x+k+a}
\end{equation}
and
\begin{equation}\label{v=f(x,y)}
v = -\dfrac{x{y}}{x-y}
\end{equation}
so that
\begin{equation}\label{v=f(u)}
v =  -\dfrac{x{y}}{k(x-y)} = \dfrac{u}{k}x(-y) = \dfrac{u}{k}\bigg(\dfrac{1}{u}- (a+k)\bigg)(a+k) = \dfrac{a+k}{k}\big(1-(a+k)u\big).
\end{equation}
Hence $p_a^\#$ is part of the line  $(a+k)^2{u} +kv = a+k$. $(x,y) \in p_a$. Using equation \eqref{u=f(x,y)}, we have $u= \frac{1}{x-y}: \frac{1}{a+k} \to \frac{1}{a+k+1}$ as $x:0 \to 1$. After plugging the values we found for $u$ in equation \eqref{v=f(u)}, we obtain $v:0 \to \frac{a+k}{k(a+k+1)}$ as $x:0 \to 1$, which is the result. 
\end{proof}
\begin{proof}[Proof of theorem \ref{P_a}.] When $a > 0$, the boundary of $P_a$ consists of the four disjoint line segments $p_a, \{0\} \times (-k-a-1,-k-a], p_{a+1}$ and $\{1\} \times (-k-a-1, -k-a]$. $P_a$ includes the first line segment but not the other three. We use equation \eqref{u=f(x,y)} to write
\begin{equation}\label{vertical}
u = \dfrac{1}{x-y}: \dfrac{1}{x+a+k} \to \dfrac{1}{x+a+k+1} \hspace{2pc} \text{as} \hspace{1pc} y: -k-a \to -k-a-1.
\end{equation}
Also, the same equation yields $y = x - \frac{1}{u}$,  which together with equation \eqref{v=f(x,y)} yields $v= -\frac{u}{k}xy = \frac{x}{k}(1-xu)$. Then
\begin{equation}\label{f_infty}
\{0\} \times (-k-a-1,-k-a] \mapsto \bigg(\frac{1}{a+k+1},\frac{1}{a+k}\bigg] \times \{0\}
\end{equation}
and $\{1\} \times (-k-a-1, -k-a] \mapsto$
\begin{equation}\label{f_0}
(u,v): \bigg(\dfrac{1}{a+k+1}, \dfrac{a+k}{k(a+k+1)}\bigg) \to \bigg(\dfrac{1}{a+k+2},\dfrac{a+k+1}{k(a+k+2)}\bigg),
\end{equation}
which is part of the line $u+k{v}=1$ from $\frac{a+k+1}{k(a+k+2)}$ to $\frac{a+k}{k(a+k+1)}$, including the former point but not the latter. From proposition \ref{a=0,0<k<1}, we know $\Psi$ is injective on $P_a$ where $k$ and $a$ are as in the hypothesis, hence it maps $P_a$ bijectively onto its image $P_a^\#$. Since $\Psi$ is also continuous it must map the interior and boundary of $P_a$ to that of $P_a^\#$. Thus $P_a^\#$ is the region bounded by $p_a^\#, \big(\frac{1}{a+k+1}, \frac{1}{a+k}\big]\times \{0\}, p_{a+1}^\#$ and the line segment $u+k{v}=1, \tab u \in \big[\frac{1}{a+k+1},\frac{1}{a+k+2}\big)$. The region $P_a^\#$ only includes the first line segments of its boundary, which is precisely the desired result. After setting $a=0$ and changing the half open intervals to open intervals, we prove the desired result for $P_0^\#$ as well.  
\end{proof}
For all $k>0$, denote the space $\Psi(\Omega') = \displaystyle{\bigcup_{a\ge0}}P^\#_a$ by $\Gamma'_{(0,k)} = \Gamma'$. We use basic plane geometry and observe that
\begin{corollary}\label{u_v_bound}
$\Gamma'$ is the intersection of the unbounded regions $k{u} + v \le 1,  u + k{v} < 1 , u > 0$ and $v > 0$ in the $u{v}$ plane. The boundary of $\Gamma'$ is the quadrangle in the $(u,v)$ plane with vertices $\big(\frac{1}{k}, 0\big), \big(\frac{1}{k+1},\frac{1}{k+1}\big), \big(0,\frac{1}{k}\big)$ and $(0,0)$. In particular, when $k \ge 1$, we have $u>0, \tab v>0, \tab u +k{v} < 1$ and $k{u} + v < 1$ for all $(u,v) \in \Gamma'$.
\end{corollary}  
Plugging $k=1$ in theorem \ref{P_a}, we see that $P_0^\#$ degenerates to the open region whose boundary is the triangle with vertices $\big(\frac{1}{2},0\big), \big(1,0\big)$ and $\big(\frac{1}{3},\frac{2}{3}\big)$. Thus $\Gamma'_{(0,1)}$ degenerates to the open region whose boundary is triangle with vertices $(1,0), (0,1)$ and $(0,0)$.\\
\begin{corollary}\label{Gamma}
When $k \ge 1$, the space of approximation pairs $\Gamma = \Im(\Psi)$ is a proper subset of $\Gamma'$ and $\Gamma' - \Gamma$ has zero Lebesgue $\IR^2$ measure.  
\end{corollary}
\begin{proof}
$\Gamma = \Psi(\Omega) \subset \Psi(\Omega') = \Psi\bigg(\displaystyle{\bigcup_{a\ge0}}P_a\bigg) = \displaystyle{\bigcup_{a\ge0}}P^\#_a = \Gamma'$. Since $\IQ_{(0,k)}$ is countable, it follows that $\Omega' - \Omega$ has Lebesgue $\IR^2$ measure zero. Then $\Gamma' - \Gamma = \Psi(\Omega') - \Psi(\Omega) \subset \Psi(\Omega' - \Omega)$ has Lebesgue $\IR^2$ measure zero as well.\\ 
\end{proof} 

\newpage

\includegraphics[scale=.55]{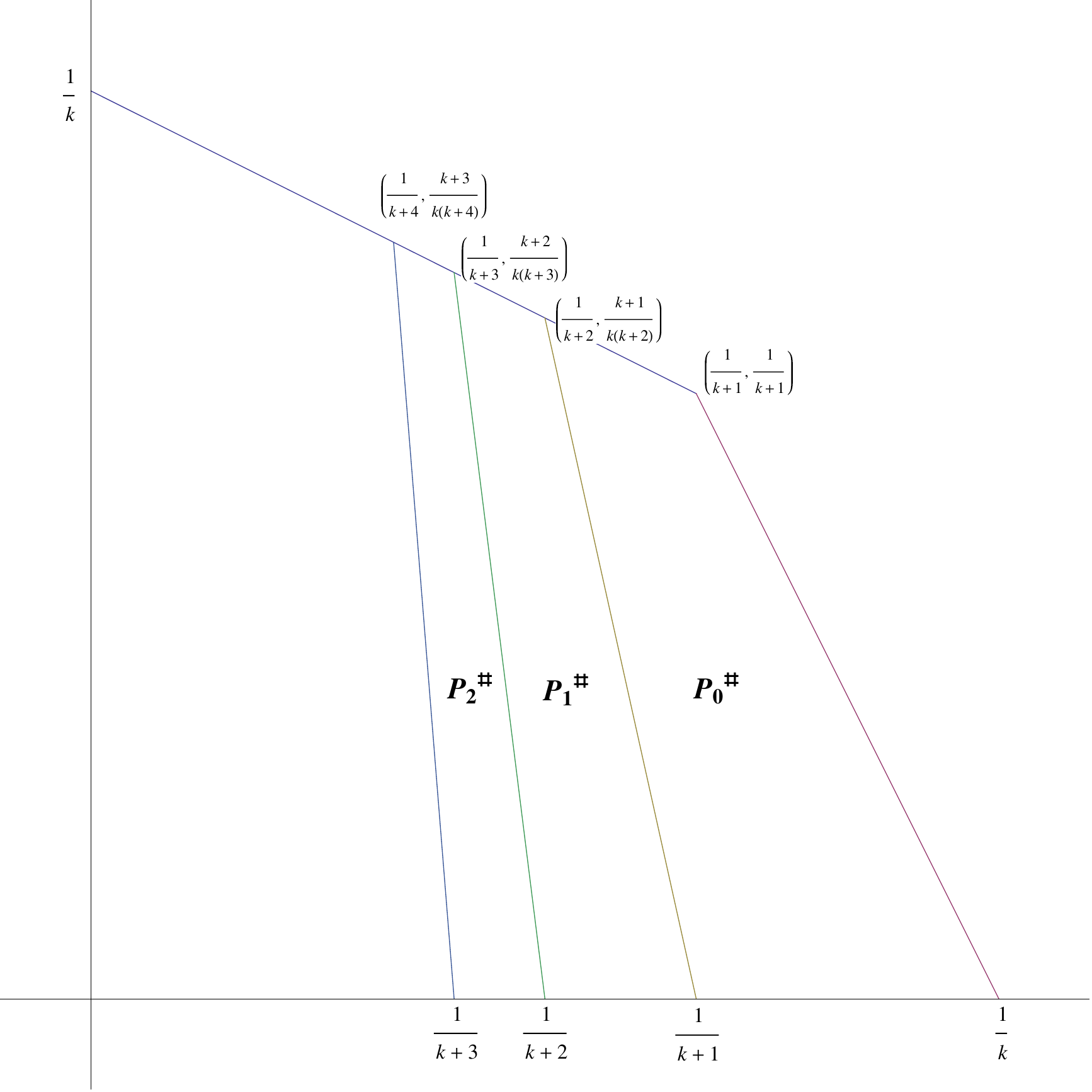}
\vspace{-2pc}
\begin{center}
{\it the region $\Gamma'_{(0,k)}$ when $k > 1$}
\end{center}

\includegraphics[scale=.55]{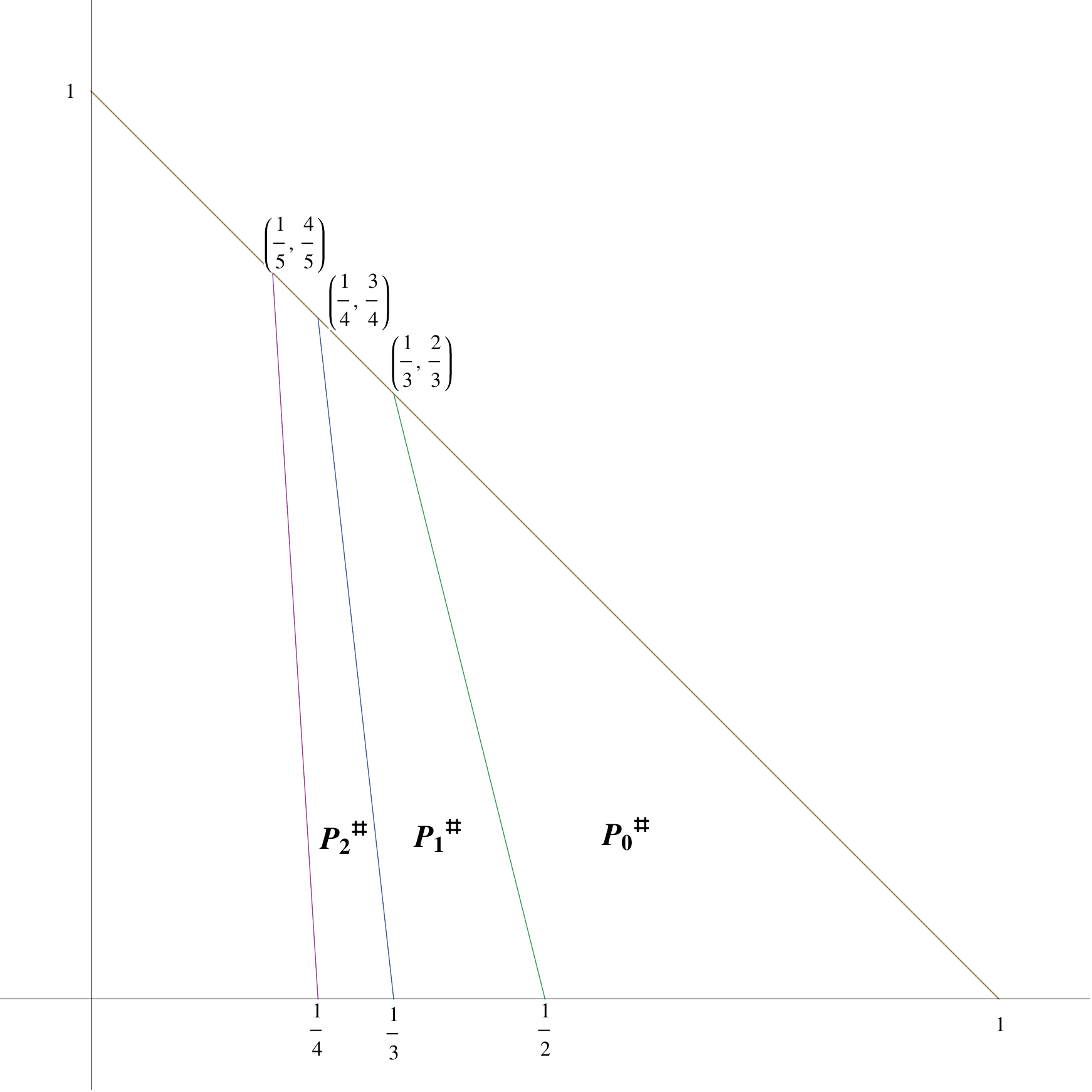}
\vspace{-2pc}
\begin{center}
{\it the region $\Gamma'_{(0,1)}$}
\end{center}
\section{The space of approximation pairs when k is less than one}{}
\noindent When $0<k<1$, proposition \ref{a=0,0<k<1} proves $\Psi$ is not injective on $P_0$. However, this theorem \ref{P_a} does provide the image of $\Omega' - P_0 = \displaystyle{\bigcup_{a \ge 1}}P_a$ under $\Psi$. Thus, in order to finish the characterization of the space of approximation coefficients for these cases we prove:
\begin{theorem}\label{P_0_k<1}
When $0<k<1$, $P_0^\#$ is the intersection of the unbounded regions $u{k}+v < 1, \tab v > 0, \tab  (k+1)^2u+k{v} > k+1, \tab u+k{v} < 1$ and $4k{u}{v} \le 1$. 
\end{theorem}
\begin{proof}
After restricting $P_0$ to its part on or below $x+y=0$, we are left with the region whose boundary is the 5-gon with vertices $(k,-k), (1,-1), (1,-k), (1,-k-1),  (0,-k-1)$ and $(0,-k)$, which includes the part of its perimeter, which is the open line segment connecting the first vertex to the second vertex. When $y=-k$, we use \eqref{u=f(x,y)} and obtain $u=\frac{1}{x-y}:\frac{1}{k} \to \frac{1}{2k}$ as $x:0 \to k$. Plugging $a=0$ to formula \eqref{v=f(u)} yields $(0,k) \times \{-k\} \mapsto (u,v):\big(\frac{1}{2k},\frac{1}{2}\big) \to \big(\frac{1}{k},0\big)$, which is an open segment of the line $uk+v=1$ and plugging $a=0$ to formula \eqref{f_infty} yields $\{0\} \times (-k-1,-k) \mapsto \big(\frac{1}{k+1},\frac{1}{k}\big)\times\{0\}$. Next, we use lemma \ref{p_a} to show $(0,1) \times \{-(k+1)\} \mapsto p_1^\# :(u,v)= \big(\frac{1}{k+1},0\big) \to \big(\frac{1}{k+2},\frac{k+1}{k(k+2)}\big)$, which is the open segment of the line $(k+1)^2u+k{v}=k+1$ between these two points. When $x=1, \tab y:-(k+1) \to -1$, we use formula \eqref{u=f(x,y)} and write $u=\frac{1}{x-y}:\frac{1}{k+2} \to \frac{1}{2}$. Plugging $a=0$ to formula \eqref{v=f(u)} yields $\{1\} \times (-(k+1), -1) \mapsto (u,v):\big(\frac{1}{k+2},\frac{k+1}{k(k+2)}\big) \to \big(\frac{1}{2},\frac{1}{2k}\big)$, which is the open segment of the line $u+k{v}=1$ between these two points. Finally, if $y=-x$ then $u = \frac{1}{x-y}= \frac{1}{2x}$ hence $x = \frac{1}{2u}$ and $v= -\frac{u}{k}x{y} = \frac{1}{4k{u}}$. Thus the line $x+y=0$ is mapped under $\Psi$ to the hyperbola $4k{u}v=1$ in the $u{v}$ plane. The points $(1,-1)$ and $(k,-k)$ map to $\big(\frac{1}{2},\frac{1}{2k}\big)$ and $\big(\frac{1}{2k}, \frac{1}{2}\big)$ under $\Psi$, so that the open segment of the line $x+y=0$ from $(1,-1)$ to $(k,-k)$ maps to the open arc on this hyperbola from $\big(\frac{1}{2},\frac{1}{2k}\big)$ to $\big(\frac{1}{2k}, \frac{1}{2}\big)$. Since $\Psi$ is a continuous bijection on and below the line $x+y=0$, it maps the interior and boundary of this 5-gon bijectively into the interior and boundary of the region in the $u{v}$ plane whose boundary we have just determined and which coincides with the hypothesis. Using proposition \ref{Psi_fold}, we know that the part of $P_0$ which is above the line $x+y=0$ has the same image under $\Psi$ as its reflection about this line. Since this reflection is also contained in $P_0$, we conclude that $P_0$ is mapped in its entirety onto this region and conclude the result.      
\end{proof}
Define the regions in the $u{v}$ plane:
\begin{equation}\label{tilde_Gamma}
\tilde{\Gamma} :=  \Gamma' \cup \big\{(u,v) : \frac{1}{2} \le u \le \frac{1}{2k}, 4k{u}v \le 1\big\}
\end{equation}
\begin{theorem}
When $0<k<1$, $\Gamma$ is a proper subset of $\tilde{\Gamma}$ and $\tilde{\Gamma} - \Gamma$ has zero Lebesgue $\IR^2$ measure.
\end{theorem}
\begin{proof}
Given $0<k<1$, we use the last theorem and theorem \ref{P_a} to verify that $\Psi(\Omega') = \Psi\big(\bigcup_{a\ge0}P_a\big)  = \bigcup_{a\ge0}P^\#_a = \tilde{\Gamma}$. Since $\IQ_{(m,k)}$ is countable, it follows that $\Omega' - \Omega$ has Lebesgue $\IR^2$ measure zero. Then $\Gamma' - \Gamma = \Psi(\Omega') - \Psi(\Omega) \subset \Psi(\Omega' - \Omega)$ has Lebesgue $\IR^2$ measure zero as well. 
\end{proof}
\vfill
\includegraphics[scale=.55]{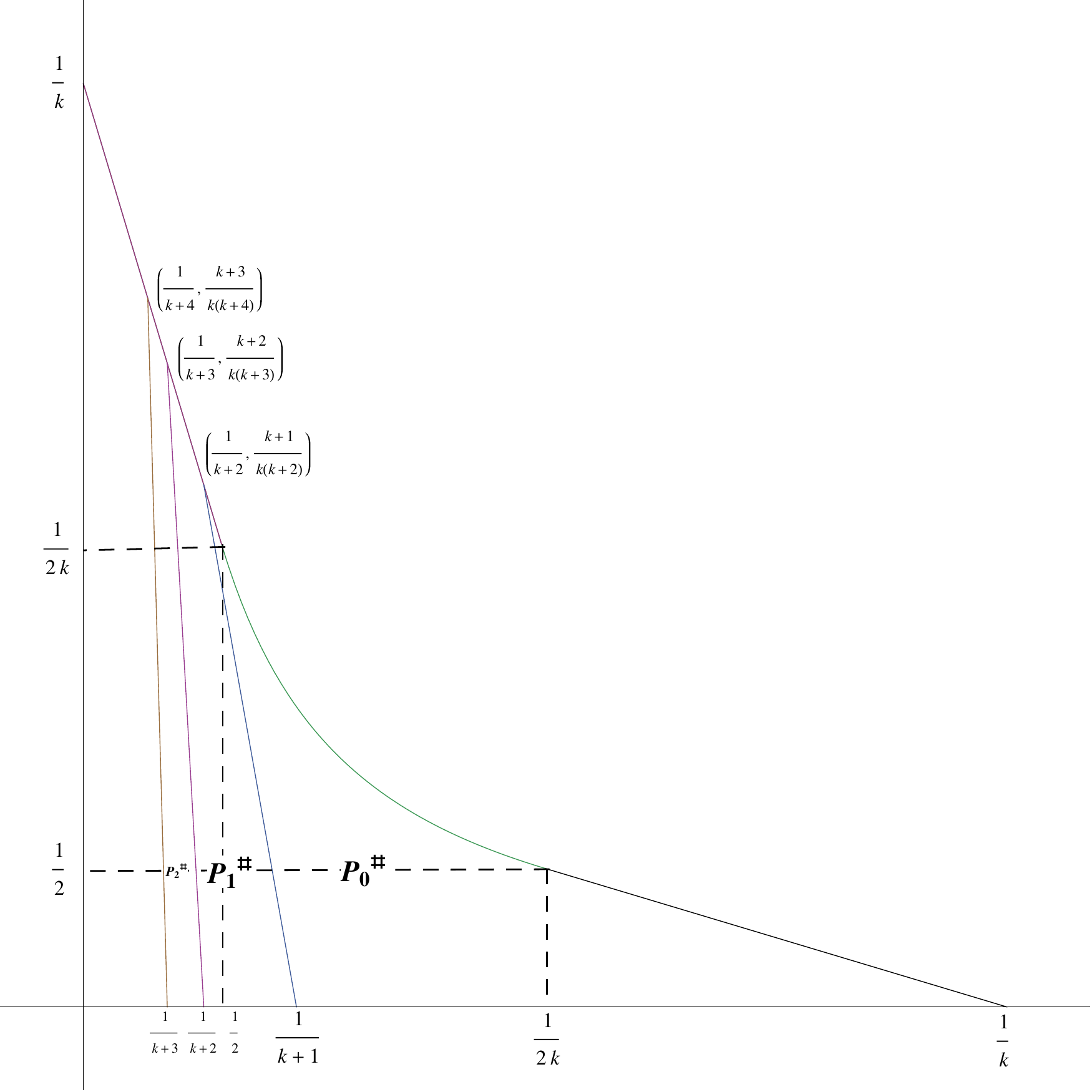}
\vspace{-2pc}
\begin{center}
{\it the region $\tilde{\Gamma}_{(0,k)}$ when $0 < k < 1$}
\end{center}

\section{From approximation pairs to dynamic pairs}{}
\noindent Define
\begin{equation}\label{D}
D_{(0,k)}(u,v) = D(u,v) :=\sqrt{1-4k{u}v}.
\end{equation} 
\begin{theorem}
When $k \ge 1$, the map $\Psi:\Omega' \to \Gamma'$ is a homeomorphism with inverse:
\begin{equation}\label{Psi_inv}
\Psi^{-1}(u,v) := \bigg(\dfrac{1-D(u,v)}{2u}, -\dfrac{1+D(u,v)}{2u}\bigg),  
\end{equation}
\end{theorem}
\begin{proof} 
We already know from proposition \ref{Psi_fold} that $\Psi$ is a continuous injection from $\Omega'$ onto its image $\Gamma'$. It is left to prove that $\Psi^{-1}$ is well defined and continuous on $\Gamma'$ and that it is the inverse from the right for $\Psi$ on $\Gamma'$. Given $(u,v) \in \Gamma'$ set 
\begin{equation}\label{(x,y)=Psi_inv(u,v)}
(x,y) := \Psi^{-1}(u,v).
\end{equation}
From corollary \ref{u_v_bound} we know that $\Gamma'$ lies entirely underneath the line $k{u}+v=1$ in the $u{v}$ plane. The only point of intersection for this line and the hyperbola $4k{u}v=1$ is the point $(u,v) = \big(\frac{1}{2k},\frac{1}{2}\big)$, hence $\Gamma$ must lie underneath this hyperbola as well. We conclude $4k{u}v < 1$ for all $(u,v) \in \Gamma'$, hence $D(u,v) = \sqrt{1-4k{u}v} > 0$ for all $(u,v) \in \Gamma'$. In particular $x$ and $y$ are real. Since $k{u}+v < 1$ from corollary \ref{u_v_bound}, we have $u{k} + v < 1$ so that $-4k{u}v > 4u^2{k^2} - 4{k}u$ and $D(u,v)^2 = 1 - 4k{u}v >  4u^2{k^2} - 4{k}u + 1 = (2{k}u-1)^2$. Conclude $1 + D(u,v) > 2k{u}$ and $y = - \frac{1+D(u,v)}{2u} < -k$. To prove $x \in (0,1)$, we first observe that $D(u,v) = \sqrt{1-4k{u}v} < 1$, so that $1 - D(u,v) > 0$, which proves $x = \frac{1-D(u,v)}{2u} > 0$. If we further assume by contradiction that $x = \frac{1-D(u,v)}{2u} \ge 1$ then 
\[\big(1-D(u,v)\big)\big(1+D(u,v)\big) \ge 2u\big(1+D(u,v)\big)\] 
so that 
\[4k{u}v = 1 - D(u,v)^2 = \big(1-D(u,v)\big)\big(1+D(u,v)\big) \ge 2u\big(1+D(u,v)\big).\] 
This implies $2k{v}-1>D(u,v) \ge 0$ so that $(2k{v}-1)^2 \ge D(u,v)^2$, hence $4k^2v^2 - 4k{v} + 1 \ge 1-4k{v}$ and $u +k{v} \ge 1$, in contradiction to corollary \ref{u_v_bound}. Therefore $x \in (0,1)$ and we conclude that $(x,y) \in \Omega'$ and $\Psi^{-1}:\Gamma' \to \Omega'$ is well defined. Also, $\Psi^{-1}$  is clearly continuous on $\Gamma'$ thus it remains to show that $\Psi\Psi^{-1}(u,v)=(u,v)$.\\ 
From equations \eqref{Psi_inv} and \eqref{(x,y)=Psi_inv(u,v)}, we obtain 
\[\Psi(x,y) = \Psi\Psi^{-1}(u,v) = \Psi \big(\frac{1 - D(u,v)}{2u}, -\frac{1+D(u,v)}{2u}\big).\] 
Using the definition \eqref{Psi} of $\Psi$, the first component of $\Psi(x,y)$ is 
\[\dfrac{1}{x-y} = \bigg(\frac{1 - D(u,v)}{2u} - \big(-\dfrac{1+D(u,v)}{2u}\big)\bigg)^{-1} = \left(\dfrac{2}{2u}\right)^{-1} = u\] and its second component is 
\[-\dfrac{x{y}}{k(x-y)} = -\dfrac{u}{k}x{y} = \bigg(-\dfrac{u}{k}\bigg)\bigg(-\dfrac{1}{4u^2}\bigg)\big(1-D(u,v)^2\big) = \dfrac{1}{4k{u}}\cdot4k{u}v = v.\] T
herefore $(u,v) = \Psi\Psi^{-1}(u,v)$, hence $\Psi^{-1}$ is the right inverse for $\Psi$, as desired.
\end{proof}
The restriction of $\Psi$ to $\Omega$ yields the image $\Gamma = \Psi(\Omega)$ and proves 
\begin{corollary}
When $k \ge 1$, $\Psi:\Omega \to \Gamma$ is a homeomorphism with the same inverse $\Psi^{-1}$ as in formula \eqref{Psi_inv}.
\end{corollary}
\section{The space of approximation pairs revisited}{}\label{SOAPRev}
\noindent When $k \ge 1$, define the regions $F_{(0,k,a)} = F_a := \big(\frac{k}{a+k+1},\frac{k}{a+k}\big] \times (-\infty,-k)$, for $a \ge 1, \tab F_{(0,k,0)} = F_0 := \big(\frac{k}{k+1},1\big) \times (-\infty,-k)$ and  $F_a^\# := \Psi(F_a)$ for all $ a \ge 0$. Then $F_a \cap F_b = \emptyset$ when $a \ne b$ and $\Omega' = \displaystyle{\bigcup_{a\ge0}F_a}$ imply $\Gamma' = \displaystyle{\bigcup_{a\ge0}F_a^\#}$, where these unions are disjoint in pairs. Since $T$ is a homeomorphism from $\Delta^{(1)}_{a}$ onto $(0,1)$, we use the formula for $\ccT$ \eqref{ccT_explicit} and the definition for $P_a$ from section \ref{SOAP} to see that $\ccT(F_a)=P_a$, hence $\Psi\ccT\Psi^{-1}F_a^\#= \Psi\ccT(F_a) = \Psi(P_a) = P_a^\#$ for all $a \ge 0$.
\begin{theorem}\label{F_a}
When $k \ge 1$
\begin{equation}
F_a^\# = \delta{P_a^\#}
\end{equation} 
where $\delta:(u,v) \mapsto (v,u)$ is the reflection about the diagonal $u=v$. Thus when $a >0, \tab F_a^\#$ is the bounded region in the $(u,v)$ which is the intersection of the unbounded regions $k{u} + (a+k)^2v \le a+k, \tab k{u} + (a+k+1)^2v > a+k+1, \tab k{u}+v < 1$ and $u > 0$ and $F_0^\#$ is the intersection of the unbounded regions $u + k{v} < 1, \tab k{u} + (k+1)^2v > k+1, \tab k{u}+v < 1$ and $u > 0$. 
\end{theorem}
To prove this theorem, we first let $f_a$ be the open vertical ray $\big\{\frac{k}{a+k}\big\} \times (-\infty,-k)$ and let $f_a^\#$ be its image under $\Psi$. Then
\begin{lemma}\label{f_a}
When $k \ge 1$, we have $f_a^\# = \delta(p_a^\#)$.
\end{lemma}
\begin{proof}
When $x = \frac{k}{a+k}$ and $y:-k \to -\infty$ we use formula \eqref{u=f(x,y)} to obtain 
\[u = \dfrac{1}{x-y} = \dfrac{a+k}{k-(a+k)y}: \dfrac{a+k}{k(a+k+1)} \to 0.\] 
Solving for $y$ in terms of $x$ and $u$ yields $y = x - \frac{1}{u} = \frac{k}{a+k} - \dfrac{1}{u}$. As $u: \frac{a+k}{k(a+k+1)} \to 0$, we have 
\[v = -\dfrac{u}{k}x{y} = -\dfrac{u}{k}\bigg(\dfrac{k}{a+k}\bigg)\bigg(\dfrac{k}{a+k} - \dfrac{1}{u}\bigg) = \dfrac{1}{a+k}\bigg(1 - \dfrac{k}{a+k}u\bigg):\dfrac{1}{a+k+1} \to \dfrac{1}{a+k}.\] 
Therefore $f_a^\#$ is the open segment of the line $k{u} + (a+k)^2v = a+k$ from $\big(\frac{a+k}{k(a+k+1)}, \frac{1}{a+k+1}\big)$ to $\big(0,\frac{1}{a+k}\big)$. Comparing $f_a$ with $p_a$ from the proof of lemma \ref{p_a}, we see that $f_a^\#$ is indeed the reflection of $p_a^\#$ along the diagonal $u=v$ in the $u{v}$ plane. 
\end{proof}
\begin{proof}[Proof of theorem \ref{F_a}.]
The boundary of $F_a$ consists of $f_a, \tab f_{a+1}$ and the part of $p_0$ which connects the two rays. From theorem \ref{P_a}, we know that the boundary of $P_a$ consists of $p_a, \tab p_{a+1}$ and the part of $f_0$ which connects these line segments. Therefore, we see that it is enough to show that $f_a^\#$ is the reflection of $p_a^\#$ along the diagonal $u=v$ and conclude the result from the lemma. 
\end{proof}
\vspace{1pc}
\includegraphics[scale=.55]{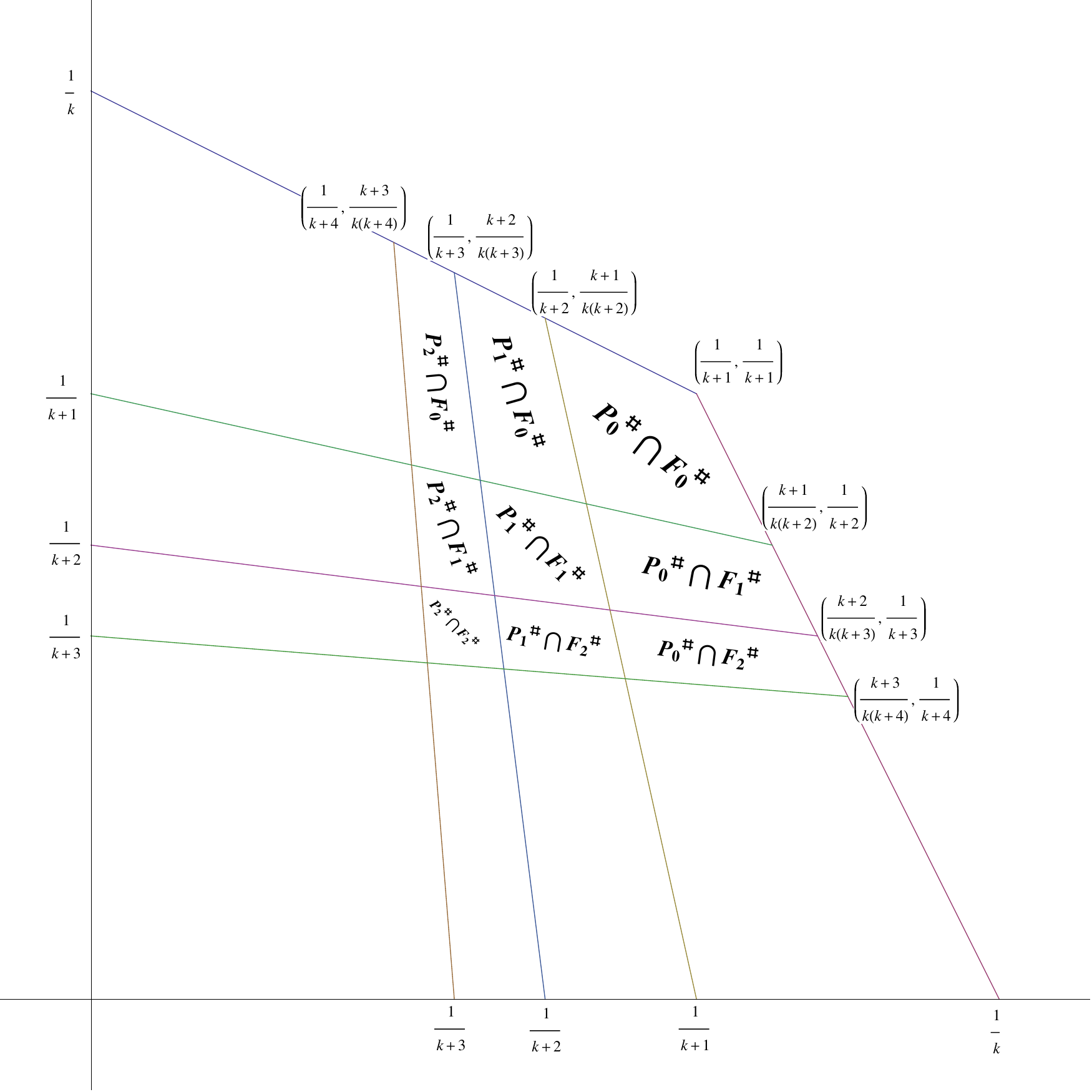}
\vspace{-2pc}
\begin{center}
{\it The region $\Gamma'_{(0,k)}$ revisited}
\end{center}
\newpage
\section{The difference bisequence of approximation coefficients}{}
\noindent In this section we find bounds for the \bf{difference bisequence of approximation coefficients} \[\big\{\eta_n(x_0,y_0)\big\}_{-\infty}^\infty := \big\{\big|\theta_n(x_0,y_0) - \theta_{n-1}(x_0,y_0)\big|\big\}_{-\infty}^\infty.\]
\begin{lemma}\label{P_a_cap_F_b}
When $k \ge 1$ the region $P_a^\# \cap F_b^\#$ is quadrangle with vertices
\[\bigg(\dfrac{b+k}{k+(a+k)(b+k)}, \dfrac{a+k}{k + (a+k)(b+k)}\bigg), \tab \bigg(\dfrac{b+k}{k+(a+k+1)(b+k)},\dfrac{a+k+1}{k + (a+k+1)(b+k)}\bigg),\]
\[\bigg(\dfrac{b+k+1}{k+(a+k)(b+k+1)},\dfrac{a+k}{k + (a+k)(b+k+1)}\bigg)\]
and
\[\bigg(\dfrac{b+k+1}{k+(a+k+1)(b+k+1)},\dfrac{a+k+1}{k + (a+k+1)(b+k+1)}\bigg)\]
excluding two of its boundary line segments.
\end{lemma}
\begin{proof} 
Apply $\Psi(x,y) := \big(\frac{1}{x-y}, -\frac{x{y}}{k(x-y)}\big)$ to $p_a \cap f_b = \big(\frac{k}{b+k}, -(a+k)\big)$ and obtain
\begin{equation}\label{p_a_cap_f_b}
p_a^\# \cap f_b^\# = \bigg( \dfrac{b+k}{k + (a+k)(b+k)}, \dfrac{a+k}{k + (a+k)(b+k)}\bigg).
\end{equation}
Result now follows since $P_a \cap F_b \subset \Omega$ is the region interior to the quadrangle whose sides are $p_a,f_b,p_{a+1}$ and $f_{b+1}$ including the first and second line segments but not the third and fourth.
\end{proof}
\begin{theorem}
Given $(x_0,y_0) \in \Omega$, we write $a_n := a_n(x_0,y_0)$ for all $n \in \mathbb{Z}$ and fix $N \in \mathbb{Z}$. If $a := \min\{a_N,a_{N+1}\}, \tab A :=  \max\{a_N,a_{N+1}\}, \tab b :=  \min\{a_{N+1},a_{N+2}\}$ and $B :=  \max\{a_{N+1},a_{N+2}\}$ then
\[\eta_N(x_0,y_0)^2 + \eta_{N+1}(x_0,y_0)^2 \le  \bigg(\dfrac{B+k+1}{k + (a+k)(B+k+1)} - \dfrac{b+k}{k + (A+k+1)(b+k)}\bigg)^2\]
\[ + \bigg(\dfrac{A+k+1}{k + (A+k+1)(b+k)}- \dfrac{a+k}{k + (a+k)(B+k+1)}\bigg)^2.\] 
\end{theorem}
\begin{proof}
The greatest possible distance between two points in a subset $X$ of the Euclidean metric space $(\mathbb{E}^2,\operatorname{d})$ is the diameter of $X$, $\operatorname{diam}(X) := \displaystyle{\sup_{x,y \in X}\big(\operatorname{d}(x,y)\big)}$. When $X$ is a convex quadrangle with opposite acute angles, its diameter equals the length of the diagonal connecting the vertices for these acute angles. Let $R_{(0,k,a,b,A,B)} = R$ be the closure of $\displaystyle{\bigcup_{a \le i \le A, \tab b \le j \le B}(P_i^\# \cap F_j^\#)}$. From theorem \ref{P_a_cap_F_b}, we see that $R$ is a convex quadrangle, whose boundary consists of sections from the segments $p_a^\#, \tab p_{A+1}^\#, \tab f_b^\#$ and $f_{B+1}^\#$. From the calculations of \eqref{P_a} and \eqref{F_a}, we see that the slopes of $p_a^\#,p_{A+1}^\#$ lie in $(-\infty,-1)$ and the slope of $f_b^\#,f_{B+1}^\#$ lie in $(-1,0)$. Conclude that $R$ has opposite acute angles and that its diameter is the length of the diagonal connecting the vertices $ p_a^\# \cap f_{B+1}^\#$ and $p_{A+1}^\#\ \cap f_{b}^\#$. From equation \eqref{p_a_cap_f_b} and the distance formula on $\mathbb{E}^2$ we have 
\[\operatorname{diam(R)}^2 = \operatorname{d}(p_a^\# \cap f_{B+1}^\#, p_{A+1}^\# \cap f_b^\#)^2 = \bigg(\dfrac{A+k+1}{k + (A+k+1)(b+k)} - \dfrac{a+k}{k + (a+k)(B+k+1)}\bigg)^2\]
\[ + \tab  \bigg(\dfrac{B+k+1}{k + (a+k)(B+k+1)} - \dfrac{b+k}{k + (A+k+1)(b+k)}\bigg)^2.\]
Since $(\theta_{N-1},\theta_n) \in P_{a_N}^\# \cap F_{a_{N+1}}^\# \subset R$ and $(\theta_N,\theta_{N+1}) \in P_{a_{N+1}}^\# \cap F_{a_{N+2}}^\# \subset R$, we have $\eta_N^2 + \eta_{N+1}^2 = (\theta_N - \theta_{N-1})^2 +  (\theta_{N+1} - \theta_N)^2 \le \operatorname{diam(R)}^2$ and conclude the result. 
\end{proof}
\begin{corollary}
Given $(x_0,y_0) \in \Omega$, write $a_n := a_n(x_0,y_0)$ for all $n \in \mathbb{Z}$. Fixing $N \in \mathbb{Z}$, we let $l := \min\{a_N,a_{N+1},a_{N+2}\}$ and $L :=  \max\{a_N,a_{N+1},a_{N+2}\}$. Then
\[\eta_N^2 + \eta_{N+1}^2 \le 2\tab\bigg(\dfrac{L-l+1}{k + (l+k)(L+k+1)}\bigg)^2.\]
\end{corollary}
\begin{proof}
We let $a,b,A$ and $B$ be as in the statement of the previous theorem. Since $l \le a \le A \le L$ and $l \le b \le B \le L$, we clearly have
\[\dfrac{A+k+1}{k + (b+k)(A+k+1)} \le \dfrac{L+k+1}{k + (b+k)(L+k+1)} \le \dfrac{L+k+1}{k + (l+k)(L+k+1)}\] 
and
\[\dfrac{a+k}{k + (B+k+1)(a+k)} \ge \dfrac{l+k}{k + (B+k+1)(l+k)} \ge \dfrac{l+k}{k + (L+k+1)(l+k)},\]
so that 
\[0 \le \dfrac{A+k+1}{k + (A+k+1)(b+k)} - \dfrac{a+k}{k + (a+k)(B+k+1)}\] 
\[\le \dfrac{L+k+1}{k + (L+k+1)(l+k)} - \dfrac{l+k}{k + (L+k+1)(l+k)} = \dfrac{L-l+1}{k + (L+k+1)(l+k)}.\] 
This argument remains identical after we exchange $a$ for $b$ and $A$ for $B$. The result now follows directly from the theorem.
\end{proof}
If we further assume that $l=L$ then this corollary implies that 
\[\eta_N^2 < \eta_N^2 + \eta_{N+1}^2 \le 2\tab\bigg(\dfrac{1}{k + (l+k)(l+k+1)}\bigg)^2.\] 
After taking the square root and plugging $k=1$, we conclude that
\begin{corollary}
If $b_n = b_{n+1}$ for some $n \ge 1$ in the RCF expansion for $x_0 \in (0,1) - \IQ$, then 
\[\abs{\theta_n-\theta_{n-1}} < \dfrac{\sqrt{2}}{b^2+3b+4}.\]
In particular, if $b_n = b_{n+1} =1$ then
\[\abs{\theta_n-\theta_{n-1}} < \dfrac{1}{4\sqrt{2}}.\]
\end{corollary} 
  \chapter{BAC for Renyi-like maps}{}
In this chapter, we improve the results of Haas and Molnar in \cite{HM,Molnar}. The new results start from section \ref{SOAPRevR} and the material leading to this section is given for the sake of completeness. Since we are only interested in the Renyi case $m=1$, we fix $k > 1$ and, in order to ease the notation, omit the subscript $\square_{(1,k)}$ throughout. 
\section{From dynamic pairs to approximation pairs}{}
\noindent We begin our investigation of the bi-sequence of approximation coefficients for $\ccT$ by defining the map $\Psi_{(1,k)} = \Psi: \{(x,y) \in \IR^2 : x - y \ne 0\}\to \IR^2$ by
\begin{equation}\label{Psi_R}
(u,v) = \Psi(x,y) := \bigg(\dfrac{1}{x-y}, \dfrac{(1-x)(1-y)}{k(x-y)}\bigg).
\end{equation}
Clearly, $\Psi$ is well defined and continuous on the region $\{(x,y) \in \IR^2 : x - y < 0\}$, hence it is well defined and continuous on its subset $\Omega' = [0,1) \times (-\infty,1-k]$.
\begin{proposition}\label{Psi_fold_R}
The image of any point under $\Psi$ is identical to its reflection about the line $x+y=2$ and is injective on the region $\{(x,y) \in \IR^2 : x + y < 2, \tab x \ne y\}$. In particular, $\Psi$ is injective on $\Omega'$.
\end{proposition}
\begin{proof}
\[\Psi(x, y)= \bigg(\dfrac{1}{x-y},\dfrac{(1-x)(1-y)}{k(x-y)}\bigg) = \bigg(\dfrac{1}{x-y},\dfrac{(x-1)(y-1)}{k(x-y)}\bigg)\]
\[ = \bigg(\dfrac{1}{(2-y)-(2-x)}, \dfrac{\big(1 - (2-y)\big)\big(1 - (2-x)\big)}{k\big((2-y) - (2-x)\big)}\bigg) = \Psi(2-y,2-x)\] 
proves $\Psi$ is invariant under reflection about the line $x+y=2$. Assume both the points $(x_1,y_1)$ and $(x_2,y_2)$ lie on or below this line and that both points have the same image under $\Psi$. Write $(u_1,v_1) =  \Psi(x_1,y_1)=  \Psi(x_2,y_2) = (u_2,v_2)$. Then $\frac{1}{x_1-y_1} = u_1 = u_2 = \frac{1}{x_2-y_2}$ implies
\begin{equation}\label{x-y_R}
x_1 - y_1 = x_2 - y_2.
\end{equation}
The equality 
\[\dfrac{u_1}{k}(1-x_1)(1-y_1) = v_1 = v_2 = \dfrac{u_2}{k}(1-x_2)(1-y_2)\] 
implies that $(1-x_1)(1-y_1) = (1-x_2)(1-y_2)$. Together with the basic algebraic equality $4\alpha\beta = (\alpha+\beta)^2 - (\alpha-\beta)^2$, using $\alpha=1-x_1,\tab \beta=1-y_1$, we have
\[\big(2 - (x_1+y_1)\big)^2 - (x_1-y_1)^2 = 4(1-x_1)(1-y_1) = 4(1-x_2)(1-y_2) = \big(2 - (x_2+y_2)\big)^2 -  (x_2-y_2)^2.\] 
Using equation \eqref{x-y_R}, this last equality is reduced to $\big(2 - (x_1+y_1)\big)^2 = \big(2 - (x_2+y_2)\big)^2$. Since we are assuming $x_1+y_1 \le 2$ and $x_2+y_2 \le 2$, we derive the equality $2-(x_1+y_1) = 2-(x_2+y_2)$ and conclude that $x_1+y_1 = x_2+y_2$. This result and another application of \eqref{x-y_R} prove that 
\[2x_1 = (x_1+y_1) + (x_1-y_1) = (x_2+y_2) + (x_2-y_2) = 2x_2\] 
so that $x_1=x_2$ and then that $y_1=y_2$ as well. Therefore, $\Psi$ is injective on or below that line $x+y=2$.
\end{proof}
From the formulas \eqref{theta_dynamic} and \eqref{ccT_explicit_R}, we obtain 
\[\dfrac{1}{\theta_n} = x_{n+1} - y_{n+1} = \bigg(\dfrac{kx_n}{1-x_n} - a_{n+1}\bigg)- \bigg(\dfrac{ky_n}{1-y_n} - a_{n+1} \bigg) = \dfrac{k(x_n-y_n)}{(1-x_n)({1-y_n)}},\] 
so that
\begin{equation}\label{Psi_theta_R} 
\Psi(x_n,y_n) = (\theta_{n-1}, \theta_n)
\end{equation}
is the approximation pair for $(x_0,y_0)$ at time $n$.
\section{The space of approximation pairs}{}\label{SOAP_R}
\noindent Our goal in this section is to determine the finer structure of the space of approximation pairs $\Gamma_{(1,k)} = \Gamma := \operatorname{Im}(\Psi) =\Psi(\Omega) \subset \IR^2$. We begin by defining the regions $P_{(1,k,a)} = P_a := (0,1) \times (-a-k,1-a-k] \subset \Omega'$
 when $a >0$, $P_{(1,k,0)} = P_0 := (0,1) \times (-k,1-k)$ and $P_{(1,k,a)}^\# = P_a^\# := \Psi(P_a) \subset \IR^2$ for all $a \ge 0$. Then $P_a \cap P_b = \emptyset$ whenever $a \ne b$ and $\Omega' = \displaystyle{\bigcup_{a\ge0}P_a}$.
\begin{theorem}\label{P_a_R}
$P_a^\#$ is the region in the $(u,v)$ which is the intersection of the unbounded regions $(a+k)^2{u}-k{v} \le a+k, \tab (a+k+1)^2u -k{v} > a+k+1, \tab k{v} -u < 1$ and $v > 0$. $P_0^\#$ is the open quadrangle in the $(u,v)$ plane, which is the intersection of the unbounded regions $k{u}- v < 1, \tab (k+1)^2{u} -k{v} > k+1, \tab k{v} -u < 1$ and $v > 0$.  
\end{theorem}
To prove this theorem, we first let $p_{(1,k,a)} := p_a$ be the open horizontal line segment $(0,1) \times \{1-k-a\}$ and let $p_a^\#$ be its image under $\Psi$.
\begin{lemma}\label{p_a_R}
$p_a^\#$ is the open line segment $(a+k)^2{u} -kv = a+k$ between $( \frac{1}{a+k},0)$ and $(\frac{1}{a+k-1},\frac{a+k}{k(a+k-1)})$.
\end{lemma}
\begin{proof}
Given $(x,y) \in \Omega$, we write $(u,v) = \Psi(x,y) \in \Gamma$. From the definition of $\Psi$, \eqref{Psi_R}, we have
\begin{equation}\label{u=f(x,y)_R}
u=\dfrac{1}{x-y}
\end{equation}
and
\begin{equation}\label{v=f(x,y)_R}
v=\dfrac{(1-x)(1-y)}{k(x-y)} = \dfrac{u}{k}(1-x)(1-y). 
\end{equation}
$(x,y) \in p_a$ implies $y=1-k-a$. Using equation \eqref{u=f(x,y)_R}, we write $x = \frac{1}{u} + y = \frac{1}{u} + 1 - k - a$ and $v = \frac{(1-x)(1-y)}{k(x-y)} = \frac{u}{k}(1-x)(1-y) = \frac{u}{k}\big((a+k) - \frac{1}{u}\big)(a+k)$. Conclude
\begin{equation}\label{v=f(u)_R}
v = \dfrac{(a+k)^2}{k}u - \dfrac{a+k}{k}.
\end{equation}
Hence $p_a^\#$ is part of the line  $(a+k)^2{u} - kv = a+k$. When $(x,y) \in p_a$, we have $y=1-k-a$ so that, after using equation \eqref{u=f(x,y)_R}, we obtain $u= \frac{1}{x-y}: \frac{1}{a+k-1} \to \frac{1}{a+k}$ as $x:0 \to 1$. After plugging the values we found for $u$ in formula \eqref{v=f(u)_R}, we conclude that $v: \frac{a+k}{k(a+k-1)} \to 0$ as $x:0 \to 1$, yielding the desired result.
\end{proof}
\begin{proof}[Proof of theorem \ref{P_a_R}] When $a > 0$, the boundary of $P_a$ consists of the four disjoint line segments: $p_a, \tab  \{1\} \times (-(a+k), 1-(a+k)], \tab p_{a+1}$ and $\{0\} \times (-(a+k),1-(a+k)]$. $P_a$ includes the first line segment but not the other three. We use equation \eqref{u=f(x,y)_R} to write 
\begin{equation}\label{vertical_R}
u = \dfrac{1}{x-y} = \dfrac{1}{x+a+k} \to \dfrac{1}{x+a+k-1} \hspace{1pc} \text{as $y: -k-a \to 1-k-a$}.
\end{equation}
Also, the same equation yields $y =  x - \frac{1}{u}$, which together with equation \eqref{v=f(x,y)_R} imply $v= \frac{u}{k}(1-x)(1-y) = \frac{u}{k}(1-x)^2 + \frac{1-x}{k}$. Conclude 
\begin{equation}\label{f_infty_R}
\{1\} \times \big(-(a+k),1-(a+k)\big] \mapsto \bigg[\frac{1}{a+k+1},\frac{1}{a+k}\bigg) \times \{0\}
\end{equation}
and $\{0\} \times \big(-(a+k), 1-(a+k)\big] \mapsto$
\begin{equation}\label{f_0_R}
(u,v): \bigg(\dfrac{1}{a+k}, \dfrac{a+k+1}{k(a+k)}\bigg) \to \bigg(\frac{1}{a+k-1},\frac{a+k}{k(a+k-1)}\bigg),
\end{equation}
which is the segment of the line $kv-u=1$ from $(\frac{1}{a+k-1},\frac{a+k}{k(a+k-1)})$ to $(\frac{1}{a+k},\frac{a+k+1}{k(a+k)})$  including the former point but not the latter point. From proposition \ref{Psi_fold_R}, we know $\Psi$ is injective on $\Gamma$, hence $\Psi$ maps $P_a$ bijectively onto its image $P_a^\#$. Since $\Psi$ is also continuous it must map the interior of $P_a$ to the interior $P_a^\#$ and the boundary of $P_a$ to the boundary $P_a^\#$. Thus $P_a^\#$ is the region bounded by $p_a^\#$, $\big(\frac{1}{a+k+1}, \tab \frac{1}{a+k}\big] \times \{0\}, \tab p_{a+1}^\#$ and the line segment $k{v}-u=1, \tab u \in [\frac{1}{a+k},\frac{1}{a+k-1})$. $P_a^\#$ only includes the first line segments of its boundary, which is precisely the desired result. After setting $a=0$ and changing the half open intervals to open intervals, we prove the desired result for $P_0^\#$ as well.  
\end{proof}
For all $k>1$, we denote the space $\Psi(\Omega') = \displaystyle{\bigcup_{a\ge0}}P^\#_a$ by $\Gamma'_{(1,k)} = \Gamma'$. We use basic plane geometry and observe the following:
\begin{corollary}\label{u_v_bound_R} 
$\Gamma'$ is the intersection of the unbounded regions $k{u} - v < 1,  u - k{v} < 1 , u > 0$ and $v > 0$ in the $u{v}$ plane. The boundary of $\Gamma'$ is the quadrangle in the $(u,v)$ plane with vertices $(\frac{1}{k}, 0), (\frac{1}{k-1},\frac{1}{k-1}), (0,\frac{1}{k})$ and $(0,0)$. In particular, for all $(u,v) \in \Gamma'$, we have $\max\{k{v} - u,  k{u} - v \} < 1$.   
\end{corollary}
\begin{corollary}\label{Gamma}
$\Gamma$ is a proper subset of $\Gamma'$ and $\Gamma' - \Gamma$ has zero Lebesgue $\IR^2$ measure.  
\end{corollary}
\begin{proof}
$\Gamma = \Psi(\Omega) \subset \Psi(\Omega') = \Psi\bigg(\displaystyle{\bigcup_{a\ge0}}P_a\bigg) = \displaystyle{\bigcup_{a\ge0}}P^\#_a = \Gamma'$. Since $\IQ_{(1,k)}$ is countable, it follows that $\Omega' - \Omega$ has Lebesgue $\IR^2$ measure zero. Then $\Gamma' - \Gamma = \Psi(\Omega') - \Psi(\Omega) \subset \Psi(\Omega' - \Omega)$ has Lebesgue $\IR^2$ measure zero as well. 
\end{proof} 
The following proposition is a direct consequence of the characterization of the region $\Gamma'$.
\begin{proposition}\label{uniform_bound_R}
$C_0 = \frac{1}{k-1}$ is a uniform upper bound for the bi-sequence of approximation coefficients $\{\theta_n(x_0,y_0)\}_{-\infty}^\infty$ for all $(x_0,y_0) \in \Omega$.
\end{proposition}
\vspace{1pc}
\includegraphics[scale=.55]{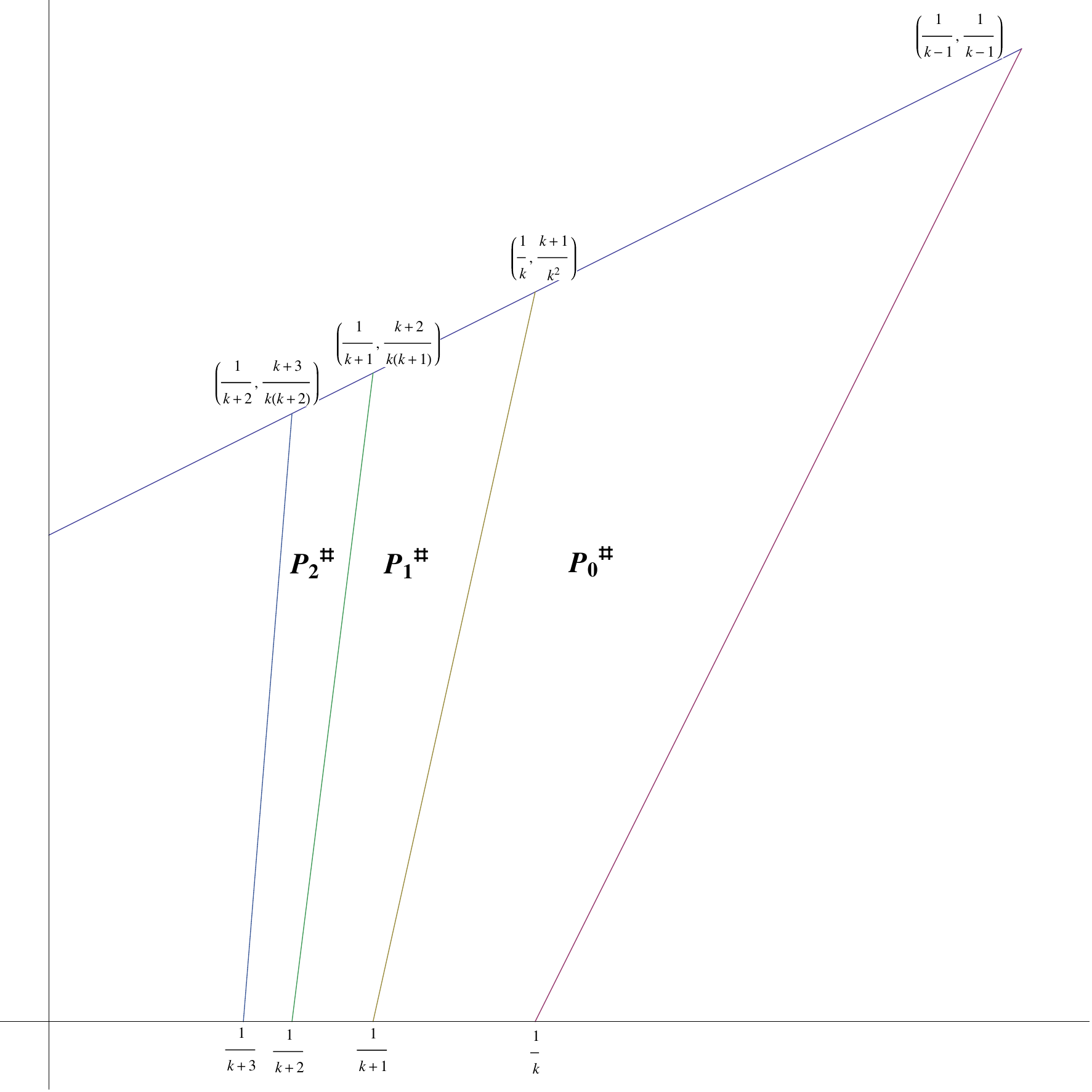}
\begin{center} {\it The region $\Gamma'_{(1,k)}$} \end{center}
\section{From approximation pairs to dynamic pairs}{}
\noindent Define
\begin{equation}\label{D_R}
D_{(1,k)}(u,v) = D(u,v) :=\sqrt{1+4kuv}.
\end{equation} 
\begin{theorem}
$\Psi:\Omega' \to \Gamma'$ is a homeomorphism with inverse
\begin{equation}\label{Psi_inv_R}
\Psi^{-1}(u,v) = \bigg(1 + \dfrac{1-D(u,v)}{2u}, 1 -\dfrac{1+D(u,v)}{2u}\bigg),
\end{equation}
\end{theorem}
\begin{proof}
We already know from proposition \ref{Psi_fold_R} that $\Psi$ is a continuous injection from $\Omega'$ onto its image $\Gamma'$. It is left to prove that $\Psi^{-1}$ is well defined and continuous on $\Gamma'$ and that it is the inverse from the right for $\Psi$ on $\Gamma'$. Given $(u,v) \in \Gamma'$, set 
\begin{equation}\label{(x,y)=Psi_inv(u,v)_R}
(x,y) := \Psi^{-1}(u,v).
\end{equation}
Clearly $D(u,v)$ is real, hence $x$ and $y$ are real. From corollary \ref{u_v_bound_R}, we have $D(u,v) = \sqrt{1+4k{u}v} > 1$ so that $1 - D(u,v) < 0$ and $x = 1 + \frac{1 - D(u,v)}{2u} < 1$. From the same corollary, we have $k{v}- u < 1$ hence $D(u,v)^2 = 1 + 4u{k}v < 1 + 4u +4u^2 = (2u+1)^2$, so that $D(u,v) < 2u +1$. Then $\frac{1 - D(u,v)}{2u} > -1$ and $x= 1 + \frac{1-D(u,v)}{2u} > 0$. Also $1+4k{u}v > 4u^2k^2 - 4k{u} +1$, which implies $\sqrt{1+4k{u}v} > 2k{u} - 1$, so that $\frac{1+\sqrt{1+4kuv}}{2u} = 1-y > k$. Then $y=1-\frac{1+D(u,v)}{2u} < 1-k$. Conclude that $(x,y) \in \Omega'$ hence $\Psi^{-1}$ is well defined. Also $\Psi^{-1}$ is clearly continuous on $\Gamma'$ thus it remains to show that $\Psi\Psi^{-1}(u,v)=(u,v)$.\\
From equation \eqref{Psi_inv_R} and \eqref{(x,y)=Psi_inv(u,v)_R}, we obtain 
\[\Psi(x,y) = \Psi\Psi^{-1}(u,v) = \Psi \bigg(1 + \frac{1 - D(u,v)}{2u}, 1-\frac{1+D(u,v)}{2u}\bigg).\] 
Using the definition \eqref{Psi_R} of $\Psi$, the first component of $\Psi(x,y)$ is 
\[\dfrac{1}{x-y} = \bigg(1 + \frac{1 - D(u,v)}{2u} - \big(1-\dfrac{1+D(u,v)}{2u}\big)\bigg)^{-1} = \big(\dfrac{2}{2u}\big)^{-1} = u\] 
and its second component is 
\[\dfrac{(1-x)(1-y)}{k(x-y)} = \dfrac{u}{k}(1-x)(1-y) = \dfrac{u}{k}\bigg(\dfrac{D(u,v)-1}{2u}\bigg)\bigg(\dfrac{D(u,v)+1}{2u}\bigg)\]
\[ =\dfrac{u}{k}\cdot\frac{D(u,v)^2-1}{4u^2} = \dfrac{1}{4uk}\bigg((1+4k{u}v) - 1\bigg) = v.\] 
Therefore, $(u,v) = \Psi\Psi^{-1}(u,v)$. Hence, $\Psi^{-1}$ is the right inverse for $\Psi$ as desired.
\end{proof}
The restriction of $\Psi$ to $\Omega$ yields the image $\Gamma = \Psi(\Omega)$ and proves 
\begin{corollary}
For all $k>1$, $\Psi:\Omega \to \Gamma$ is a homeomorphism with the same inverse $\Psi^{-1}$ as in formula \eqref{Psi_inv_R}.
\end{corollary}
\section{The space of approximation pairs revisited}{}\label{SOAPRevR}
\noindent Define the regions $F_{(1,k,a)} = F_a := \big(\frac{a}{a+k},\frac{a+1}{a+k+1}\big) \times (-\infty,1-k)$ when $a > 0$, $F_{(1,k,0)} = F_0 := (0,\frac{1}{k+1}) \times (-\infty,1-k)$ and  $F_a^\# := \Psi(F_a)$ for all $ a \ge 0$. Then $F_a \cap F_b = \emptyset$ when $a \ne b$ and $\Omega' = \displaystyle{\bigcup_{a\ge0}F_a}$ imply $\Gamma' = \displaystyle{\bigcup_{a\ge0}F_a^\#}$, where these unions are disjoint in pairs. Since $T$ is a homeomorphism from $\Delta^{(1)}_{a}$ onto $(0,1)$, we use the formula for $\ccT$ \eqref{ccT_explicit_R} and the definition for $P_a$ from section \ref{SOAP_R} to see that $\ccT(F_a)=P_a$, hence $\Psi\ccT\Psi^{-1}F_a^\#= \Psi\ccT(F_a) = \Psi(P_a) = P_a^\#$ for all $a \ge 0$.
\begin{theorem}\label{F_a_R}
\begin{equation}
F_a^\# = \delta{P_a^\#}
\end{equation} 
where $\delta:(u,v) \mapsto (v,u)$ is the reflection about the diagonal $u=v$. Thus when $a >0, \tab F_a^\#$ is the bounded region in the $(u,v)$ which is the intersection of the unbounded regions $(a+k)^2{v}-ku \le a+k, \tab (a+k+1)^2v -k{u} > a+k+1, \tab ku - v < 1$ and $u > 0$ and $F_0^\#$ is the intersection of the unbounded regions $k{v} - u < k, \tab (k+1)^2v - k{u} > k+1, \tab k{u} - v < 1$ and $u > 0$. 
\end{theorem}
To prove this theorem, we first let $f_a$ be the open vertical ray $\big\{\dfrac{a}{a+k}\big\} \times (-\infty, 1-k)$ and let $f_a^\#$ be its image under $\Psi$. Then
\begin{lemma}\label{f_a_R}
$f_a^\#=\delta{p_a^\#}$.
\end{lemma}
\begin{proof}
When $x = \dfrac{a}{a+k}$ and $y:1-k \to -\infty$, we have
\[u = \dfrac{1}{x-y} = \dfrac{a+k}{a-(a+k)y}: \dfrac{a+k}{k(a+k-1)} \to 0.\]
Solving for $y$ in terms of $x$ and $u$ yields 
\[y = x - \dfrac{1}{u} = \dfrac{a}{a+k} - \dfrac{1}{u}\]
and
\[v = \dfrac{(1-x)(1-y)}{k(x-y)} = \dfrac{u}{k}(1-x)(1-y) =  \dfrac{u}{k}(1-x)\big(1-x + \dfrac{1}{u}\big) = \dfrac{u}{k}\cdot\dfrac{k}{a+k}\bigg(\dfrac{k}{a+k} + \dfrac{1}{u}\bigg)\]
\[ = \dfrac{1}{a+k}\bigg(1 + \dfrac{k}{a+k}u\bigg):\dfrac{1}{a+k-1} \to \dfrac{1}{a+k} \hspace{1pc} \text{ as $\hspace{1pc} u:\dfrac{a+k}{k(a+k-1)} \to 0$}.\] 
Therefore, $f_a^\#$ is the segment of the line $(a+k)^2v -ku = a+k$, from $\big(\frac{a+k}{k(a+k-1)},\frac{1}{a+k-1}\big)$ to $\big(0,\frac{1}{a+k}\big)$. Comparing $f_a^\#$ with $p_a^\#$ from lemma \ref{p_a_R}, we see that $f_a^\#$ is indeed the reflection of $p_a^\#$ along the diagonal $u=v$ in the $(u,v)$ plane. 
\end{proof}
\begin{proof}[Proof of theorem \ref{F_a_R}.]
The boundary of $F_a$ consists of $f_a, \tab f_{a+1}$ and the part of $p_0$ which connects the two rays. From theorem \eqref{P_a_R}, we know that the boundary of $P_a$ consists of $p_a \tab p_{a+1}$ and the part of $f_0$ which connects these line segments. Therefore, we see that it is enough to show that $f_a^\#$ is the reflection of $p_a^\#$ along the diagonal $u=v$ and conclude the result from the lemma. The proof for the $a=0$ case is identical once we remove the boundaries $p_0, f_0$ from $P_0, F_0$.   
\end{proof}
\vspace{2pc}
\includegraphics[scale=.55]{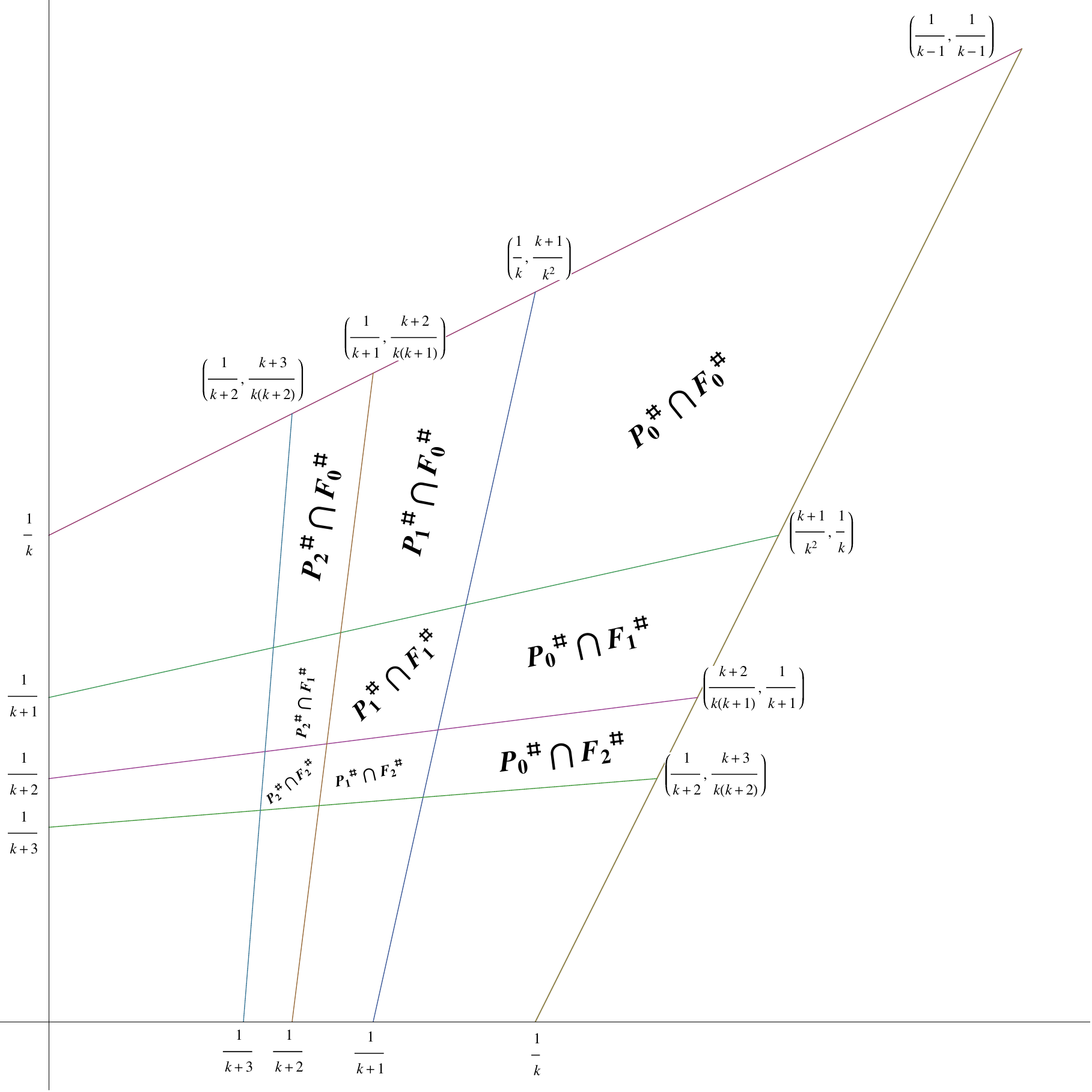}
\begin{center} {\it The region $\Gamma'_{(1,k)}$ revisited} \end{center}
\newpage
\section{The difference bisequence of approximation coefficients}{}
\noindent In this section we find bounds for the difference bisequence of approximation coefficients \[\big\{\eta_n(x_0,y_0)\big\}_{-\infty}^\infty := \big\{\big|\theta_n(x_0,y_0) - \theta_{n-1}(x_0,y_0)\big|\big\}_{-\infty}^\infty.\]
\begin{lemma}\label{P_a_cap_F_b_R}
The region $P_a^\# \cap F_b^\#$ is the quadrangle with vertices
\[\bigg(\dfrac{b+k}{b+(a+k-1)(b+k)}, \dfrac{a+k}{b + (a+k-1)(b+k)}\bigg), \tab \bigg(\dfrac{b+k}{b+(a+k)(b+k)},\dfrac{a+k+1}{k + (a+k)(b+k)}\bigg),\]
\[\bigg(\dfrac{b+k+1}{b+(a+k-1)(b+k+1)},\dfrac{a+k}{b + (a+k-1)(b+k+1)}\bigg)\]
and
\[\bigg(\dfrac{b+k+1}{b+(a+k)(b+k+1)},\dfrac{a+k+1}{b + (a+k)(b+k+1)}\bigg)\]
excluding two of its boundary line segments.
\end{lemma}
\begin{proof} 
Apply $\Psi(x,y) = \big(\frac{1}{x-y}, -\frac{x{y}}{k(x-y)}\big)$ to $p_a \cap f_b = \big(\frac{b}{b+k}, 1-(a+k)\big)$ and obtain
\begin{equation}\label{p_a_cap_f_b_R}
p_a^\# \cap f_b^\# = \bigg( \dfrac{b+k}{b + (a+k-1)(b+k)}, \dfrac{a+k}{b + (a+k-1)(b+k)}\bigg).
\end{equation}
The result now follows since $P_a \cap F_b \subset \Omega$ is the region interior to the quadrangle whose sides are $p_a,f_b,p_{a+1}$ and $f_{b+1}$ including the first and second line segments but not the third and fourth.
\end{proof}
\begin{theorem}
Given $(x_0,y_0) \in \Omega$, we write $a_n := a_n(x_0,y_0)$  for all $n \in \mathbb{Z}$ and fix $N \in \mathbb{Z}$. If $a := \min\{a_N,a_{N+1}\}, \tab A :=  \max\{a_N,a_{N+1}\}, \tab b :=  \min\{a_{N+1},a_{N+2}\}$ and $B :=  \max\{a_{N+1},a_{N+2}\}$, then
\[\eta_N(x_0,y_0)^2 + \eta_{N+1}(x_0,y_0)^2 \le \bigg(\dfrac{b+k}{b + (a+k-1)(b+k)} - \dfrac{B+k+1}{B+1 + (A+k)(B+k+1)}\bigg)^2\] 
\[+ \bigg(\dfrac{a+k}{b + (a+k-1)(b+k)} - \dfrac{A+k+1}{B+1 + (A+k)(B+k+1)}\bigg)^2.\]
\end{theorem}
\begin{proof}
The greatest possible distance between two points in a subset $X$ of the Euclidean metric space $(\mathbb{E}^2,\operatorname{d})$ is the diameter of $X$, $\operatorname{diam}(X) := \displaystyle{\sup_{x,y \in X}\big(\operatorname{d}(x,y)\big)}$. When $X$ is a convex quadrangle with opposite acute angles, its diameter equals the length of the diagonal connecting the vertices for these acute angles. Let $R_{(1,k,a,b,A,B)} = R$ be the closure of $\displaystyle{\bigcup_{a \le i \le A, \tab b \le j \le B}(P_i^\# \cap F_j^\#)}$. From theorem \ref{P_a_cap_F_b_R}, we see that $R$ is a convex quadrangle, whose boundary consists of the line sections $p_a^\#, \tab p_{A+1}^\#, \tab f_b^\#$ and $f_{B+1}^\#$. From the calculations of \eqref{P_a_R} and \eqref{F_a_R}, we see that the slopes of $p_a^\#,p_{A+1}^\#$ lie in $(1,\infty)$ and the slope of $f_b^\#,f_{B+1}^\#$ lie in $(0,1)$. Conclude that $R$ has opposite acute angles and that its diameter is the length of the diagonal connecting the vertices $ p_a^\# \cap f_b^\#$ and $p_{A+1}^\#\ \cap f_{B+1}^\#$ of $R$. From equation \eqref{p_a_cap_f_b_R} and the distance formula on $\mathbb{E}^2$ we have 
\[\operatorname{d}(p_a^\# \cap f_b^\#, p_{A+1}^\# \cap f_{B+1}^\#)^2 =  \bigg(\dfrac{b+k}{b + (a+k-1)(b+k)} - \dfrac{B+k+1}{B+1 + (A+k)(B+k+1)}\bigg)^2\] 
\[+ \bigg(\dfrac{a+k}{b + (a+k-1)(b+k)} - \dfrac{A+k+1}{B+1 + (A+k)(B+k+1)}\bigg)^2 = \operatorname{diam(R)}^2.\] 
Since $(\theta_{N-1},\theta_N) \in P_{a_N}^\# \cap F_{a_{N+1}}^\# \subset R$ and $(\theta_N,\theta_{N+1}) \in P_{a_{N+1}}^\# \cap F_{a_{N+2}}^\# \subset R$, we have $\eta_N^2 + \eta_{N+1}^2 = (\theta_N - \theta_{N-1})^2 +  (\theta_{N+1} - \theta_N)^2 \le \operatorname{diam(R)}^2$  and conclude the result.
\end{proof}
\begin{corollary}
Given $(x_0,y_0) \in \Omega$, write $a_n := a_n(x_0,y_0)$ for all $n \in \mathbb{Z}$. Fixing $N \in \mathbb{Z}$, we let $l := \min\{a_N,a_{N+1},a_{N+2}\}$ and $L :=  \max\{a_N,a_{N+1},a_{N+2}\}$. Then
\[\eta_N^2 + \eta_{N+1}^2 \le 2\tab\bigg(\dfrac{l+k}{(l+k)^2-k} -  \dfrac{L+k+1}{(L+k+1)^2-k}\bigg)^2.\]
\end{corollary}
\begin{proof}
We let $a,b,A$ and $B$ be as in the statement of the previous theorem. Since $a \le A$ and $b \le B$, we have
\[(b+k)(B+1) = b{B} + k{B} +  b + k >  b{B} + k{b} + b = b(B+k+1)\]
and
\[(b+k)(B+k+1)(A+k) > (b+k)(B+k+1)(a+k-1)\]
hence
\[\dfrac{b+k}{b + (a+k-1)(b+k)} > \dfrac{B+k+1}{B+1 + (A+k)(B+k+1)}.\]
Also, since $l \le \min\{a,b\}$, we have
\[\dfrac{b+k}{b + (a+k-1)(b+k)} \le \dfrac{l+k}{l + (a+k-1)(l+k)} \le \dfrac{l+k}{l + (l+k-1)(l+k)} = \dfrac{l+k}{(l+k)^2-k}\]  
and since $L \ge \max\{A,B\}$ we have
\[\dfrac{B+k+1}{B+1 + (A+k)(B+k+1)} \ge \dfrac{L+k+1}{L+1 + (A+k)(L+k+1)}\]
\[\ge \dfrac{L+k+1}{L+1 + (L+k)(L+k+1)} = \dfrac{L+k+1}{(L+k+1)^2-k}.\]
Then
\[0 \le \dfrac{b+k}{b + (a+k-1)(b+k)} - \dfrac{B+k+1}{B+1 + (A+k)(B+k+1)} \le \dfrac{l+k}{(l+k)^2-k} -  \dfrac{L+k+1}{(L+k+1)^2-k}.\]
This argument remains identical after we exchange $a$ for $b$ and $A$ for $B$. The result now follows directly from the theorem.
\end{proof}
\section{Remarks about the k equal one case}{}\label{RAk=1C}
\noindent In this section, we pay attention to the special case $m=k=1$. This corresponds to the classical Renyi map and backwards continued fraction expansion. From formula \eqref{Omega'} we have
$$\Omega'_{(1,1)} =  \displaystyle{\lim_{k \to 1^+}\Omega'_{(1,k)}} = [0,1) \times (-\infty, 0].$$
Similarly, from the characterization of $\Gamma'_{(1,k)}$ in section \ref{SOAP_R}, we obtain the space of approximation pairs:
\begin{proposition}\label{Gamma_1_R}
$\Gamma'_{(1,1)} = \displaystyle{\lim_{k \to 1^+}\Gamma'_{(1,k)}}$ is the unbounded region which is the intersection of $v-u < 1, \tab u-v \le 1, \tab u>0$ and $v\ge 0$.
\end{proposition}
Given $k>1=m$, the point $(x_0,y_0) := ([\tab \overline{0} \tab], 1-0-k-[\tab \overline{0} \tab] ) = (0,1-k)$
will generate the the dynamic bisequence $(x_n,y_n) = (0,1-k)$ with corresponding constant bisequence of approximation coefficients $\{\theta_n\} = \big\{\frac{1}{x_{n+1} - y_{n+1}}\big\} = \big\{\frac{1}{k-1}\big\}$. Also, proposition \ref{uniform_bound_R} implies that $\frac{1}{k-1}$ is the largest possible value for $\theta_n$. As $k \to 1^{+}, \frac{1}{k-1} \to \infty$, enabling us to conclude that there is no uniform bound on the sequence of approximation coefficients for the classical Renyi case. However, since $\Gamma_{(1,1)}$ is bounded by $v-u < 1$ and  $u-v < 1$ we obtain at once
\begin{theorem}
 Let $m=k=1$ and $(x_0,y_0) \in \Omega_{(1,1)}$. Then 
 \[\eta_n(x_0,y_0) = \abs{\theta_n(x_0,y_0) - \theta_{n-1}(x_0,y_0)} < 1\] 
for all $n\in\mathbb{Z}$.
\end{theorem}
\noindent Therefore, even though in general there exists no uniform upper bound for the approximation coefficients, the difference bi-sequence of approximation coefficients $\{\eta_n\}$ is uniformly bounded by one.

  \chapter{The arithmetic of the BAC for Gauss-like maps}{}
In this chapter, we focus on the Gauss-like cases $m=0$. We fix $m=0, \tab k > 1$ and omit the subscript $\square_{(0,k)}$ throughout. We will establish the connection between the bisequences $\{a_n(x_0,y_0)\}_{-\infty}^\infty$ and $\{\theta_n(x_0,y_0)\}_{-\infty}^\infty$ as well as generalize the theorems of the introduction to these cases. Finally, we will determine the first two members of the Markoff sequence which construe the top part of the associated Markoff Spectrum and the corresponding golden and silver ratios.\\ 

Using our abbreviated notation, we reiterate the Gauss-like maps as $T(x):[0,1) \to [0,1)$ where $T{x} = A{x} - \lfloor A{x} \rfloor$ and $A{x} = \frac{k(1-x)}{x}, \tab A(0) = 0$. Starting with $(x_0,y_0) \in \Omega$, write $(x_n,y_n) := \ccT^n(x_0,y_0)$ for all $n \in \mathbb{Z}$, where
\begin{equation}\label{ccT_explicit}
(x_{n+1},y_{n+1}) = \ccT(x_n,y_n) = \bigg(\dfrac{k}{x_n}- k - a_n, \tab \dfrac{k}{y_n} - k - a_n \bigg) 
\end{equation}
and
\begin{equation}\label{ccT_inv_explicit}
(x_{n-1},y_{n-1}) = \ccT^{-1}(x_n,y_n) = \bigg(\dfrac{k}{a_n+k+x_n}, \tab \dfrac{k}{a_n+k+y_n}\bigg).
\end{equation}

\newpage
\section{The pair of present digits}{}
\noindent Given $(x_0,y_0) \in \Omega$, we call the pair of digits $(a_n, a_{n+1}) := \big(a_n(x_0,y_0),a_{n+1}(x_0,y_0)\big)$ the \bf{pair of present digits} for $(x_0,y_0)$ at time $n \in \mathbb{Z}$. The next theorem determines this pair using the pair of approximation coefficients of $(x_0,y_0)$ at time $n$. First, define the quantity
\begin{equation}\label{D_n}
D_{(0,k,n)} = D_n := D(\theta_{n-1},\theta_n) = \sqrt{1-4k\theta_{n-1}\theta_n},
\end{equation}
as in formula \eqref{D}.
\begin{theorem}\label{thm_a_n+1}
Let $(\theta_{n-1},\theta_n)$ be the approximation pair of $(x_0,y_0)$ at time $n$. Then
\begin{equation}\label{a_n}
a_n = \bigg\lfloor \dfrac{2k\theta_n}{1-D_n} - k \bigg\rfloor
\end{equation}
and
\begin{equation}\label{a_n+1}
a_{n+1} = \bigg\lfloor \dfrac{2k\theta_{n-1}}{1-D_n} -k \bigg\rfloor.
\end{equation}
Lowering all indexes by one in \eqref{a_n+1} enables us to obtain the following alternative representation of $a_n$ as a function of $(\theta_{n-2},\theta_{n-1})$:
\begin{equation}\label{a_n_alt}
a_n = \bigg\lfloor \dfrac{2k\theta_{n-2}}{1-D_{n-1}} -k \bigg\rfloor.
\end{equation}
\end{theorem}
\begin{proof}
Apply the definition of $\Psi^{-1}$ \eqref{Psi_inv} and obtain
\begin{equation}\label{present_digits}
(x_n,y_n) =  \Psi^{-1}(\theta_{n-1},\theta_n) = \bigg(\dfrac{1-D_n}{2\theta_{n-1}}, -\dfrac{1+D_n}{2\theta_{n-1}}\bigg).
\end{equation}
Using formulas \eqref{x_n} and \eqref{ccT_explicit}, write $x_n = [a_{n+1},r_{n+2}] = \frac{k}{a_{n+1}+k+[r_{n+2}]}$, so that the first components in the exterior terms of formula \eqref{present_digits} equate to 
\[a_{n+1} + k + [r_{n+2}] = \frac{2k\theta_{n-1}}{1-D_n}.\] 
Since $[r_{n+2}] <1$, we conclude 
\[a_{n+1} = \big\lfloor a_{n+1} + [r_{n+2}] \big\rfloor = \bigg\lfloor \dfrac{2k\theta_{n-1}}{1-D_n} -k \bigg\rfloor,\] 
which is the proof of formula \eqref{a_n+1}. Using formula \eqref{y_n}, write $y_n + a_n + k = [s_n]$, so that the second components in the exterior terms of formula \eqref{present_digits} equate to $a_n + k + [s_n] = \frac{1+D_n}{2\theta_{n-1}}$. Since $[s_n] <1$, we conclude
\[a_n = \big\lfloor a_n + [s_n] \big\rfloor = \bigg\lfloor  \dfrac{1+D_n}{2\theta_{n-1}} - k \bigg\rfloor.\] 
The proof of formula \eqref{a_n} now follows from the chain of equalities
\[\dfrac{1+D_n}{2\theta_{n-1}} = \dfrac{(1+D_n)(1-D_n)}{2\theta_{n-1}(1-D_n)} = \dfrac{1-D_n^2}{2\theta_{n-1}(1-D_n)}  = \frac{1 - (1 - 4k\theta_{n-1}\theta_n)}{2\theta_{n-1}(1-D_n)} = \dfrac{2k\theta_n}{1-D_n}.\] 
\end{proof}

\newpage
\section{Extending approximation pairs}{}
\noindent In this section, we will prove that knowing any two consecutive approximation coefficients is enough to generate the entire bi-sequence of approximation coefficients. More specifically, we will show that we can extend the approximation pair of $(x_0,y_0)$ at time $n$, $(\theta_{n-1},\theta_n),$ to the approximation pair of $(x_0,y_0)$ at time $n \pm 1$. We start by Defining the conjugation of $\ccT$ by $\Psi$ to be
\[\ccK_{(0,k)} = \ccK:\Gamma \to \Gamma, \hspace{1pc} (u,v) \mapsto \Psi\ccT\Psi^{-1}.\]
Since $\ccT$ and $\Psi$ are continuous bijections, the map $\ccK$ must be a continuous bijection as well. We have 
\[(\theta_n,\theta_{n+1}) =  \Psi(x_{n+1},y_{n+1}) =  \Psi\ccT(x_n,y_n) = \Psi\ccT{\Psi^{-1}}(\theta_{n-1},\theta_n) = \ccK(\theta_{n-1},\theta_n).\] 
In words: The conjugation of the map $\ccT$ by $\Psi$ applied to the approximation pair at time $n$ is the approximation pair at time $n+1$. Next, for every non-negative integer $a$, define the function  
\begin{equation}\label{g_a}
g_{(0,k,a)} = g_a: \Gamma \to \IR , \tab g_a(u,v) := u + \dfrac{D(u,v)}{k}(a+k) - \dfrac{v}{k}(a+k)^2,
\end{equation}
where $D(u,v) = \sqrt{1-4k{u}v}$ is as in formula \eqref{D}. Since $D(u,v)$ is well defined for all $(u,v) \in \Gamma$, $g_a$ is clearly a well defined continuous function on $\Gamma$.
\begin{lemma}
Let $u,v$ be such that $(u,v) \in \Gamma$ and let $a := \big\lfloor \frac{2k{u}}{1-D(u,v)} -k \big\rfloor$. Then 
\begin{equation}\label{K}
\ccK(u,v) = \big(v,g_a(u,v)\big)
\end{equation}
and
\begin{equation}\label{u=g_a}
u = g_a\big(g_a(u,v),v\big).
\end{equation} 
\end{lemma}
\begin{proof}
Given $(u,v) \in \Gamma$, set $(x_0,y_0) := \Psi^{-1}(u,v) \in  \Omega$. From the definition of $\Psi$ \eqref{Psi}, we have
\begin{equation}\label{u,v}
(u,v) = \Psi(x_0,y_0) = \bigg(\dfrac{1}{x_0-y_0}, -\dfrac{x_0y_0}{k(x_0-y_0)}\bigg),
\end{equation}
so that
\begin{equation}\label{v=f(x,y)}
v = \bigg(\dfrac{k}{x_0} - \dfrac{k}{y_0}\bigg)^{-1}.
\end{equation}
Next, let $(x_1,y_1) := \ccT(x_0,y_0)$ and  $(v',w) := \Psi(x_1,y_1)$. From the formula for $\ccT$ \eqref{ccT_explicit}, we have $(x_1,y_1) = (\frac{k}{x_0} - a_1 - k, \frac{k}{y_0} - a_1 - k)$ and after using the definition of $\Psi$ \eqref{Psi}, we obtain 
\[(v',w) = \bigg(\dfrac{1}{x_1 - y_1},-\dfrac{x_1{y_1}}{k(x_1 - y_1)}\bigg) = \bigg(\big(\dfrac{k}{x_0}-\dfrac{k}{y_0}\big)^{-1}, -\dfrac{x_1{y_1}}{k(x_1 - y_1)}\bigg).\] 
Using equation \eqref{v=f(x,y)}, we conclude that $v=v'$ and 
\begin{equation}\label{v,w}
(v,w) = \bigg(\dfrac{1}{x_1-y_1}, -\dfrac{x_1y_1}{k(x_1-y_1)}\bigg).
\end{equation}
From the definition of $\Psi^{-1}$ \eqref{Psi_inv} and formula \eqref{a_n}, we obtain 
\[(x_1,y_1) = \ccT(x_0,y_0) = \ccT\Psi^{-1}(u,v) = \ccT\bigg(\dfrac{1-D(u,v)}{2u},-\dfrac{1+D(u,v)}{2u}\bigg),\] 
hence
\begin{equation}\label{x_1,y_1}
(x_1,y_1) = \bigg(\dfrac{2k{u}}{1-D(u,v)} - k - a,-\dfrac{2k{u}}{1+D(u,v)} - k -a \bigg)
\end{equation}
where 
\begin{equation}\label{a=a_1}
a = a_1(x_0,y_0) = \bigg\lfloor \dfrac{2k{u}}{1-D(u,v)} -k \bigg\rfloor.
\end{equation}
Therefore $w = f_a(u,v)$, where $a$ is as in the hypothesis.\\
Using the definition of $\Psi$ \eqref{Psi} and equation \eqref{x_1,y_1}, we find that the first component of 
$\Psi(x_1,y_1) = \Psi\ccT\Psi^{-1}(u,v)$ is \[(x_1-y_1)^{-1} = \bigg(\dfrac{2ku}{1-D(u,v)} + \dfrac{2k{u}}{1+D(u,v)} \bigg)^{-1} = \bigg(\dfrac{4k{u}}{1-D(u,v)^2}\bigg)^{-1} = v,\] 
as expected from equation \eqref{u,v}. Another application of equation \eqref{x_1,y_1} allows us to compute the second component of $\Psi(x_1,y_1)$
\[w = -\frac{x_1y_1}{k(x_1 - y_1)} = -\frac{v}{k}x_1y_1 = \frac{v}{k}\bigg(\dfrac{2k{u}}{1-D(u,v)} - (a+k)\bigg)\bigg(\dfrac{2k{u}}{1+D(u,v)} + (a+k)\bigg)\] 
\[= \dfrac{v}{k}\bigg(\dfrac{4k^2u^2}{1-D(u,v)^2} + \dfrac{4k{u}D(u,v)}{1-D(u,v)^2}(a+k) - (a+ k)^2\bigg)= \dfrac{v}{k}\bigg(\dfrac{4k^2u^2}{4k{u}v} + \frac{D(u,v)}{v}(a+k) - (a+ k)^2\bigg).\] Conclude
\begin{equation}\label{w}
w = u + \dfrac{D(u,v)}{k}(a+k) - \dfrac{v}{k}(a+k)^2 = g_a(u,v),
\end{equation}
which is the proof of formula \eqref{K}. \\
To prove the second part, we use the definition of $\Psi^{-1}$ \eqref{Psi_inv} and equation \eqref{v,w} to write
\[(x_1,y_1) = \Psi^{-1}(v,w) = \bigg(\dfrac{1-D(v,w)}{2v},-\frac{1+D(v,w)}{2v}\bigg).\] 
Together with formulas \eqref{ccT_inv_explicit} and \eqref{a=a_1}, we obtain 
\[(x_0,y_0) = \ccT^{-1}(x_1,y_1) = \bigg(\dfrac{k}{a+k+x_1},\frac{k}{a+k+y_1}\bigg)\] 
and conclude
\begin{equation}\label{x_0,y_0}
(x_0,y_0)= \bigg(k\big((a+k) + \dfrac{1-D(v,w)}{2v}\big)^{-1},k\big((a+k) - \dfrac{1+D(v,w)}{2v}\big)^{-1}\bigg).
\end{equation}
Using the definition of $\Psi$ \eqref{Psi}, we rewrite $(u,v) = \Psi(x_0,y_0)$ as
\[(u,v) = \Psi(x_0,y_0) = \bigg(\dfrac{1}{x_0-y_0},-\dfrac{x_0y_0}{k(x_0-y_0)}\bigg) = \bigg(-\dfrac{k}{x_0y_0}\cdot \big(-\dfrac{x_0y_0}{k(x_0-y_0)}\big), -\dfrac{x_0y_0}{k(x_0-y_0)}\bigg)\] 
\[= \bigg(-\dfrac{k}{x_0y_0}\cdot\big(\dfrac{k}{x_0} - \dfrac{k}{y_0}\big)^{-1}, \big(\dfrac{k}{x_0} - \dfrac{k}{y_0}\big)^{-1}\bigg).\] 
We plug in the values for $x_0,y_0$ as in equation \eqref{x_0,y_0} and obtain
\[\big(\frac{k}{x_0} - \frac{k}{y_0}\big)^{-1} = \bigg(\big((a_1+k) + \frac{1-D(v,w)}{2v}\big) - \big((a_1+k) - \frac{1+D(v,w)}{2v}\big)\bigg)^{-1} = \big(\frac{2}{2v}\big)^{-1} = v,\] 
as expected from equation \eqref{u,v}. Using equation \eqref{x_0,y_0} again, we find that the first component of $\Psi(x_0,y_0)$ is
\[u=-\dfrac{k}{x_0y_0}\cdot\bigg(\dfrac{k}{x_0} - \dfrac{k}{y_0}\bigg)^{-1} = -\dfrac{k}{x_0y_0}v =  -\dfrac{v}{k}\cdot\bigg((a+k) + \dfrac{1-D(v,w)}{2v}\bigg)\cdot\bigg((a+k) - \dfrac{1+D(v,w)}{2v}\bigg)\] 
\[= -\dfrac{v}{k}\cdot\bigg((a+k)^2 -\dfrac{2D(v,w)}{2v}(a+k) - \dfrac{1-D(v,w)^2}{4v^2}\bigg) = w + \dfrac{D(v,w)}{k}(a+k) -\frac{v}{k}(a+k)^2 = g_a(w,v).\] 
Formula \eqref{w} now asserts the validity of equation \eqref{u=g_a} and completes the proof.
\end{proof}
\begin{theorem}\label{theta_pm_1}
For all $n \in \Z$
\begin{equation}\label{theta_n+1}
\theta_{n+1} = g_{a_{n+1}}(\theta_{n-1},\theta_n),
\end{equation}
and
\begin{equation}\label{theta_n-1}
\theta_{n-1} = g_{a_{n+1}}(\theta_{n+1},\theta_n),
\end{equation}
\end{theorem}
\begin{proof}
After setting $(u,v) := (\theta_{n-1},\theta_n)$, the first part is obtained at once from formulas \eqref{a_n+1} and \eqref{K}. Then the second part follows from formula \eqref{u=g_a}.
\end{proof}
\begin{corollary}\label{a_n+1+k}
\begin{equation}
a_{n+1} + k = \dfrac{1}{2\theta_n}(D_n + D_{n+1}).
\end{equation}
\end{corollary}
\begin{proof}
After replacing $\theta_{n+1}$ with its value in \eqref{theta_n+1} and plugging into equation \eqref{theta_n-1}, we obtain 
\[\theta_{n-1} = \theta_{n-1} + \frac{D_n}{k}(a_{n+1}+k) - \dfrac{\theta_n}{k}(a_{n+1}+k)^2 + \dfrac{D_{n+1}}{k}(a_{n+1}+k) -\frac{\theta_n}{k}(a_{n+1}+k)^2,\] 
which yields the desired result after the appropriate cancellations and rearrangements.
\end{proof}
Combining \eqref{theta_n+1} and \eqref{theta_n-1} with \eqref{a_n} and \eqref{a_n+1} allows us to extend the approximation pair at time $n$ without reference to the digits of expansion.
\begin{theorem}\label{theta_extension} 
Let $\theta_n := \theta_n(x_0,y_0)$ be the approximation coefficient for $(x_0,y_0)$ at time $n$ and write $D_n := D(\theta_{n-1},\theta_n) = \sqrt{1 - 4k\theta_{n-1}\theta_n}$ as in \eqref{D}. for all $n \in \mathbb{Z}$. Then for all $n \in \mathbb{Z}$ we have
\[\theta_{n+1} = \theta_{n-1} + \dfrac{D_n}{k}\bigg(\bigg\lfloor\dfrac{2k\theta_{n-1}}{1-D_n} -k \bigg\rfloor +k\bigg) - \dfrac{\theta_n}{k}\bigg(\bigg\lfloor \dfrac{2k\theta_{n-1}}{1-D_n} -k \bigg\rfloor +k\bigg)^2,\] 
\[\theta_{n-2} =  \theta_n + \dfrac{D_{n}}{k}\bigg(\bigg\lfloor \dfrac{2k\theta_n}{1-D_n} -k \bigg\rfloor+k\bigg) - \dfrac{\theta_{n-1}}{k}\bigg(\bigg\lfloor\dfrac{2k\theta_n}{1-D_n} -k \bigg\rfloor+k\bigg)^2.\] 
Consequently, the entire bi-sequence $\{\theta_n\}_{-\infty}^\infty$ can be recovered from a single pair of successive members.   
\end{theorem}
\section{The constant bi-sequence of approximation coefficients}
\noindent For any non-negative integer $a$, define the (0,k)-irrational constants 
\begin{equation}\label{xi_a}
\xi_{(0,k,a)} := [ \tab \overline{a} \tab ]_{(0,k)} = [ \tab\overline{a}\tab ]  = [a,a,...],
\end{equation} 
and
\begin{equation}\label{C_a}
C_{(0,k,a)} := \dfrac{1}{\sqrt{(a+k)^2 + 4k}}
\end{equation}
Since $m=0$ and $k$ are fixed, we will simply denote these constants as $\xi_a$ and $C_a$.
\begin{theorem}\label{theta_n_constant}
Let $(x_0,y_0) \in \Omega$. Write $a_n := a_n(x_0,y_0), \tab \theta_n := \theta_n(x_0,y_0)$ for all $n \in \mathbb{Z}$ and let $a$ be a non-negative integer. Then the following are equivalent:
\begin{enumerate}
\item $a_n = a$ for all $n \in \mathbb{Z}$.
\item $(x_0,y_0) = (\xi_a,-a-k-\xi_a)$.
\item $\theta_{-1} = \theta_0 = C_a$.   
\item $\theta_n = C_a$ for all $n \in \mathbb{Z}$.  
\end{enumerate}
\end{theorem}
\begin{proof}
$\\ (i) \implies (ii): $ follows directly from formulas \eqref{x_n}, \eqref{y_n} and the definition of $\xi_a$ \eqref{xi_a}.\\
$(ii) \implies (iii): $  When $x = \xi_a = [\tab \overline{a} \tab]$, we have $a_1(x,y)=a_1(x) = a$. Furthermore, $T$ acts as a left shift operator on the digits of expansion, hence it fixes $\xi_a$. From the definition of $T$ \eqref{T_m_k}, we have  \[\xi_a = [\tab \overline{a} \tab] = T([\tab \overline{a} \tab])= T(\xi_a) = \dfrac{k}{\xi_a} - k -a,\] 
so that
\begin{equation}\label{xi_a_quadratic}
\xi_a^2 + (a+k)\xi_a - k = 0.
\end{equation}
Using the quadratic formula, we obtain the roots $\frac{1}{2}\big(-(a+k)\pm\sqrt{(a+k)^2 +4k}\big)$. Since the smaller root is clearly negative, we conclude
\[\xi_a =  \dfrac{1}{2}\big(\sqrt{(a+k)^2 +4k}-(a+k)\big),\]
which, together with formula \eqref{C_a}, provide the relationship
\begin{equation}\label{xi_C_a}
C_a = \dfrac{1}{2\xi_a+(a+k)}.
\end{equation}
The definition \eqref{Psi} of $\Psi$ now yields 
\[\Psi(x_0,y_0) = \Psi\big(\xi_a, -(a+k+\xi_a)\big) = \bigg(\dfrac{1}{2\xi_a+(a+k)}, \big(\dfrac{1}{2\xi_a+(a+k)}\big)\dfrac{1}{k}\xi_a(a + k + \xi_a) \bigg)\]
\[=\bigg(\dfrac{1}{2\xi_a+(a+k)},\dfrac{1}{2\xi_a+(a+k)}\dfrac{1}{k}\big(\xi_a^2 + (a+k)\xi_a\big)\bigg) = (C_a,C_a),\] 
where the last equality is obtained from equations \eqref{xi_a_quadratic} and \eqref{xi_C_a}. 
If $x_0 = \xi_a$ and $y_0 = -a-k-\xi_a$ then combining this last observation and formula \eqref{theta_dynamic} yields
\begin{equation}\label{Psi_C_a}
(\theta_{-1},\theta_0) = \Psi(x_0,y_0) = (C_a,C_a).
\end{equation}
$(iii) \implies (iv): $ We use the definition of $g_a$ \eqref{g_a} and formula \eqref{C_a} to obtain 
\[g_a(C_a,C_a) = C_a + \dfrac{a+k}{k}\sqrt{1 - \frac{4k}{(a+k)^2 + 4k}} - \dfrac{C_a}{k}(a+k)^2\] 
\[= C_a + \dfrac{(a+k)^2}{k\sqrt{(a+k)^2 + 4k}} - \dfrac{C_a}{k}(a+k)^2 = \dfrac{k + (a+k)^2 - (a+k)^2}{k\sqrt{(a+k)^2 + 4k}} = C_a,\] 
and conclude
\begin{equation}\label{g_C_a}
g_a(C_a,C_a) = C_a.
\end{equation}
If $\theta_{-1} = \theta_0 = C_a$, then, since $\Psi$ is a bijection, formula \eqref{Psi_C_a} implies $(x_0,y_0) = (\xi_a, -a-k-\xi_a)$ and $a_n(x_0,y_0) = a$ for all $n \in \mathbb{Z}$. Equations \eqref{theta_n+1} and \eqref{g_C_a} now prove 
\[\theta_{1} = g_{a_1}(\theta_{-1},\theta_0) = g_a(C_a,C_a) = C_a.\] 
Similarly, equations \eqref{theta_n-1} and \eqref{g_C_a} prove 
\[\theta_{-2} = g_{a_1}(\theta_0,\theta_{-1}) = g_a(C_a,C_a) = C_a.\] 
\noindent We repeat this argument with $\theta_0 = \theta_1 = C_a$ and $\theta_{-2} = \theta_{-1} = C_a$ and obtain $\theta_{-3} = \theta_2 = C_a$ as well. The proof that $\{\theta_n\}_{-\infty}^\infty = \{C_a\}$ is the indefinite extension of equation \eqref{g_C_a} to all $n \in \mathbb{Z}$.\\
$(iv) \implies (i) : $ If $\theta_n = C_a$ for all $n \in \mathbb{Z}$, then $D_n = \sqrt{1-4k{C_a}^2}$ for all $n \in \mathbb{Z}$ and Corollary \ref{a_n+1+k} implies that
\[ (a_{n+1} + k)^2 = \dfrac{1}{4\theta_n^2}(D_n + D_{n+1})^2 = \dfrac{1}{4C_a^2}4(1-4k{C_a}^2) = \dfrac{1}{C_a^2} - 4k. \]
But from formula \eqref{C_a}, we know that $(a_{n+1}+k)^2 = \frac{1}{C_{a_{n+1}}^2} - 4k$ so we must have $a_{n+1} = a$ for all $n \in \mathbb{Z}$.
\end{proof}
\begin{corollary}
Let $(x_0,y_0) \in \Omega$ and write $a_n := a_n(x_0,y_0)$ and $\theta_n := \theta_n(x_0,y_0)$ for all $n \in \mathbb{Z}$. Then $\{a_n\}_{-\infty}^\infty$ is constant if and only if $\{\theta_n\}_{-\infty}^\infty$ is constant.
\end{corollary}
\begin{proof}
The necessary condition follows immediately from the previous theorem. Suppose $\{\theta_n\} = \{\theta\}$ is constant and write $D := \sqrt{1-4k\theta^2}$ so that $D = D_0 = D_1$ is as in formula \eqref{D_n}. Then corollary \ref{a_n+1+k} yields
\[(a_1+k)^2 = \big(\dfrac{1}{2\theta_0}(D_0+D_1)\big)^2 = \dfrac{4D^2}{4\theta^2} = \dfrac{1 - 4k\theta^2}{\theta^2} = \dfrac{1}{\theta^2} - 4k.\] 
But from formula \eqref{C_a}, we know that $(a_1+k)^2 = \frac{1}{C_{a_1}^2} - 4k$. Since $\theta_n > 0$, we conclude that $\theta = C_{a_1}$. The previous theorem now proves that $a_n(x_0,y_0) = a_1$ for all $n \in \mathbb{Z}$.
\end{proof}
\newpage
\section{The arithmetic of approximation coefficients}{}
\begin{theorem}\label{triple_thm}
Let $(x_0,y_0) \in \Omega$ and write $\theta_n := \theta_n(x_0,y_0)$ and $a_{n+1} := a_{n+1}(x_0,y_0)$. Then
\begin{equation}\label{triple_min}
\min\{\theta_{n-1},\theta_n,\theta_{n+1}\} \le C_{a_{n+1}}
\end{equation}
and 
\begin{equation}\label{triple_max}
\max\{\theta_{n-1},\theta_n,\theta_{n+1}\} \ge C_{a_{n+1}}.
\end{equation}
Furthermore, $C_{a_{n+1}}$ is the best possible constant, that is, it cannot be replaced by a smaller constant in the minimum case or a larger constant in the maximum case.
\end{theorem}
\begin{proof}
Assume, by contradiction, that $\min\{\theta_{n-1},\theta_n,\theta_{n+1}\} > C_{a_{n+1}}$. Then using the definition of $C_a$ \eqref{C_a}, we obtain 
\[\min\{\theta_{n-1}\theta_n,\theta_n\theta_{n+1}\} > C_{a_{n+1}}^2 = \frac{1}{(a_{n+1}+k)^2+4k},\] 
hence 
\[D_n{D_{n+1}} \le \max\{D_n^2,D_{n+1}^2\} =  \max\{1-4k\theta_{n-1}\theta_n, 1-4k\theta_n\theta_{n+1}\} < 1-\dfrac{4k}{(a_{n+1}+k)^2+4k}.\] 
We conclude
\begin{equation}\label{max_D^2}
\max\{ D_n{D_{n+1}},D_n^2,D_{n+1}^2\} < 1-\dfrac{4k}{(a_{n+1}+k)^2+4k} = \dfrac{(a_{n+1}+k)^2}{(a_{n+1}+k)^2+4k}.  
\end{equation}
Also, since we are assuming $\theta_n > C_{a_{n+1}}$, we have 
\[\dfrac{1}{4\theta_n^2} < \dfrac{1}{4C_{a_{n+1}}^2} = \dfrac{(a_{n+1}+k)^2 + 4k}{4}.\] 
Using this last observation together with formulas \eqref{a_n+1+k} and \eqref{max_D^2}, we obtain the contradiction 
\[(a_{n+1}+k)^2 = \frac{1}{4\theta_n^2}(D_n^2 + D_{n+1}^2 + 2D_n{D_{n+1}})\]
\[< \frac{(a_{n+1}+k)^2 + 4k}{4}\bigg(\frac{(a_{n+1}+k)^2}{(a_{n+1}+k)^2+4k}+\frac{(a_{n+1}+k)^2}{(a_{n+1}+k)^2+4k}+\frac{2(a_{n+1}+k)^2}{(a_{n+1}+k)^2+4k}\bigg) =(a_{n+1}+k)^2,\] 
which proves the inequality \eqref{triple_min}.\\ 

Next, assume by contradiction that $\max\{\theta_{n-1},\theta_n,\theta_{n+1}\} < C_{a_{n+1}}$. Then using the definition of $C_a$ \eqref{C_a}, we obtain  
\[\max\{\theta_{n-1}\theta_n,\theta_n\theta_{n+1}\} < C_{a_{n+1}}^2 = \frac{1}{(a_{n+1}+k)^2+4k},\] 
hence 
\[D_n{D_{n+1}} \ge \min\{D_n^2,D_{n+1}^2\} > \min\{1-4k\theta_{n-1}\theta_n, 1-4k\theta_n\theta_{n+1}\} > 1-\dfrac{4k}{(a_{n+1}+k)^2+4k}.\] 
We conclude
\begin{equation}\label{min_D^2}
\min\{D_n^2,D_{n+1}^2,D_nD_{n+1}\} > 1-\dfrac{4k}{(a_{n+1}+k)^2+4k} = \dfrac{(a_{n+1}+k)^2}{(a_{n+1}+k)^2+4k}.
\end{equation}
Also, since we are assuming $\theta_n < C_{a_{n+1}}$, we have 
\[\dfrac{1}{4\theta_n^2} > \dfrac{1}{4C_{a_{n+1}}^2} = \dfrac{(a_{n+1}+k)^2 + 4k}{4}.\] 
Using this last observation together with formulas \eqref{a_n+1+k} and \eqref{min_D^2}, we obtain the contradiction 
\[(a_{n+1}+k)^2 = \frac{1}{4\theta_n^2}(D_n^2 + D_{n+1}^2 + 2D_nD_{n+1})\] 
\[> \frac{(a_{n+1}+k)^2 + 4k}{4}\bigg(\frac{(a_{n+1}+k)^2}{(a_{n+1}+k)^2+4k}+\frac{(a_{n+1}+k)^2}{(a_{n+1}+k)^2+4k}+\frac{2(a_{n+1}+k)^2}{(a_{n+1}+k)^2+4k}\bigg) =(a_{n+1}+k)^2,\] 
which proves inequality \eqref{triple_max}.\\ 
Finally, if $(x,y) = \big(\xi_a,-(a+k+\xi_a)\big) \in \Omega$, then we conclude from theorem \ref{theta_n_constant} that $a_{n+1} = a$ and $\theta_n = C_a$ for all $n \in \mathbb{Z}$. Thus $C_{a_{n+1}}$ is the best possible constant. 
\end{proof}
\noindent From the definition of $C_a$ \eqref{C_a}, it is clear that $C_a \le C_0$ for all $a \ge 0$. Then this theorem proves 
\begin{corollary}\label{hurwitz}
Given  $(x_0,y_0) \in \Omega$, the inequality $\theta_n(x_0,y_0) \le C_0 = \frac{1}{\sqrt{k^2 + 4k}}$ holds for infinitely many $n$'s. Furthermore we cannot replace $C_0$ with any smaller constant. 
\end{corollary}
\begin{remark} 
In the classical $k=1$ case, $C_0$ is, as expected, the Hurwitz Constant $\dfrac{1}{\sqrt{5}}$.
\end{remark}
\newpage
\section{The first two members of the Markoff Sequence}{}
\noindent By corollary \ref{hurwitz}, the first member of the Markoff sequence \\
$\displaystyle{\sup_{(x_0,y_0) \in \Omega}\bigg\{\liminf_{n \to\pm\infty}}\big\{\theta_n(x_0,y_0)\big\}\bigg\}$ is the general Hurwitz Constant $C_0 = \frac{1}{\sqrt{k^2+4k}}$. Denote the countable set of all $(x_0,y_0) \in \Omega$ for which the digits expansion for $(x_0,y_0)$ starts from or end with an infinite tail of zeros by $\Omega_0$.
\begin{theorem}\label{silver}
The next member in the Markoff sequence is the quantity
\[\displaystyle{\sup_{(x_0,y_0) \in \Omega - \Omega_0}\bigg\{\liminf_{n \in \mathbb{Z}}}\big\{\theta_n(x_0,y_0)\big\}\bigg\} = \dfrac{1}{\sqrt{k^2+6k+1}}.\] 
Furthermore, this bound is realized for the pair $(x_0,y_0) := (\xi_1, -k-1-\xi_1)$.
\end{theorem}
\begin{proof}
Suppose $(x_0,y_0) \in \Omega$ is such that there exists $N \ge 0$ with $a_n(x_0,y_0) = 0$ for all $n \ge N$. Then from formula \eqref{intro_ccT_explicit} we have $x_n = \xi_0$ and $y_n = -k - [\displaystyle{\overbrace{0,0,0,...,0}^{\text{$n-N$ times}}},s_{(n-N)}]$ for all $n > N$, hence $(x_n,y_n) = \ccT^n(x,y) \to (\xi_0,-k-\xi_0)$ as $n \to \infty$. Similarly, if there exists $N' \ge 0$ such that $a_n(x_0,y_0) = 0$ for all $n \le -N'$ then from formula \eqref{intro_ccTinv_explicit} we have $y_n = -k - \xi_0$ and $x_n = [\displaystyle{\overbrace{0,0,0,...,0}^{\text{$N'-n$ times}}},r_{(n-N'+1)}]$ for all $n < -N'$, hence $(x_n,y_n) = \ccT^n(x,y) \to (\xi_0,-k-\xi_0)$ as $n \to -\infty$. In either case we have $(\theta_{n-1},\theta_n) = \Psi(x_n,y_n) \to C_0$ by the continuity of $\Psi$ and theorem \ref{theta_n_constant}. Thus, the next member of the Markoff sequence is $\displaystyle{\sup_{(x_0,y_0) \in \Omega - \Omega_0}\bigg\{\liminf_{n \in \mathbb{Z}}}\big\{\theta_n(x_0,y_0)\big\}\bigg\}$.\\
 
Take $(x_0,y_0) := (\xi_1,-k-1-\xi_1) = \big([ \overline{1} ], -k-1- [ \overline{1} ]\big)  \in \Omega - \Omega_0$. Then theorem \ref{theta_n_constant} proves $\{\theta_n(x_0,y_0)\}$ is the constant bisequence $\theta_n = C_1$, so that 
\[\displaystyle{\sup_{(\Omega - \Omega_0)}\bigg\{\liminf_{n \to\pm\infty}}\big\{\theta_n(x_0,y_0)\big\}\bigg\} \ge  C_1 = \dfrac{1}{\sqrt{k^2+6k+1}}.\] 
It is left to show that $C_1$ cannot be replaced with any larger constant, that is
\begin{equation}\label{mu_C_1}
\displaystyle{\sup_{(\Omega - \Omega_0)}\bigg\{\liminf_{n \to\pm\infty}}\big\{\theta_n(x_0,y_0)\big\}\bigg\} \le C_1.  
\end{equation}
Given $(x_0,y_0) \in \Omega-\Omega_0$, its expansion must contain a digit which is not zero infinitely often. First suppose $(x_0,y_0)$ contains an infinite amount of digits which are larger than 2. Using the formula for $y_n$ \eqref{y_n}, we obtain $\ccT^n(x_0,y_0) = (x_n,y_n) = (x_n,-k-a_n-[s_n])$, with $a_n \ge 3$ for infinitely many $n \in \mathbb{Z}$. Then formula \eqref{theta_dynamic} yields the inequality
\[\dfrac{1}{\theta_{n+1}(x_0,y_0)} = x_n - y_n = x_n + k + a_n + [s_n] \]
\[ > k+3 = \sqrt{(k+3)^2} = \sqrt{k^2 + 6k +9} > \sqrt{k^2 + 6k +1}\] 
infinitely many times, hence the inequality \eqref{mu_C_1} applies to $(x_0,y_0)$. Conclude that we need only consider those $(x_0,y_0) \in \Omega - \Omega_0$ having only finitely many digits greater than 2. If $a_n(x_0,y_0) = 1$ only finitely many times, then $a_n(x_0,y_0) = a_{n+1}(x_0,y_0) = 2$ for infinitely many $n$'s. After using the formulas for $x_n$ and $y_n$ \eqref{x_n} and \eqref{y_n}, the equality 
\[\ccT^n(x_0,y_0) = (x_n,y_n) = ([a_{n+1},r_{n+2}],-k-a_n-[s_n] =  ([2,r_{n+2}],-k-2-[s_n])\] 
holds for infinitely many $n \in \mathbb{Z}$. Fixing such $N$, we have $[r_{N+2}] < 1$ hence we use equation \eqref{ccT_explicit} to obtain $x_N = [2,r_{N+2}] = \frac{k}{k + 2 + [r_{N+2}]} > \frac{k}{k+3}$. Formula \eqref{theta_dynamic} now yields the inequality
\[\frac{1}{\theta_{N+1}(x_0,y_0)} = x_N-y_N  = x_N + k +2 + [s_N] >  \frac{k}{k+3} + k+2\] 
\[ = \sqrt{\bigg(\frac{k}{k+3} + (k+2)\bigg)^2} = \sqrt{\bigg(\big(\frac{k}{k+3}\big)^2 + 3\bigg) + 1 + k^2 + 4k + \frac{2k(k+2)}{k+3}}\] 
\[ = \sqrt{\frac{k^2 + 3(k+3)^2}{(k+3)^2} + 1 + k^2 + 4k + \frac{2k(k+2)}{k+3}} >  \sqrt{\frac{3(k+3)}{k+3}+ 1 + k^2 + 4k + \frac{2k(k+2)}{k+3}}\] 
\[=\sqrt{\frac{2k^2 +7k + 9}{k+3}+ 1 + k^2 + 4k} > \sqrt{\frac{2k^2 + 6k}{k+3}+ 1 + k^2 + 4k} = \sqrt{k^2 + 6k +1}.\] 
Since this inequality will occur infinitely often, we may exclude this family of dynamic pairs as well. We are left with are the pairs $(x_0,y_0)$ for which $a_{n+1}(x_0,y_0)=1$ infinitely often. Plugging $a_{n+1}=1$ in theorem \ref{triple_thm} yields the validity of inequality \eqref{mu_C_1} for these dynamic pairs as well and concludes the proof.
\end{proof}
\section{The one-sided sequence of approximation coefficients}
\noindent We end this chapter by quoting these results which apply to the one-sided sequence of approximation coefficients as well. Fix $x_0 \in (0,1) - \IQ_{(0,k)}$ and write $a_n := a_n(x_0)$ and $\theta_n = \theta_n(x_0) = \frac{1}{x_n - Y_n}$ for all $n \ge 1$, where $x_n$ and $Y_n$ are the future and past of $x_0$ at time $n$ as in formulas \eqref{x_n} and \eqref{Y_n} and $\theta_n(x_0)$ is as in the definition \eqref{theta_x_0}. Using formulas \eqref{ccT_explicit} and \eqref{ccT_inv_explicit}, we see that the maps $\ccT$ and $\ccT^{-1}$ are well defined on $(x_n,Y_n)$ and that $(x_n,Y_n) = \ccT^n(x_0,Y_0)$ for $n \ge 1$. Using Haas' result \eqref{theta_future_past} and the definition of the map $\Psi$ \eqref{Psi}, we see that $\Psi(x_n,Y_n) = (\theta_{n-1},\theta_n)$ for all $n \ge 1$. Then, the proofs of theorems \ref{a_n}, \ref{theta_pm_1} and \ref{theta_extension} remain the same, after we restrict $n \ge 1$ and replace $y_n$ with $Y_n$. Consequently, these results apply to the one-sided sequences $\{a_n\}_1^\infty$ and $\{\theta_n\}_0^\infty$.\\

Recall from remark \ref{digit_remark} that if $\{b_n\}_1^\infty$ is the RCF sequence of digits and $\{a_n\}_1^\infty$ is the (0,1)-sequence of digits for an irrational number on the interval, then $b_n = a_n + 1$ for all $n \ge 1$. In particular, we can apply theorem \ref{a_n} for the classical Gauss case $m=0$ and $k=1$ to prove
\begin{theorem}\label{a_n+1_classical}
Let $\{b_n\}_1^\infty$ and $\{\theta_n\}_0^\infty$ be the classical RCF sequences of digits and approximation coefficients for $x_0 \in (0,1) -\IQ$. Then
\[
b_{n+1} =  \bigg\lfloor \frac{2\theta_{n-1}}{1 - \sqrt{1-4\theta_{n-1}\theta_n}} \bigg\rfloor = \bigg\lfloor \frac{2\theta_{n+1}}{1 - \sqrt{1 - 4\theta_{n+1}\theta_n}}\bigg\rfloor,  \hspace{1pc} n \ge 2.
\]
\end{theorem}

  \chapter{The arithmetic of the BAC for Renyi-like maps}{}
In this chapter, we focus on the Renyi-like case $m=1$. We fix $k > m = 1$ and omit the subscript $\square_{(1,k)}$ throughout. We will establish the connection between the bisequences $\{a_n(x_0,y_0)\}_{-\infty}^\infty$ and $\{\theta_n(x_0,y_0)\}_{-\infty}^\infty$ as well as generalize the theorems of the introduction to these cases.\\ 

Using our abbreviated notation, we reiterate the Renyi-like maps as $T{x}:[0,1) \to [0,1), \tab T{x} = A{x} - \lfloor A{x} \rfloor$, where $A{x} =  \dfrac{k{x}}{1-x}$. Starting with $(x_0,y_0) \in \Omega$, write $(x_n,y_n) := \ccT^n(x_0,y_0)$ for all $n \in \mathbb{Z}$, where
\begin{equation}\label{ccT_explicit_R}
(x_{n+1},y_{n+1}) = \ccT(x_n,y_n) = \bigg(\dfrac{k{x_n}}{1-x_n} - a_n, \tab \dfrac{k{y_n}}{1-y_n} - a_n \bigg)
\end{equation}
and
\begin{equation}\label{ccT_inv_explicit_R}
(x_{n-1},y_{n-1}) = \ccT^{-1}(x_n,y_n) = \bigg(\dfrac{a_n + x_n}{a_n+k+x_n}, \tab \dfrac{a_n + x_n}{a_n+k+y_n}\bigg).
\end{equation}

\section{The pair of present digits}{}
\noindent Given $(x_0,y_0) \in \Omega$, we call the pair of digits $(a_n, a_{n+1}) := \big(a_n(x_0,y_0),a_{n+1}(x_0,y_0)\big)$ the \bf{pair of present digits} for $(x_0,y_0)$ at time $n \in \mathbb{Z}$. The next theorem determines this pair using the pair of approximation coefficients of $(x_0,y_0)$ at time $n$. First, define the quantity
\begin{equation}\label{D_n_R}
D_{(1,k,n)} = D_n :=  D(\theta_{n-1},\theta_n) = \sqrt{1+4k\theta_{n-1}\theta_n}
\end{equation}
as in formula \eqref{D_R}.
\begin{theorem}\label{thm_a_n+1_R}
Let $(\theta_{n-1},\theta_n)$ be the approximation pair of $(x_0,y_0)$ at time $n \in \mathbb{Z}$. Then
\begin{equation}\label{a_n_R}
a_n = \bigg\lfloor \dfrac{2k\theta_n}{D_n-1} - k \bigg\rfloor,  
\end{equation}
and
\begin{equation}\label{a_n+1_R}
a_{n+1} =  \bigg\lfloor \dfrac{2k\theta_{n-1}}{D_n-1} - k \bigg\rfloor. 
\end{equation}
Lowering all indexes by one in \eqref{a_n+1_R} enables us to obtain the following alternate representation of $a_n$ as a function of $(\theta_{n-2},\theta_{n-1})$:
\begin{equation}\label{a_n_alt_R}
a_n = \bigg\lfloor \dfrac{2k\theta_{n-2}}{D_{n-1}-1} -k \bigg\rfloor.
\end{equation}
\end{theorem}
\begin{proof}
Apply the definition of $\Psi^{-1}$ \eqref{Psi_inv_R} and obtain
\begin{equation}\label{present_digits_R}
(x_n,y_n) = \Psi^{-1}(\theta_{n-1},\theta_n) = \bigg(1+\dfrac{1-D_n}{2\theta_{n-1}},1 -\dfrac{1+D_n}{2\theta_{n-1}}\bigg).
\end{equation} 
Using formulas \eqref{x_n} and \eqref{ccT_explicit_R}, write $x_n = [a_{n+1},r_{n+2}] = 1 - \frac{k}{a_{n+1}+k+[r_{n+2}]}$, so that the first components in the exterior terms of formula \eqref{present_digits_R} equate to 
\[1 - \frac{k}{a_{n+1} + k + [r_{n+2}]} = 1 + \frac{1-D_n}{2\theta_{n-1}}.\] 
Since $[r_{n+2}] <1$, we conclude
\[a_{n+1}  =   \big\lfloor a_{n+1} + [r_{n+2}] \big\rfloor = \bigg\lfloor \dfrac{2k\theta_{n-1}}{D_n-1} - k \bigg\rfloor ,\]
which is the proof of formula \eqref{a_n+1_R}. Using formula \eqref{y_n}, write $1 - k - a_n - y_n = [s_n]$, so that the second components in the exterior terms of formula \eqref{present_digits_R} equate to $1 - a_n - k - [s_n] = 1 - \frac{1+D_n}{2\theta_{n-1}}$. Since $[s_n] <1$, we conclude
\[a_n = \lfloor a_n + [s_n] \rfloor = \bigg\lfloor \frac{D_n+1}{2\theta_{n-1}} - k \bigg\rfloor\]
The proof of formula \eqref{a_n_R} now follows from the chain of equalities
\[\dfrac{D_n + 1}{2\theta_{n-1}} = \dfrac{(D_n + 1)(D_n - 1)}{2\theta_{n-1}(D_n - 1)} = \dfrac{D_n^2-1}{2\theta_{n-1}(D_n-1)} =  \dfrac{1 + 4k\theta_{n-1}\theta_n - 1}{2\theta_{n-1}(D_n-1)} =\dfrac{2k\theta_n}{D_n-1}.\]
\end{proof}
\section{Extending approximation pairs}{}
\noindent 
In this section, we will prove that knowing any two consecutive approximation coefficients is enough to generate the entire bi-sequence of approximation coefficients. More specifically, we will show that we can extend the approximation pair of $(x_0,y_0)$ at time $n$, $(\theta_{n-1},\theta_n),$ to the approximation pair of $(x_0,y_0)$ at time $n \pm 1$. We start by defining the conjugation of $\ccT$ by $\Psi$ to be
\[\ccK_{(1,k)} = \ccK:\Gamma \to \Gamma, \hspace{1pc} (u,v) \mapsto \Psi\ccT\Psi^{-1}.\]
Since $\ccT$ and $\Psi$ are continuous bijections, the map $\ccK$ must be a continuous bijection as well. We have 
\[(\theta_n,\theta_{n+1}) =  \Psi(x_{n+1},y_{n+1}) =  \Psi\ccT(x_n,y_n) = \Psi\ccT{\Psi^{-1}}(\theta_{n-1},\theta_n) = \ccK(\theta_{n-1},\theta_n).\] 
In words: the conjugation of the map $\ccT$ by $\Psi$ applied to the approximation pair at time $n$ is the approximation pair at time $n+1$. Next, for every non-negative integer $a$, define the function  
\begin{equation}\label{g_a_R}
g_{(1,k,a)} = g_a:  \Gamma \to \IR , \tab g_a(u,v) := u - \dfrac{D(u,v)}{k}(a+k+1) + \dfrac{v}{k}(a+k+1)^2,
\end{equation}
where $D(u,v) = \sqrt{1+4kuv}$ is as in formula \eqref{D_R}. Since $D(u,v)$ is well defined for all $(u,v) \in \Gamma$, $g_a$ is clearly a well defined continuous function on $\Gamma$.
\begin{lemma}
Let $u,v$ be such that $(u,v) \in \Gamma$ and let $a := \big\lfloor \frac{2ku}{D(u,v)-1} -k \big\rfloor$. Then 
\begin{equation}\label{K_R}
\ccK(u,v) = \big(v,g_a(u,v)\big)
\end{equation}
and
\begin{equation}\label{u=g_a_R}
u = g_a\big(g_a(u,v),v\big).
\end{equation}
\end{lemma}
\begin{proof}
Given $(u,v) \in \Gamma$, set $(x_0,y_0) := \Psi^{-1}(u,v) \in  \Omega$. From  the definition of $\Psi$ \eqref{Psi_R}, we have
\begin{equation}\label{u,v_R}
(u,v) = \Psi(x_0,y_0) = \bigg(\dfrac{1}{x_0-y_0}, \dfrac{(1-x_0)(1-y_0)}{k(x_0-y_0)}\bigg),
\end{equation}
so that
\begin{equation}\label{v=f(x,y)_R}
v = \bigg(\dfrac{k{x_0}}{1-x_0} - \dfrac{k{y_0}}{1-y_0}\bigg)^{-1}.
\end{equation}
Next, let $(x_1,y_1) := \ccT(x_0,y_0)$ and  $(v',w) := \Psi(x_1,y_1)$. From the formula for $\ccT$ \eqref{ccT_explicit_R}, we have $(x_1,y_1) = \big(\frac{k{x_0}}{1 - x_0} - a_1, \frac{k{y_0}}{1 - y_0} - a_1 - k\big)$ and after using the definition of $\Psi$ \eqref{Psi_R}, we obtain 
\[(v',w) = \bigg(\dfrac{1}{x_1 - y_1},\dfrac{(1-x_1)(1-y_1)}{k(x_1 - y_1)}\bigg) = \bigg(\big(\frac{k{x_0}}{1 - x_0}-\frac{k{y_0}}{1 - y_0}\big)^{-1}, \frac{(1-x_1)(1-y_1)}{k(x_1 - y_1)}\big)\bigg).\] 
Using equation \eqref{v=f(x,y)_R}, we conclude that $v=v'$, hence
\begin{equation}\label{v,w_R}
(v,w) = \bigg(\dfrac{1}{x_1-y_1}, \dfrac{(1-x_1)(1-y_1)}{k(x_1-y_1)}\bigg).
\end{equation}
From the definition of $\Psi^{-1}$ \eqref{Psi_inv_R} and formula \eqref{a_n_R}, we obtain 
\[(x_1,y_1) = \ccT(x_0,y_0) = \ccT\Psi^{-1}(u,v) = \ccT\bigg(1 + \dfrac{1-D(u,v)}{2u},1-\dfrac{1+D(u,v)}{2u}\bigg),\] 
hence
\begin{equation}\label{x_1,y_1_R}
(x_1,y_1) = \bigg(\dfrac{2ku}{D(u,v)-1} - k - a,\dfrac{2ku}{D(u,v)+1} - k - a\bigg),
\end{equation}
where 
\begin{equation}\label{a=a_1_R}
a = a_1(x_0,y_0) = \bigg\lfloor \dfrac{2ku}{D(u,v)-1} -k \bigg\rfloor.
\end{equation} 
Therefore $w = f_a(u,v)$, where $a$ is as in the hypothesis.\\
Using the definition of $\Psi$ \eqref{Psi_R} and equation \eqref{x_1,y_1_R}, we find that the first component of $\Psi(x_1,y_1) = \Psi\ccT\Psi^{-1}(u,v)$ is 
\[(x_1-y_1)^{-1} = \bigg(\dfrac{2ku}{D(u,v)-1} - \frac{2ku}{D(u,v)+1} \bigg)^{-1} = \bigg(\dfrac{4ku}{D(u,v)^2 - 1}\bigg)^{-1} = \bigg(\frac{4ku}{(1+4kuv)-1}\bigg)^{-1} = v,\] 
as expected from equation \eqref{v,w_R}. Another application of equation \eqref{x_1,y_1_R} allows us to compute $w$, which is the second component of $\Psi(x_1,y_1)$ as
\[\frac{(1-x_1)(1-y_1)}{k(x_1 - y_1)} = \dfrac{v}{k}(1-x_1)(1-y_1) = \frac{v}{k}\bigg(a+k+1-\frac{2ku}{D(u,v)-1}\bigg)\bigg(a+k+1-\frac{2ku}{D(u,v)+1}\bigg)\]
\[ = \frac{v}{k}\bigg(\dfrac{4k^2u^2}{D(u,v)^2-1} - \dfrac{4kuD(u,v)}{D(u,v)^2-1}(a+k+1) + (a+ k+1)^2\bigg)\]
\[ = \frac{v}{k}\bigg(\dfrac{k{u}}{v} - \dfrac{D(u,v)}{v}(a+k+1) + (a+ k+1)^2\bigg) =  u - \dfrac{D(u,v)}{k}(a+k+1) + \dfrac{v}{k}(a+k+1)^2 = g_a(u,v).\] 
We conclude that
\begin{equation}\label{w_R}
w = g_a(u,v),
\end{equation}
which is the proof of formula \eqref{K_R}. \\
To prove the second part, we use the definition of $\Psi^{-1}$ \eqref{Psi_inv_R} and equation \eqref{v,w_R} to write 
\[(x_1,y_1) = \Psi^{-1}(v,w) = \bigg(1+\dfrac{1-D(v,w)}{2v},1-\frac{1+D(v,w)}{2v}\bigg).\] 
Together with formulas \eqref{ccT_inv_explicit_R} and \eqref{a=a_1_R}, we obtain 
\[(x_0,y_0) = \ccT^{-1}(x_1,y_1) = \bigg(1 - \dfrac{k}{a+k+x_1}, 1-\dfrac{k}{a+k+y_1}\bigg)\] 
and conclude
\begin{equation}\label{x_0,y_0_R}
(x_0,y_0) = \bigg(1 - k\big((a+k+1) + \dfrac{1-D(v,w)}{2v}\big)^{-1},1 - k\big((a+k+1) - \dfrac{1+D(v,w)}{2v}\big)^{-1}\bigg).
\end{equation}
Using the definition of $\Psi$ \eqref{Psi_R}, we rewrite $(u,v) = \Psi(x_0,y_0)$ as
\[\Psi(x_0,y_0) = \bigg(\dfrac{1}{x_0-y_0}, \dfrac{(1-x_0)(1-y_0)}{k(x_0-y_0)}\bigg)\]
\[ = \bigg(\dfrac{k}{(1-x_0)(1-y_0)}\cdot \big(\dfrac{(1-x_0)(1-y_0)}{k(x_0-y_0)}\big),\dfrac{(1-x_0)(1-y_0)}{k(x_0-y_0)}\bigg)\] 
\[= \bigg(\dfrac{k}{(1-x_0)(1-y_0)}\cdot\big(\frac{k}{1 - x_0} - \frac{k}{1 - y_0}\big)^{-1}, \big(\dfrac{k}{1 - x_0} - \frac{k}{1 - y_0}\big)^{-1}\bigg).\] 
We plug in the values for $x_0,y_0$ as in equation \eqref{x_0,y_0_R} and obtain 
\[\bigg(\dfrac{k}{1 - x_0} - \dfrac{k}{1 - y_0}\bigg)^{-1} = \bigg(\big((a_1+k+1) + \frac{1-D(v,w)}{2v}\big) - \big((a_1+k+1) - \frac{1+D(v,w)}{2v}\big)\bigg)^{-1} = v,\] 
as expected from equation \eqref{u,v_R}. Using equation \eqref{x_0,y_0_R} again, we find that the first component of $\Psi(x_0,y_0)$ is
\[u= \dfrac{k}{(1-x_0)(1-y_0)}\cdot\bigg(\dfrac{k}{1-x_0} - \dfrac{k}{1-y_0}\bigg)^{-1} = \dfrac{k}{(1-x_0)(1-y_0)}v\] 
\[= k{v}\bigg(\big(\dfrac{1}{k}(a+k+1) + \dfrac{1 - D(v,w)}{2v}\big)\big(\dfrac{1}{k}(a+k+1) - \dfrac{1+D(v,w)}{2v}\big)\bigg)\] \[=\dfrac{v}{k}\bigg(\frac{D(v,w)^2 - 1}{4v^2}  - \dfrac{D(v,w)}{v}(a+k+1) + (a+k+1)^2\bigg)\] 
\[= \dfrac{v}{k}\bigg(w - \dfrac{D(v,w)}{k}(a+k+1) + \dfrac{v}{k}(a+k+1)^2\bigg) = g_a(w,v).\] 
Formula \eqref{w_R} now asserts the validity of equation \eqref{u=g_a_R} and completes the proof.
\end{proof}
\begin{theorem}\label{theta_pm_1_R}
For all $n \in \Z$
\begin{equation}\label{theta_n+1_R}
\theta_{n+1} = g_{a_{n+1}}(\theta_{n-1},\theta_n),
\end{equation}
and
\begin{equation}\label{theta_n-1_R}
\theta_{n-1} = g_{a_{n+1}}(\theta_{n+1},\theta_n),
\end{equation}
\end{theorem}
\begin{proof}
After setting $(u,v) := (\theta_{n-1},\theta_n)$, the first part is obtained at once from formulas \eqref{a_n+1_R} and \eqref{K_R}. Then the second part follows from formula \eqref{u=g_a_R}.
\end{proof}
\begin{corollary}\label{a_n+1+k_R}
For all $n \in \Z$
\begin{equation}
a_{n+1} + k +1 = \dfrac{1}{2\theta_n}(D_n + D_{n+1}).
\end{equation}
\end{corollary}
\begin{proof}
After replacing $\theta_{n+1}$ with its value in \eqref{theta_n+1_R} and plugging into equation \eqref{theta_n-1_R}, we obtain 
\[\theta_{n-1} = \theta_{n-1} - \dfrac{D_n}{k}(a_{n+1}+k+1) + \dfrac{\theta_n}{k}(a_{n+1}+k+1)^2 - \dfrac{D_{n+1}}{k}(a_{n+1}+k+1) + \dfrac{\theta_n}{k}(a_{n+1}+k+1)^2,\] 
which yields the desired result after the appropriate cancellations and rearrangements.
\end{proof}
Combining the formulas \eqref{theta_n+1_R} and \eqref{theta_n-1_R} with \eqref{a_n_R} and \eqref{a_n+1_R} allows us to extend the approximation pair at time $n$ without reference to the digits of expansion.
\begin{theorem}\label{theta_extension_R} 
Let $\theta_n := \theta_n(x_0,y_0)$ be the approximation coefficient for $(x_0,y_0)$ at time $n$ and write $D_n := D(\theta_{n-1},\theta_n) = \sqrt{1 + 4k\theta_{n-1}\theta_n}$ as in formula \eqref{D_R} for all $n \in \mathbb{Z}$. Then, for all $n \in \mathbb{Z}$, we have
\[\theta_{n+1} = \theta_{n-1} - \dfrac{D_n}{k}\bigg(\bigg\lfloor \dfrac{2k\theta_{n-1}}{D_n - 1} -k \bigg\rfloor +k + 1\bigg) + \dfrac{\theta_n}{k}\bigg(\bigg\lfloor \dfrac{2k\theta_{n-1}}{D_n - 1} -k \bigg\rfloor +k + 1\bigg)^2, \]
\[\theta_{n-2} =  \theta_n - \dfrac{D_{n}}{k}\bigg(\bigg\lfloor \dfrac{2k\theta_n}{D_n - 1} -k \bigg\rfloor+k + 1\bigg) + \dfrac{\theta_{n-1}}{k}\bigg(\bigg\lfloor \dfrac{2k\theta_n}{D_n - 1} -k \bigg\rfloor+k + 1\bigg)^2.\]
Consequently, the entire bi-sequence $\{\theta_n\}_{-\infty}^\infty$ can be recovered from a single pair of successive members.  
\end{theorem}
\newpage
\section{The constant bi-sequence of approximation coefficients}
\noindent For any non-negative integer $a$, define the constants  
\begin{equation}\label{xi_a_R}
\xi_{(1,k,a)} := [\tab \overline{a} \tab]_{(1,k)} = [a,a,...]_{(1,k)}
\end{equation} 
\begin{equation}\label{C_infty_R}
C_{(1,k,\infty)} := 0
\end{equation}
and
\begin{equation}\label{C_a_R}
C_{(1,k,a)} := \dfrac{1}{\sqrt{(a+k-1)^2 + 4a}}
\end{equation}
which, after a brief calculation, may alternatively be written as 
\begin{equation}\label{C_a_alt_R}
C_{(1,k,a)} := \dfrac{1}{\sqrt{(a+k+1)^2 - 4k}}.
\end{equation}
Since $m=1$ and $k$ are fixed, we will simply denote these constants as $\xi_a, C_a$ and $C_\infty$. Given two non-negative integers $a$ and $b$, it is clear that  

\begin{equation}\label{C_a_C_b_R}
a \le b \tab \text{if and only if}\tab  C_b \le C_a
\end{equation}
and that this inequality remains true if we allow $a$ or $b$ to equal $\infty$.

\begin{theorem}\label{theta_n_constant_R}
Fix $(x_0,y_0) \in \Omega$, write $a_n := a_n(x_0,y_0), \tab \theta_n := \theta_n(x_0,y_0)$ for all $n \in \mathbb{Z}$ and let $a$ be a non-negative integer. Then the following are equivalent:
\begin{enumerate}
\item $a_n=a$ for all $n \in \mathbb{Z}$. 
\item $(x_0,y_0) = (\xi_a,1-k-a-\xi_a)$.
\item $\theta_{-1} = \theta_0 = C_a$.   
\item $\theta_n = C_a$ for all $n \in \mathbb{Z}$. 
\end{enumerate}
\end{theorem}
\begin{proof}
$\\ \tab (i) \implies (ii): $ follows directly from formulas \eqref{x_n}, \eqref{y_n} and the definition of $\xi_a$ \eqref{xi_a_R}.\\
$(ii) \implies (iii): $  When $x = \xi_a = [ \tab\overline{a}\tab ]$, we have $a_1(x_0,y_0)= a$. Furthermore, $T$ acts as a left shift operator on the digits of expansion, hence it fixes $\xi_a$. From the definition of $T$ \eqref{T_m_k}, we have 
\[\xi_a = [\tab \overline{a} \tab] = T(\xi_a) = \dfrac{k\xi_a}{1-\xi_a} -a,\]
so that
\begin{equation}\label{xi_a_quadratic_R}
\xi_a^2 + (a+k-1)\xi_a - a = 0.
\end{equation}
Using the quadratic formula, we obtain the roots $\frac{1}{2}\big(-(a + k - 1) \pm \sqrt{(a+k-1)^2 +4a}\big)$. Since the smaller root is clearly negative, we conclude 
\[\xi_a =  \frac{1}{2}\big(\sqrt{(a+k-1)^2 +4a}-(a+k-1)\big),\] 
which, together with \eqref{C_a_R} provides the relationship
\begin{equation}\label{xi_a_C_a_R}
C_a = \dfrac{1}{2\xi_a+(a+k-1)}.
\end{equation}
The definition \eqref{Psi_R} of $\Psi$ now yields 
\[\Psi\big(\xi_a, 1-(a+k+\xi_a)\big) = \bigg(\dfrac{1}{2\xi_a+(a+k-1)}, \tab \big(\dfrac{1}{2\xi_a+(a+k-1)}\big)\dfrac{1}{k}(1-\xi_a)(a+k+\xi_a) \bigg)\] 
\[=  \bigg(\dfrac{1}{2\xi_a+(a+k-1)}, \tab \dfrac{1}{2\xi_a+(a+k-1)}\big(-\dfrac{1}{k}\big)(\xi_a-1)(a+k+\xi_a) \bigg)\] 
\[= \bigg(\dfrac{1}{2\xi_a+(a+k-1)}, \tab -\dfrac{1}{k}\cdot\dfrac{1}{2\xi_a+(a+k-1)}\big(\xi_a^2 + (a+k-1)\xi_a -a -k\big)\bigg) = (C_a,C_a), \] 
where the last equality is obtained from equations \eqref{xi_a_quadratic_R} and \eqref{xi_a_C_a_R}. If $x_0 = \xi_a$ and $y_0 = 1-k-a-\xi_a$ then combining this last observation and formula \eqref{theta_dynamic} yields
\begin{equation}\label{Psi_C_a_R}
(\theta_{-1},\theta_0) = \Psi(x_0,y_0) = (C_a,C_a).
\end{equation}
$(iii) \implies (iv): \tab$  We use the definition of $g_a$ \eqref{g_a_R} and formula \eqref{C_a_alt_R} to obtain 
\[g_a(C_a,C_a) = C_a - \dfrac{(a+k+1)}{k}\sqrt{1+4k{C_a^2}} + \dfrac{(a+k+1)^2}{k}C_a\] 
\[= C_a - \dfrac{(a+k+1)^2}{k}C_a + \dfrac{(a+k+1)^2}{k}C_a = C_a,\] 
and conclude
\begin{equation}\label{g_C_a_R}
g_a(C_a,C_a) = C_a.
\end{equation}
If $\theta_{-1} = \theta_0 = C_a$ then since $\Psi$ is a bijection, formula \eqref{Psi_C_a_R} implies $(x_0,y_0) = (\xi_a, 1- a -k -\xi_a)$ and $a_n(x_0,y_0) = a$ for all $n \in \mathbb{Z}$. Equations \eqref{theta_n+1_R} and \eqref{g_C_a_R} now prove 
\[\theta_{1} = g_{a_1}\big(\theta_{-1},\theta_0\big) = g_a(C_a,C_a) = C_a.\] 
Similarly, equations \eqref{theta_n-1_R} and \eqref{g_C_a_R} prove 
\[\theta_{-2} = g_{a_1}\big(\theta_0,\theta_{-1}\big) = g_a(C_a,C_a) = C_a.\] 
\noindent We repeat this argument with $\theta_0 = \theta_1 = C_a$ and $\theta_{-2} = \theta_{-1} = C_a$ and obtain $\theta_{-3} = \theta_2 = C_a$ as well. The proof that $\{\theta_n\}_{-\infty}^\infty = \{C_a\}$ is the indefinite extension of equation \eqref{g_C_a_R} for all $n \in \mathbb{Z}$.\\  
$(iv) \implies (i) : $ If $\theta_n = C_a$ for all $n \in \mathbb{Z}$, then $D_n = \sqrt{1+4k{C_a}^2}$ for all $n \in \mathbb{Z}$ and Corollary \ref{a_n+1+k_R} implies that
\[ (a_{n+1} + k + 1)^2 = \dfrac{1}{4\theta_n^2}(D_n + D_{n+1})^2 = \dfrac{1}{4C_a^2}4(1+4k{C_a}^2) = \dfrac{1}{C_a^2} + 4k. \]
But from formula \eqref{C_a_R}, we know that $(a_{n+1}+k + 1)^2 = \frac{1}{C_{a_{n+1}}^2} + 4k$ so we must have $a_{n+1} = a$ for all $n \in \mathbb{Z}$.
\end{proof}
\newpage
\begin{corollary}
Let $(x_0,y_0) \in \Omega$ and write $a_n := a_n(x_0,y_0)$ and $\theta_n := \theta_n(x_0,y_0)$ for all $n \in \mathbb{Z}$. Then $\{a_n\}_{-\infty}^\infty$ is constant if and only if $\{\theta_n\}_{-\infty}^\infty$ is constant.
\end{corollary}
\begin{proof}
The necessary condition follows immediately from the previous theorem. Suppose $\{\theta_n\} = \{\theta\}$ is constant and write $D := \sqrt{1+4k\theta^2}$ so that $D = D_0 = D_1$ is as in formula \eqref{D_n}. Then corollary \ref{a_n+1+k_R} yields
\[(a_1+k+1)^2 = \big(\dfrac{1}{2\theta_0}(D_0+D_1)\big)^2 = \dfrac{4D^2}{4\theta^2} = \dfrac{1 + 4k\theta^2}{\theta^2} = \dfrac{1}{\theta^2} + 4k.\] 
But from formula \eqref{C_a_R}, we know that $(a_1+k +1)^2 = \frac{1}{C_{a_1}^2} + 4k$. Since $\theta_n > 0$, we conclude that $\theta = C_{a_1}$. The previous theorem now proves that $a_n(x_0,y_0) = a_1$ for all $n \in \mathbb{Z}$.
\end{proof}
\section{The arithmetic of approximation coefficients}
\begin{lemma}\label{triple_R}
Suppose $(x_0,y_0) \in \Omega$ and write $a_n := a_n(x_0,y_0)$ and $\theta_n := \theta_n(x_0,y_0)$, for all $n \in \mathbb{Z}$. If $\theta_n = \max\{\theta_{n-1},\theta_n,\theta_{n+1}\}$ then $\theta_n \le C_{a_{n+1}} = \frac{1}{\sqrt{(a+k-1)^2 + 4a}}$ with equality precisely when $\theta_{n-1} = \theta_n = \theta_{n+1} = C_{a_{n+1}}$. Similarly, if $\theta_n = \min\{\theta_{n-1},\theta_n,\theta_{n+1}\}$ then $\theta_n \ge C_{a_{n+1}}$ with equality precisely when $\theta_{n-1} = \theta_n = \theta_{n+1} = C_{a_{n+1}}$.
\end{lemma}
\begin{proof}
If $\theta_n = \max\{\theta_{n-1},\theta_n,\theta_{n+1}\}$ then 
\[D_n = \sqrt{1+4k\theta_{n-1}\theta_n} \le \sqrt{1+4k\theta_n^2}\] 
with equality precisely when $\theta_{n-1}=\theta_n$ and 
\[D_{n+1} = \sqrt{1+4k\theta_n\theta_{n+1}} \le \sqrt{1+4k\theta_n^2}\] 
with equality precisely when $\theta_{n+1}=\theta_n$. Conclude that the inequality
\[\frac{1}{4\theta_n^2}\big(D_n^2+D_{n+1}^2 +2D_nD_{n+1} \big) \le \frac{1}{4\theta_n^2}\big(4(1 + 4k\theta_n^2)\big)\] 
must hold and cannot be replaced with equality unless $\theta_{n-1}=\theta_n=\theta_{n+1}$. In this case, theorem \ref{theta_n_constant_R} proves that $\theta_{n-1} = \theta_n = \theta_{n+1} = C_{a_{n+1}}$. Otherwise, we may replace the weak inequality with a strict one. If we further assume by contradiction that $\theta_n \ge C_{a_{n+1}}$ then corollary \ref{a_n+1+k_R} and formula \eqref{C_a_alt_R} with $a=a_{n+1}$, yield the contradiction 
\[(a_{n+1}+k+1)^2 = \frac{1}{4\theta_n^2}\big(D_{n-1}^2+D_n^2 +2D_{n-1}D_n \big) < \frac{1}{4\theta_n^2}\big(4(1 + 4k\theta_n^2)\big)\] 
\[= \frac{1}{\theta_n^2} + 4k \le \frac{1}{C_{a_{n+1}}^2} + 4k = (a_{n+1}+k+1)^2,\] 
which proves that $\theta_n$ must be strictly smaller than $C_{a_{n+1}}$ as desired. Similarly, suppose $\theta_n = \min\{\theta_{n-1},\theta_n,\theta_{n+1}\}$. Then 
\[D_n = \sqrt{1+4k\theta_{n-1}\theta_n} \ge \sqrt{1+4k\theta_n^2}\] 
with equality precisely when $\theta_{n-1}=\theta_n$ and 
\[D_{n+1} = \sqrt{1+4k\theta_n\theta_{n+1}} \ge \sqrt{1+4k\theta_n^2}\] 
with equality precisely when $\theta_{n+1}=\theta_n$. Conclude that the inequality 
\[\frac{1}{4\theta_n^2}\big(D_n^2+D_{n+1}^2 +2D_nD_{n+1} \big) \ge \frac{1}{4\theta_n^2}\big(4(1 + 4k\theta_n^2)\big)\]
must hold and cannot be replaced with equality unless $\theta_{n-1}=\theta_n=\theta_{n+1}$. In this case, theorem \ref{theta_n_constant_R} proves that $\theta_{n-1} = \theta_n = \theta_{n+1} = C_{a_{n+1}}$. Otherwise, we may replace the weak inequality with a strict one. If we further assume by contradiction that $\theta_n \le C_{a_{n+1}}$ then corollary \ref{a_n+1+k_R} and formula \eqref{C_a_alt_R} with $a=a_{n+1}$, yield the contradiction 
\[(a_{n+1}+k+1)^2 = \frac{1}{4\theta_n^2}\big(D_{n-1}^2+D_n^2 +2D_{n-1}D_n \big) > \frac{1}{4\theta_n^2}\big(4(1 + 4k\theta_n^2)\big)\] 
\[= \dfrac{1}{\theta_n^2} + 4k \ge \frac{1}{C_{a_{n+1}}^2} + 4k = (a_{n+1}+k+1)^2,\] 
which proves that $\theta_n$ must be strictly larger than $C_{a_{n+1}}$ completing the proof.\\
\end{proof}
\begin{theorem}\label{thm_mu_bounds_R}
Given $(x_0,y_0) \in \Omega$, we let $a_n := a_n(x_0,y_0)$ and $\theta_n := \theta_n(x_0,y_0)$. Let $0 \le l \le L \le \infty$ be such that
\[l = \displaystyle{\liminf_{n \in \mathbb{Z}}\{a_n\} \le \limsup_{n \in \mathbb{Z}}\{a_n\}} = L.\]
Then 
\begin{equation}\label{mu_bounds_R}
C_L \le \displaystyle{\liminf_{n \in \mathbb{Z}}\{\theta_n\} \le \limsup_{n \in \mathbb{Z}}\{\theta_n\}} \le C_l.
\end{equation}
Furthermore, these constants are the best possible, in the sense that $C_L$ cannot be replaced with a larger constant and $C_l$ cannot be replaced with a smaller constant.
\end{theorem}
\begin{proof}
From our assumption, there exists $N_0 \ge 1$ such that $l \le a_{n+1}(x_0,y_0) \le L$ for all $n \ge N_0$ and for all $n \le 1-N_0$. After using the inequality \eqref{C_a_C_b_R}, we conclude that
\begin{equation}\label{C_N_0_R}
C_L \le C_{a_{n+1}} \le C_l \tab \text{for all $n \ge N_0$ and for all $n \le 1-N_0$.} 
\end{equation}
We will first prove the validity of inequality \eqref{mu_bounds_R} when at least one of the sequences $\{\theta_n\}_{N_0}^\infty$ and $\{\theta_n\}_{1-N_0}^{-\infty}$ is eventually monotone. Then we will show that this inequality holds in general, after proving its validity when neither sequence is eventually monotone. Finally, we will prove that the constants $C_l$ and $C_L$ are the best possible by giving specific examples for which these constants are obtained.

First, suppose $\{\theta_n\}_{N_0}^{\infty}$ is eventually monotone in the broader sense. Then there exist $N_1 \ge N_0$ for which $\{\theta_n\}_{N_1}^\infty$ is monotone. By proposition \ref{uniform_bound_R}, this sequences is bounded in $[0,C_0]$, so it must converge to some real number $C \in [0, C_0]$. Thus
\[\displaystyle{\lim_{n \to \infty}D_n := \lim_{n \to \infty}\sqrt{1+4{k}\theta_{n-1}\theta_n} = \sqrt{1+4{k}C^2}}\]
Using formula \eqref{C_a_alt_R} and corollary \ref{a_n+1+k_R}, we obtain 
\[\dfrac{1}{C_{a_{n+1}}^2} + 4k = (a_{n+1}+k+1)^2 = \dfrac{1}{4\theta_n^2}\big(D_n^2+D_{n+1}^2 +2D_n{D_{n+1}} \big),\]
so that 
\[\displaystyle{\lim_{n \to \infty}\frac{1}{C_{a_{n+1}}^2} + 4k =\frac{1}{C^2}(1 + 4{k}C^2) = \frac{1}{C^2} + 4k}\] 
and $\displaystyle{\lim_{n \to \infty}C_{a_{n+1}} = C}$. Since $\big\{C_{a_n}\big\}_{n \in \mathbb{Z}}$ is a discrete set, there must exists a non-negative integer $a$ and a positive integer $N_2 \ge N_1$ such that $\theta_n = C_{a_{n+1}} = C_a = C$ for all $n \ge N_2$. But this implies from theorem \ref{theta_n_constant_R} that $\theta_n = C_a$ for all $n \in \mathbb{Z}$. Since $N_2 \ge N_1 \ge N_0$, we use the inequality \eqref{C_N_0_R} to conclude that $C_L \le \theta_n = C_a \le C_l$ for all $n \in \mathbb{Z}$, which proves the correctness of the inequality \eqref{mu_bounds_R} for this scenario. Showing that the case when $\{\theta_n\}_{1-N_0}^{-\infty}$ is eventually monotone reduces to the constant case is the same {\it mutatis mutandis}.\\

Now suppose that both the sequences $\{\theta_n\}_{N_0}^\infty$ and $\{\theta_n\}_{1-N_0}^{-\infty}$ are not eventually monotone, in the broader sense. In particular $\{\theta_n\}_{-\infty}^\infty$ is not constant, so that an application of Theorem \ref{theta_n_constant_R} yields $\theta_{n-1} \ne \theta_n$ for all $n \in \mathbb{Z}$. Let $N_1 \ge N_0$ be the first time the sequence $\{\theta_n\}_{N_0}^\infty$ changes direction, that is, we either have $\theta_{N_1} = \min\big\{\theta_{N_1-1},\theta_{N_1},\theta_{N_1+1}\big\}$ or $\theta_{N_1} = \max\big\{\theta_{N_1-1},\theta_{N_1},\theta_{N_1+1}\big\}$. We now show that $C_L < \theta_n < C_l$ for all $n \ge N_1$. Fixing $N \ge N_1$, take $N',N''$ such that $\theta_{N'}$ and $\theta_{N''}$ are the closest local extrema to $\theta_N$ in the sequence $\{\theta_n\}_{N_1}^\infty$ from the left and right. That is, $N_1 \le N' < N < N''$ and we either have 
\[\theta_{N'} < \theta_{N'+1} < ... < \theta_N < \theta_{N+1} < ... < \theta_{N''}\]
and $\theta_{N'} < \theta_{N'-1}, \tab \theta_{N''} > \theta_{N''+1}$ or  
\[\theta_N' > \theta_N'+1 > ... > \theta_N > \theta_{N+1} > ... > \theta_{N''}.\]
and $\theta_{N'} > \theta_{N'-1}, \tab \theta_{N''} < \theta_{N''+1}$. In the first case, applying the previous lemma to $\theta_{N'} = \min\{\theta_{N'-1},\theta_{N'},\theta_{N'+1}\}$ implies $\theta_{N'} > C_{a_{N'+1}}$ and  applying the previous lemma to $\theta_{N''} = \max\{\theta_{N''-1},\theta_{N''},\theta_{N''+1}\}$ implies $\theta_{N''} < C_{a_{N''+1}}$. In the second case, applying the previous lemma to $\theta_{N'} = \max\{\theta_{N'-1},\theta_{N'},\theta_{N'+1}\}$ implies $\theta_{N'} < C_{a_{N'+1}}$ and  applying the previous lemma to $\theta_{N''} = \min\{\theta_{N''-1},\theta_{N''},\theta_{N''+1}\}$ implies $\theta_{N''} > C_{a_{N''+1}}$. But $N'' > N' \ge N_1 \ge N_0$ so that $l \le a_{N'+1}, \tab a_{N''+1} \le L$. We conclude that
\[C_L \le C_{a_{N'+1}} < \theta_{N'} < \theta_N < \theta_{N''} <  C_{a_{N''+1}} \le C_l\]
in the first case and
\[C_L \le C_{a_{N''+1}} < \theta_{N''} < \theta_N < \theta_{N'} <  C_{a_{N'+1}} \le C_l\]
in the second case. In either case $C_L < \theta_N < C_l$ as desired. Similarly, we let $N_2 \ge N_0$ be the first time the sequence $\{\theta_n\}_{1-N_0}^{-\infty}$ changes direction. The proof that $C_L < \theta_n < C_l$ for all $n \le 1 - N_2$ is {\it mutatis mutandis} the same. After setting $N_3 := \max\{N_1,N_2\}$, we conclude that $C_L < \theta_n < C_l$ for all $\abs{n} > N_3$, which proves the correctness of the inequality \eqref{mu_bounds_R} for this scenario as well.\\   

Finally, we prove that $C_L$ and $C_l$ are the best possible bounds. Clearly, $C_L = 0$ is the best bound when $L = \infty$ and similarly $C_l = 0$ is the best bound when $l = L = \infty$. To prove $C_l$ is the best possible upper bound when $l < \infty$, fix $l \le L < \infty$, let $x_0$ be such that
\[a_n(x_0) := 
\begin{cases}
L & \text{if $\log_2{n}$ is a positive integer}\\
l & \text {otherwise}
\end{cases} \hspace{1pc}, \hspace{1pc} n \ge 1\]
and let $y_0$ be its \bf{reflection}, that is, $y_n := 1 - k - a_1 - [a_2,a_3,...]_{(m,k)}$. Then $l \le a_n(x_0,y_0) \le L$ for all $n \in \mathbb{Z}$ and both digits appear infinitely often. If $N \ge 1$ is such that $a_{N}(x_0,y_0) = L$ then 
\[x_{(N + \log_2{N}+1)} = [\displaystyle{\overbrace{l,l,l,l,l,...,l}^{\text{$\log_2{N}$ times}}},r_{(\log_2{N}+1)}]\] and 
\[y_{(N + \log_2{N}+1)} = 1 - k - l - [\displaystyle{\overbrace{l,l,l,l,l,...,l}^{\text{$\log_2{N}-1$ times}}},s_{(\log_2{N}-1)}].\]
Since this occurs for infinitely many $N$, there exists a subsequence $\{n_j\} \subset \mathbb{Z}$ such that $(x_{n_j},y_{n_j}) \to (\xi_l, 1 - k - l - \xi_l)$. Then theorem \ref{theta_dynamic} and formula \eqref{xi_a_C_a_R} prove $\theta_{({n_j}+1)}(x_0,y_0) = \frac{1}{x_{n_j}-y_{n_j}} \to \frac{1}{2\xi_l+k+l-1} = C_l$ as $j \to \infty$. Therefore, $C_l$ cannot be replaced with a smaller constant. Similarly, to prove $C_L$ is the best possible lower bound when $L < \infty$, fix $l \le L < \infty$, let $x_0$ be such that
\[a_n(x_0) := 
\begin{cases}
l & \text{if $\log_2{n}$ is a positive integer}\\
L & \text {otherwise}
\end{cases} \hspace{1pc}, \hspace{1pc} n \ge 1\]
and let $y_0$ be its reflection. Then $l \le a_n(x_0,y_0) \le L$ for all $n \in \mathbb{Z}$ and both digits appear infinitely often. If $N \ge 1$ is such that $a_{N}(x_0,y_0) = l$ then 
\[x_{(N + \log_2{N}+1)} = [\displaystyle{\overbrace{L,L,L,...,L}^{\text{$\log_2{N}$ times}}},r_{(\log_2{N}+1)}]\] and 
\[y_{(N + \log_2{N}+1)} = 1 - k - L - [\displaystyle{\overbrace{L,L,L,...,L}^{\text{$\log_2{N}-1$ times}}},s_{(\log_2{N}-1)}].\]
Since this occurs for infinitely many $N$, there exists a subsequence $\{n_j\} \subset \mathbb{Z}$ such that $(x_{n_j},y_{n_j}) \to (\xi_L, 1 - k - L - \xi_L)$. Then theorem \ref{theta_dynamic} and formula \eqref{xi_a_C_a_R} prove $\theta_{({n_j}+1)}(x_0,y_0) = \frac{1}{x_{n_j}-y_{n_j}} \to \frac{1}{2\xi_L+k+L-1} = C_L$ as $j \to \infty$. Therefore, $C_L$ cannot be replaced with a larger constant.  
\end{proof}
\newpage
\section{The one-sided sequence of approximation coefficients}
\noindent We end this chapter with quoting the results which apply to the one-sided sequence of approximation coefficients as well. Fix $x_0 \in (0,1) - \IQ_{(1,k)}$ and write $a_n := a_n(x_0)$ and $\theta_n = \theta_n(x_0) = \frac{1}{x_n - Y_n}$ for all $n \ge 1$, where $x_n$ and $Y_n$ are the future and past of $x_0$ at time $n$ as in formulas \eqref{x_n} and \eqref{Y_n} and $\theta_n(x_0)$ is as in the definition \eqref{theta_x_0}. Using formulas \eqref{ccT_explicit_R} and \eqref{ccT_inv_explicit_R}, we see that the maps $\ccT$ and $\ccT^{-1}$ are well defined on $(x_n,Y_n)$ and that $(x_n,Y_n) = \ccT^n(x_0,Y_0)$ for $n \ge 1$. Using Haas' result \eqref{theta_future_past} and the definition of the map $\Psi$ \eqref{Psi_R}, we see that $\Psi(x_n,Y_n) = (\theta_{n-1},\theta_n)$ for all $n \ge 1$. Then, the proofs of theorems \ref{a_n_R}, \ref{theta_pm_1_R} and \ref{theta_extension_R} remain the same, after we restrict $n \ge 1$ and replace $y_n$ with $Y_n$. Consequently, these results apply to the one-sided sequences $\{a_n\}_1^\infty$ and $\{\theta_n\}_1^\infty$.\\


\end{thethesiscore}

\renewcommand{\thepage}{\arabic{page}}
\pagestyle{plain}


\begin{thebibliography}{[88]}




\bibitem{BJW}
W. Bosma, H. Jager and F. Wiedijk, \emph{Some metrical observations on the approximation by continued fractions}, Nederl. Akad. Wetensch. Indag. Math. 45, 1983.

\bibitem{BM}
F. Bagemihl and J.R. McLaughlin, \emph{Generalization of some classical theorems concerning triples of consecutive convergents to simple continued fractions}, J. Reine Angew, Math. 221, 1966.

\bibitem{Burger}
E. B. Burger, \emph{Exploring the number jungle: A journey into diophantine analysis}, Providence, RI: Amer. Math. Soc., 2000. 

\bibitem{CF} 
T. W. Cusick and M. E. Flahive, \emph{The Markoff and Lagrange spectra}, Mathematical Surveys and Monograms no. 30, The American Mathematical Society, 1989.

\bibitem{DK} 
 K. Dajani and C. Kraaikamp, \emph{Ergodic theory of numbers}, The Carus Mathematical Monograms no. 29, The Mathematical Association of America, 2002.


\bibitem{Hardy}
G.H. Hardy and E.M. Wright, \emph{Introduction to the theory of numbers}, 4th ed. Oxford University Press, London, 1960

\bibitem{HM}
A. Haas and D. Molnar, \emph{Metrical diophantine approximation for continued fractions like maps of the interval}, Transaction of the American Mathematical Society, Volume 356(7), 2003.

\bibitem{HM2}
A. Haas and D. Molnar, \emph{The distribution of Jager pairs for continued fraction like mappings of the interval}, Pacific J. Math. 217(1), 2004. 

\bibitem{HBackwards} K. Grochening and A. Hass, \emph{Backwards continued fractions, Hecke groups and invariant measures for transformations for the interval}, Ergodic Theory \& Dynamic Systems (16), 1996.

\bibitem{HNatural} A. Haas, \emph{Invariant measures and natural extensions}, Canad. Math Bull, Vol 45(1), 2002.


\bibitem{J} H. Jager, \emph{Continued fractions and ergodic theory, Transcendental numbers and related topics}, RIMS Kukyuroko 599(1), 1986.

\bibitem{JP} W. B. Jurkat and A. Peyerimhoff, \emph{Characteristic approximation properties of quadratic irrationals}, Internat. J. Math. \& Math. Sci.(1), 1978.

\bibitem{K} A. Ya. Khintchine, \emph{Continued fractions}, P. Noordhoff Ltd. Groningen, 1963.

\bibitem{Molnar} D. Molnar,  \emph{Metrical Theory of Strongly Markoff Linear Fractional Mod One Interval Maps}, Ph.D. Thesis, University of Connecticut press, 2010.


\bibitem{Perron} O. Perron, \emph{$\ddot{U}$ber die approximation irrationaler zahlen durch rationale}, Heindelberg Akad. Wiss. Abh.(4), 1921.

\bibitem{N} H. Nakada, \emph{Metrical theory for a class of continued fraction transformations and their natural extensions}, Tokyo J. Math 4(2), 1981.

\bibitem{Renyi} A. R$\acute{e}$nyi, \emph{Representations for real numbers and their ergodic properties}, Acta Math. Acad. Sci. Hungar.(8), 1957.

\bibitem{Rohlin} V. A. Rohlin, \emph{Exact Endomorphisms of a Lebesgue space}, Izv. Akad. Nauk SSSR Ser. Mat. 25, 1961, 499-530.





\bibitem{Tong}
J. Tong, \emph{The Conjugate property of the Borel theorem on Diophantine Approximation}, Math. Z. 184(2) , 1983.

\end{thebibliography}
\end{document}